\documentclass[nosumlimits,twoside]{amsart}
\usepackage{amsfonts, amsmath, amssymb}
\usepackage[bookmarksnumbered,plainpages,driverfallback=dvipdfm,backref=page]{hyperref}
\usepackage[english]{babel}
\usepackage{orcidlink}
\usepackage{graphicx}
\usepackage{float}
\usepackage{srcltx}
\usepackage[all]{xy}
\usepackage{version}
\usepackage[T1]{fontenc}
\usepackage{xcolor}
\usepackage{enumerate}
\usepackage[normalem]{ulem}
\usepackage{bm}
\usepackage{latexsym}
\usepackage[2emode]{psfrag}
\usepackage{yhmath}
\usepackage{array}
\usepackage{dsfont}
\usepackage{mathrsfs}
\usepackage{amssymb}
\usepackage{tikz-cd}
\usepackage{comment}
\usetikzlibrary{calc,babel}

\newtheorem{theorem}{\sc Theorem}[section]
\newtheorem{proposition}[theorem]{\sc Proposition}

\newtheorem{lemma}[theorem]{\sc Lemma}
\newtheorem{corollary}[theorem]{\sc Corollary}
\theoremstyle{definition}
\newtheorem{definition}[theorem]{\sc Definition}

\theoremstyle{remark}
\newtheorem{remark}[theorem]{\sc Remark}

\newenvironment{invisible}{{\noindent\sc \colorbox{yellow}{Invisible:}\;}\color{gray}}{\medskip}
\excludeversion{invisible}

\tikzset{
  curve/.style={
    settings={#1},
    to path={
      (\tikztostart)
      .. controls ($(\tikztostart)!\pv{pos}!(\tikztotarget)!\pv{height}!270:(\tikztotarget)$)
      and ($(\tikztostart)!1-\pv{pos}!(\tikztotarget)!\pv{height}!270:(\tikztotarget)$)
      .. (\tikztotarget)\tikztonodes
    },
  },
  settings/.code={%
    \tikzset{quiver/.cd,#1}%
    \def\pv##1{\pgfkeysvalueof{/tikz/quiver/##1}}%
  },
  quiver/.cd,
  pos/.initial=0.35,
  height/.initial=0,
}

\allowdisplaybreaks
\excludeversion{invisible?}
\excludeversion{proof?}
\newcommand{\id}{\mathrm{Id}}

\newcommand{\Cc}{\mathcal{C}}

\newcommand{\Mm}{\mathcal{M}}
\newcommand{\mm}{\mathfrak{M}}

\newcommand{\Rr}{\mathcal{R}}

\newcommand{\ot}{\otimes}

\newcommand{\Hopf}{\mathsf{Hopf}_{\mathrm{coc}}(\Mm)}

\newcommand{\rd}[1]{{\color{blue}{#1}}}

{
\left\{\begin{aligned}}
{\end{aligned}
\right.
}

\title{On the semi-abelianness of cocommutative Hopf monoids}

\author{Andrea Sciandra~\orcidlink{0009-0008-0447-287X}}
\address{%
\parbox[b]{\linewidth} {Département de Mathématiques, Université Libre de Bruxelles, Boulevard du Triomphe, B-1050 Bruxelles, Belgium}}
\email{andrea.sciandra@ulb.be}
\urladdr{\url{www.andreasciandra.com}}
\author{Zhenbang Zuo~\orcidlink{0009-0008-9013-0365}}
\address{%
\parbox[b]{\linewidth}{University of Turin, Department of Mathematics ``G. Peano'', via
Carlo Alberto 10, I-10123 Torino, Italy}}
\email{zhenbang.zuo@edu.unito.it}
\date{}

\begin{document}

\maketitle

\begin{abstract}
    By providing a suitable generalization of Newman's bijective correspondence known for cocommutative Hopf algebras, we prove that the category of cocommutative Hopf monoids in any abelian symmetric monoidal category is semi-abelian, once faithful (co)flatness conditions are satisfied. This result unifies and generalizes the semi-abelianness of cocommutative Hopf algebras and of cocommutative color Hopf algebras known up to now. As a consequence of the semi-abelianness, the category of cocommutative Hopf monoids is also action representable. Finally, we prove that abelian objects in the category of cocommutative Hopf monoids coincide exactly with commutative and cocommutative Hopf monoids, which form so an abelian category. 
\end{abstract}

\tableofcontents

\section{Introduction}

Cocommutative Hopf algebras are similar to groups in various aspects, as already noted in \cite{Yana1,Yana2} where some classical isomorphism theorems for groups were proven in the Hopf algebra setting. In fact, cocommmutative Hopf algebras coincide with the category of internal groups in the cartesian monoidal category of cocommutative coalgebras and this implies that a Split Short Five Lemma holds for them. Moreover, they are proven to form a semi-abelian category by M. Gran, F. Sterck, and J. Vercruysse in \cite{GSV}, extending a previous result obtained in \cite{GKV} where the base field was assumed of zero characteristic. The semi-abelianness of the category of cocommutative Hopf algebras can be seen as a non-commutative generalization of Takeuchi’s result asserting that commutative and cocommutative Hopf algebras over a field form an abelian category \cite{Takeuchi}, that extends its finite
dimensional version due to Grothendieck.

Semi-abelian categories were introduced in \cite{JMT} to capture typical algebraic properties valid for groups, rings and algebras. They provide a good categorical framework for a meaningful treatment of radical and commutator theory, and
of (co)homology theory of non-abelian structures. They also present natural notions of semi-direct product \cite{BournJanelidze}, internal action \cite{BJK2} and crossed module \cite{Janelidze}. The classical examples of semi-abelian categories include groups, Lie algebras, rings (not necessarily unital) and commutative $\mathbb{C}^*$-algebras. 

The semi-abelianness of cocommutative Hopf algebras was then extended to cocommutative color Hopf algebras in \cite{AS}, i.e.\ cocommutative Hopf monoids in the category of vector spaces graded over an abelian group $G$, when the latter is finitely generated and the characteristic of the base field is different from 2 (not needed if the cardinality of $G$ is finite and odd). This opened the question of how far it would be possible to extend the result, by considering cocommutative Hopf monoids in an arbitrary braided monoidal category.

In this paper, we prove that the category $\Hopf$ of cocommutative Hopf monoids in any abelian symmetric monoidal category $(\Mm,\ot,\mathbf{1},\sigma)$ is exact in the sense of \cite{Barr}, i.e.\ it is a regular category and any equivalence relation inside it is a kernel pair, once $(\Mm,\ot,\mathbf{1},\sigma)$ satisfies some faithful (co)flatness conditions. Since $\Hopf=\mathsf{Grp}(\mathsf{Comon}_{\mathrm{coc}}(\Mm))$ is also protomodular in the sense of \cite{Bournproto}, it is semi-abelian once it has binary coproducts. The main result of this paper concerns the regularity of the category $\Hopf$, under the aforementioned conditions on the category $(\Mm,\ot,\mathbf{1},\sigma)$. In the case of cocommutative Hopf algebras, this result was proven in \cite{GSV} using Newman's Theorem \cite{Newman}: for any cocommutative Hopf algebra $A$, there is a bijective correspondence between the set of its Hopf subalgebras and the set of quotient left $A$-module coalgebras. More explicilty, given a Hopf subalgebra $i:K\to A$ and a quotient left $A$-module coalgebra $\pi:A\to Q$, the mutual inverse bijections $\phi_{A}$ and $\psi_{A}$ are defined by
\begin{equation}\label{bijectionNewman}
\phi_{A}(i):A\to A/AK^{+},\quad \psi_{A}(\pi):A^{\mathrm{co}Q}:=\{x\in A\ |\ (\pi\ot\mathrm{Id}_{A})\Delta_{A}(x)=\pi(1_{A})\ot x\}\to A.
\end{equation}
To obtain the regularity of $\Hopf$, we extend the aforementioned result to this more general setting. Once the category $\Hopf$ is semi-abelian, we also prove that it is action representable and that the abelian category of abelian objects in $\Hopf$ coincides with the category of commutative and cocommutative Hopf monoids. \medskip

\noindent The organization of the paper is as follows. First, in Section \ref{sec:prelimiaries}, we recall some notions and facts that are useful througout the paper. In Section \ref{sec:pointed,limits,protomodularity}, we show that $\Hopf$ is pointed (Lemma \ref{lem:Hopfpointed}) and finitely complete (Proposition \ref{prop:Hopfcocfintcomplete}). Moreover, it is also protomodular (Proposition \ref{prop:protomodularity}) as it coincides with the category $\mathsf{Grp}(\mathsf{Comon}_{\mathrm{coc}}(\Mm))$ of internal groups in the finitely complete category $\mathsf{Comon}_{\mathrm{coc}}(\Mm)$. In Section \ref{sec:colimits}, we construct coequalizers in the category $\Hopf$ (Proposition \ref{prop:coequalizerHopf}), which are used in Section \ref{sec:NewmanTheorem} to prove a bijective correspondence, for a given $A$ in $\Hopf$, between a class of subobjects of $A$ in $\Hopf$ and a class of quotients of $A$ in $\mathsf{Comon}_{\mathrm{coc}}(_{A}\Mm)$ (Theorem \ref{thm:NewmanforM}). This result generalizes Newman's Theorem for cocommutative Hopf algebras obtained in \cite{Newman} and its extension for cocommutative color Hopf algebras proven in \cite{AS}, to cocommutative Hopf monoids in arbitrary (abelian) braided monoidal categories. The bijective correspondence restricts to kernels in $\Hopf$ and quotients in $\mathsf{Comon}_{\mathrm{coc}}(_{A}\Mm)$ which are regular epimorphisms in $\Hopf$ (Corollary \ref{cor:bijectivecorrespondence}). In fact, as proven in Theorem \ref{cor:normalobjects}, kernels in $\Hopf$ are equivalently described as normal monomorphisms in the sense of Definition \ref{def:normal}. In Section \ref{sec:regularity}, we use the generalized Newman Theorem to prove that $\Hopf$ is regular, once the abelian symmetric monoidal category $(\Mm,\ot,\mathbf{1},\sigma)$ satisfies the ``faithful coflatness condition'' (Definition \ref{deffaithcoflat}) and the ``faithful flatness condition'' (Definition \ref{def:faithflatcat}). More precisely, we obtain the regular epimorphism-monomorphism factorization for any morphism in $\Hopf$ (Proposition \ref{prop:factmorp2}) which allows us to prove that regular epimorphisms and monomorphisms in $\Hopf$ coincide with morphisms in $\Hopf$ which are epimorphisms and monomorphisms in $\Mm$, respectively (Corollary \ref{cor:regepimono}). Using this, we obtain that regular epimorphisms are stable under pullbacks in $\Hopf$ (Proposition \ref{cor:stability}). In Section \ref{sec:semiabelian}, we conclude that $\Hopf$ is exact, hence semi-abelian once it has binary coproducts (Theorem \ref{thm:semiabelian}). Finally, we prove that the category of abelian objects in $\Hopf$ coincides with the category $\mathsf{Hopf}_{\mathrm{coc,com}}(\Mm)$ of commutative and cocommutative Hopf monoids in $\Mm$ (Proposition \ref{prop:abelianobjects}) and we show that the category $\Hopf$ is action representable, once $(\Mm,\ot,\mathbf{1},\sigma)$ is closed monoidal (Proposition \ref{prop:actionreprestable}).

\bigskip

\noindent\textit{Notations and conventions}. The identity morphism for an object $X$ in a category $\Mm$ will be denoted by $\mathrm{Id}_{X}$ or $1_{X}$. Given a morphism $f:A\to B$ in $\Mm$, the kernel of $f$ in $\Mm$ will be denoted by $\mathsf{ker}(f):\mathsf{Ker}(f)\to A$ and the cokernel of $f$ by $\mathsf{coker}(f):B\to \mathsf{Coker}(f)$. A monoidal category will be denoted by $(\Mm,\ot,\mathbf{1})$. Comforted by the MacLane Coherence theorem, we shall consistently be sloppy on associativity and unit constraints. Given a monoidal category $(\Mm,\ot,\mathbf{1})$, we denote the categories of monoids and comonoids in $\Mm$ by $\mathsf{Mon}(\Mm)$ and $\mathsf{Comon}(\Mm)$, respectively. Given a braided monoidal category $(\Mm,\ot,\mathbf{1},\sigma)$, we denote by $\mathsf{Bimon}(\Mm)$ and $\mathsf{Hopf}(\Mm)$ the categories of bimonoids and Hopf monoids in $\Mm$, respectively, and by 
$\mathsf{Comon}_{\mathrm{coc}}(\Mm)$ the category of 
cocommutative comonoids in $\Mm$, i.e.\ comonoids $(C,\Delta,\varepsilon)$ such that $\sigma_{C,C}\Delta=\Delta$. The antipode of an object $A$ in $\mathsf{Hopf}(\Mm)$ will be denoted by $S_A$. If $(A,m,u)$ is an object in $\mathsf{Mon}(\Mm)$, we denote by ${}_A\Mm$ (resp.\ $\Mm_A$) the category of left (resp.\ right) $A$-modules and left (resp.\ right)
$A$-linear morphisms in $\Mm$. If $(C,\Delta,\varepsilon)$ is an object in $\mathsf{Comon}(\Mm)$, we denote by ${}^{C}\!\Mm$ (resp.\ $\Mm^{C}$) the category of left (resp.\ right) $C$-comodules and left (resp.\ right) $C$-colinear morphisms in $\Mm$. Given a cartesian monoidal category $(\Mm,\times,\mathbf{1})$, the category of internal groups in $\Mm$ will be denoted by $\mathsf{Grp}(\Mm)$.

\section{Preliminaries}\label{sec:prelimiaries}

Here we recall some preliminary notions and results that will be useful throughout the paper. For basic notions about category theory we refer the reader to \cite{Borceux} and \cite{MacLane-book}.

\begin{definition}[{\cite{JMT}}]\label{def:semi-abeliancategory}
    A category $\Cc$ is \textit{semi-abelian} if: 
\begin{itemize}
    \item[1)] $\Cc$ is \textit{pointed}, i.e.\ it has a zero object (an object which is both initial and terminal); \medskip
    \item[2)] $\Cc$ is (Barr)-\textit{exact}, i.e.\ the following two facts hold:\medskip
    \begin{itemize}
        \item[i.] $\Cc$ is \textit{regular}: it has finite limits, any morphism factorizes as a regular epimorphism (i.e.\ a coequalizer of a pair of morphisms) followed by a monomorphism, and regular epimorphisms are stable under pullbacks, \medskip
        \item[ii.] any equivalence relation in $\Cc$ is a kernel pair, i.e.\ a pullback of a morphism along itself;
    \end{itemize}\medskip
    \item[3)] $\Cc$ is (Bourn)-\textit{protomodular}: since $\Cc$ is pointed and finitely complete, this is equivalent to the validity of the Split Short Five Lemma in $\Cc$ (see e.g.\ \cite[Proposition 3.1.2]{BB}); \medskip
    \item[4)] $\Cc$ has binary coproducts.
\end{itemize}
\end{definition}

A semi-abelian category $\Cc$ has automatically coequalizers for any pair of morphisms, see e.g.\ \cite[Proposition 5.1.3]{BB}, so it is finitely cocomplete. When $\Cc$ is pointed, regular, and protomodular, it is called \textit{homological}, see e.g. \cite[Definition 4.1.1]{BB}. \medskip

Let $(\Mm,\ot,\mathbf{1})$ be a monoidal category with coequalizers, $(A,m,u)$ be an object in $\mathsf{Mon}(\Mm)$, and $X\in {}_A\Mm$, $M\in \Mm_{A}$ with structure morphisms $\mu_X:A\otimes X\to X$ and $\nu_M:M\otimes A\to M$, respectively. 
Then, $(M\otimes _AX,q_{M,X})$ is defined to be the coequalizer of the pair of morphisms
$(\nu_M\otimes\mathrm{Id}_X,\mathrm{Id}_M\otimes \mu_X)$ in $\Mm$:
\begin{equation}\label{def:balancedtensor}
\begin{tikzcd}
	M\ot A\ot X && M\ot X & M\ot_{A}X
	\arrow[shift left, from=1-1, to=1-3,"\nu_{M}\otimes\mathrm{Id}_X"]
	\arrow[shift right, from=1-1, to=1-3,"\mathrm{Id}_M\otimes \mu_X"']
	\arrow[from=1-3, to=1-4,"q_{M,X}"]
\end{tikzcd}
\end{equation}
This construction provides functors $M\ot_{A}(-):{}_{A}\Mm\to\Mm$ and $(-)\ot_{A}X:\Mm_{A}\to\Mm$ for any $M$ in $\Mm_{A}$ and any $X$ in $_{A}\Mm$. For any morphism $f:X\to Y$ in $_{A}\Mm$ and object $M$ in $\Mm_{A}$, $\mathrm{Id}_{M}\otimes_A f:M\otimes _AX\to M\otimes_AY$ is the unique morphism in $\Mm$ such that $(\mathrm{Id}_{M}\otimes_A f)q_{M,X}=q_{M,Y} (\id_M\otimes f)$.
\begin{invisible}
    Indeed, $q^A_{M,Y}(\id_M\otimes f)(\id_M\otimes \mu^A_X)=q^A_{M,Y}(\id_M\otimes \mu^A_Y)(\id_M\otimes\id_A\otimes f)=q^A_{M,Y}(\nu^A_M\otimes\id_Y)(\id_M\otimes\id_A \otimes f)=q^A_{M,Y}(\nu^A_M\otimes f)=q^A_{M,Y}(\id_M\otimes f)(\nu^A_M\otimes\id_X)$.
\end{invisible}
Similarly, given $g:M\to N$ in $\Mm_A$ and $Y$ in ${}_A\Mm$, $g\otimes_A \mathrm{Id}_{Y}:M\otimes _AY\to N\otimes _A Y$ is the unique morphism in $\Mm$ such that $(g\otimes_A \mathrm{Id}_{Y})q_{M,Y}=q_{N,Y}(g\otimes\id_Y)$.
For $M\in\Mm_A$ and $Y\in {}_A\Mm$, there are canonical (natural) isomorphisms $\Upsilon_M$, $\Upsilon'_Y$ in $\Mm$:
\begin{itemize}
    \item $\Upsilon_M:M\otimes_AA\to M$, uniquely determined by $\Upsilon_Mq_{M,A}=\nu_M$; 
    \item $\Upsilon'_Y:A\otimes_AY\to Y$, uniquely determined by  $\Upsilon'_Yq_{A,Y}=\mu_Y$. 
\end{itemize}
One can check that $\Upsilon^{-1}_M=q_{M,A}(\id_M\otimes u_A)$ and $\Upsilon'^{-1}_Y=q_{A,Y}(u_A\otimes \id_Y)$. 

Let $(\Mm,\otimes,\mathbf{1})$ be a monoidal category with equalizers and $(C,\Delta,\varepsilon)$ be an object in $\mathsf{Comon}(\Mm)$. Recall from e.g.\ \cite[Definition 2.2.1]{A97} that, given a right $C$-comodule $(V, \rho_V)$  and a left $C$-comodule $(W,\lambda_W)$ in $\Mm$, their \textit{cotensor product} over $C$ in $\Mm$ is defined to be the equalizer $(V\square_C W, e_{V,W})$ of the pair of morphisms $(\rho_V\otimes \mathrm{Id}_{W}, \mathrm{Id}_{V}\otimes \lambda_W)$ in $\Mm$: 
\begin{equation}\label{def:cotensor}
\begin{tikzcd}
	V\square_CW & V\otimes W && V\ot C\ot W
	\arrow[from=1-1, to=1-2,"e_{V,W}"]
	\arrow[shift left, from=1-2, to=1-4,"\rho_{V}\otimes\mathrm{Id}_W"]
	\arrow[shift right, from=1-2, to=1-4,"\mathrm{Id}_V\otimes \lambda_W"']
\end{tikzcd}
\end{equation}
This construction provides functors $V\square_{C}(-):{}^{C}\Mm\to\Mm$ and $(-)\square_{C}W:\Mm^{C}\to\Mm$ for any $V$ in $\Mm^{C}$ and any $W$ in $^{C}\Mm$.
For any morphism $f:X\to Y$ in $^{C}\Mm$ and object $V$ in $\Mm^{C}$, $V\square_Cf:V\square_C X\to V\square_C Y$ is the unique morphism in $\Mm$ such that $e_{V,Y}(V\square_C f)=(\id_V\otimes f) e_{V,X}$. Similarly, for any morphism $g:M\to N$ in $\Mm^C$ and object $Y$ in ${}^C\Mm$, $g\square_CY: M\square_CY\to N\square_CY$ is the unique morphism in $\Mm$ such that $e_{N,Y}(g\square_CY)=(g\otimes\id_Y)e_{M,Y}$.
For $M\in\Mm^{C}$ and $Y\in{}^{C}\Mm$, we have the canonical (natural) isomorphisms $\Lambda_M$, $\Lambda'_{Y}$ in $\Mm$: 
\begin{itemize}
    \item $\Lambda _M:M\to M\square_CC$, uniquely determined by $\rho_M= e_{M,C} \Lambda_M$;
    \item $\Lambda'_Y: Y\to C\square_CY$, uniquely determined by the property $\lambda_Y=e_{C,Y}\Lambda'_Y$.
\end{itemize}
One can easily check that $\Lambda_M^{-1}=(\id_M\otimes\varepsilon_{C})e_{M,C}$ and $(\Lambda'_Y)^{-1}=(\varepsilon_{C}\otimes\id_Y)e_{C,Y}$. 

In this paper, we will usually deal with \textit{abelian monoidal categories} $(\Mm,\ot,\mathbf{1})$, i.e.\ monoidal categories which are also abelian and such that the functors $M\ot(-):\Mm\to\Mm$ and $(-)\ot M:\Mm\to\Mm$ are additive and exact, for any $M$ in $\Mm$. We will say that $\ot$ preserves equalizers and coequalizers meaning that the latter are preserved by $M\ot (-)$ and $(-)\ot M$, for any $M\in\Mm$. We recall that, since $\ot$ preserves coequalizers in $\Mm$, the category $(_{A}\Mm_{A},\ot_{A},A)$ is a monoidal category such that $\ot_{A}$ preserves coequalizers in $_{A}\Mm_{A}$, for any object $A$ in $\mathsf{Mon}(\Mm)$, see e.g.\ \cite[Theorem 1.12]{AMS}. Dually, since $\ot$ preserves equalizers, the category $({}^{C}\Mm^{C},\square_{C},C)$ is a monoidal category such that $\square_{C}$ preserves equalizers in ${}^{C}\Mm^{C}$, for any object $C$ in $\mathsf{Comon}(\Cc)$. 
\begin{invisible}
Given $V$ in ${}^{C}\Mm^{C}$ and $W$ in $\!^{C}\Mm^{C}$, we recall that $V\square_{C}W$ is an object in $\Mm^{C}$ with comodule structure determined as the unique morphism $\rho_{V\square_{C}W}:V\square_{C}W\to(V\square_{C}W)\ot C$ in $\Mm$ such that $(e_{V,W}\ot\id_{C})\rho_{V\square_{C}W}=(\id_{V}\ot\rho_{W})e_{V,W}$, so that $e_{V,W}$ is in $\Mm^{C}$. In fact, since $\ot$ preserves the equalizers, we have that the following diagram is an equalizer in $\Mm$
\[
\begin{tikzcd}
	(V\square_CW)\ot C && V\otimes W\ot C && V\ot C\ot W\ot C
	\arrow[from=1-1, to=1-3,"e_{V,W}\ot\id_{C}"]
	\arrow[shift left, from=1-3, to=1-5,"\rho_{V}\otimes\mathrm{Id}_W\ot\id_{C}"]
	\arrow[shift right, from=1-3, to=1-5,"\mathrm{Id}_V\otimes \lambda_W\ot\id_{C}"']
\end{tikzcd}
\]
Since $W$ is a bicomodule we have
\[
\begin{split}
(\rho_{V}\otimes\mathrm{Id}_W\ot\id_{C})(\id_{V}\ot\rho_{W})e_{V,W}&=(\id_{V\ot C}\ot\rho_{W})(\rho_{V}\otimes\mathrm{Id}_W)e_{V,W}=(\id_{V\ot C}\ot\rho_{W})(\id_{V}\ot\lambda_{W})e_{V,W}\\&=(\id_{V}\ot\lambda_{V}\ot\id_{C})(\id_{V}\ot\rho_{W})e_{V,W}
\end{split}
\]
there exists a unique morphism $\rho_{V\square_{C}W}:V\square_{C}W\to(V\square_{C}W)\ot C$ in $\Mm$ such that $(e_{V,W}\ot\id_{C})\rho_{V\square_{C}W}=(\id_{V}\ot\rho_{W})e_{V,W}$. Similarly, $\lambda_{V\square_{C}W}$ is uniquely determined such that $e_{V,W}$ becomes a morphism in $^{C}\Mm$.
\end{invisible}

Recall that for a pointed category $\Mm$, the kernel and the cokernel of a morphism in $\Mm$ are defined as the equalizer and the coequalizer of the morphism with the zero morphism, respectively. 
For a pointed monoidal category $(\Mm,\ot,\mathbf{1})$ with cokernels that are preserved by $M\ot(-)$ and $(-)\ot M$ for any $M\in \Mm$, an \textit{ideal} of an object $(A,m,u)$ in $\mathsf{Mon}(\Mm)$ is a pair $(I,i)$
where $I$ is an object in $_{A}\Mm_{A}$ and $i:I\to A$
is a morphism in $_{A}\Mm_{A}$ (where the $A$-bimodule structure of $A$ is given by $m_{A}$) which is a monomorphism in $\Mm$. 
Note that the object $\mathsf{Coker}(i)$ has a unique structure in $\mathsf{Mon}(\Mm)$ such that $\pi:=\mathsf{coker}(i)$ is a morphism in $\mathsf{Mon}(\Mm)$. Given an object $(C,\Delta,\varepsilon)$ in $\mathsf{Comon}(\Mm)$ and a monomorphism $i:I\to C$ in $\Mm$, 
we recall that $(I,i)$ is said to be a \textit{two-sided coideal} of $C$ if 
$\varepsilon_{C}i=0$ and $(\pi\ot\pi)\Delta_{C} i=0$, where $\pi:=\mathsf{coker}(i)$, so there exist unique morphisms $\varepsilon_{\mathsf{Coker}(i)}:\mathsf{Coker}(i)\to\mathbf{1}$ and $\Delta_{\mathsf{Coker}(i)}:\mathsf{Coker}(i)\to \mathsf{Coker}(i)\ot \mathsf{Coker}(i)$ in $\Mm$ such that $\pi$ is a morphism in $\mathsf{Comon}(\Mm)$. Given a braided monoidal category $(\Mm,\ot,\mathbf{1},\sigma)$, an object $B$ in $\mathsf{Bimon}(\Mm)$ and a monomorphism $i:I\to B$ in $\Mm$, we say that $(I,i)$ is a \textit{bi-ideal} of $B$ if $(I,i)$ is an ideal and a two-sided coideal of $B$. Given an object $H$ in $\mathsf{Hopf}(\Mm)$ and a monomorphism $i:I\to H$ in $\Mm$, we say that $(I,i)$ is a \textit{Hopf ideal} of $H$ if $(I,i)$ is a bi-ideal of $H$ and $\pi S_{H}i=0$, where $\pi:=\mathsf{coker}(i)$, so that there exists a unique morphism $S_{\mathsf{Coker}(i)}:\mathsf{Coker}(i)\to \mathsf{Coker}(i)$ in $\Mm$ such that $S_{\mathsf{Coker}(i)}\pi=\pi S_{H}$. Given a Hopf ideal $(I,i)$, the object $\mathsf{Coker}(i)$ has a unique structure in $\mathsf{Hopf}(\Mm)$ such that $\pi$ is a morphism in $\mathsf{Hopf}(\Mm)$.


We also recall that, for any morphism $f:A\to B$ in an abelian category $\Mm$, we have the following factorization diagram (the so-called image factorization)
\[\begin{tikzcd}
	\mathsf{Ker}(f) & A & B & \mathsf{Coker}(f) \\
	& \mathsf{Coker}(\mathsf{ker}(f)) & \mathsf{Ker}(\mathsf{coker}(f))
	\arrow[from=1-1, to=1-2,"\mathsf{ker}(f)"]
	\arrow[from=1-2, to=1-3, "f"]
	\arrow[from=1-2, to=2-2,"\mathsf{coker}(\mathsf{ker}(f))"']
	\arrow[from=1-3, to=1-4,"\mathsf{coker}(f)"]
	\arrow[from=2-3, to=1-3,"\mathsf{ker}(\mathsf{coker}(f))"']
	\arrow["\cong"{description}, from=2-2, to=2-3]
\end{tikzcd}\]
so that we can write $f=\mathsf{ker}(\mathsf{coker}(f))\mathsf{coker}(\mathsf{ker}(f))$. \medskip

Finally, we recall that, for a braided monoidal category $(\Mm,\ot,\mathbf{1},\sigma)$, the categories $\mathsf{Mon}(\Mm)$ and $\mathsf{Comon}(\Mm)$ are monoidal with $\ot$ and $\mathbf{1}$ (and the same constraints). 
We have the following equivalences of categories
\begin{equation}\label{equivalencescat}
\mathsf{Mon}(\mathsf{Comon}(\Mm))\cong\mathsf{Bimon}(\Mm)\cong\mathsf{Comon}(\mathsf{Mon}(\Mm)), \ \mathsf{Comon}(\mathsf{Comon}(\Mm))\cong\mathsf{Comon}_{\mathrm{coc}}(\Mm),
\end{equation}
see e.g.\ \cite[page 12]{AM}. 
Note that the monoidal categories $(\mathsf{Mon}(\Mm),\ot,\mathbf{1})$ and $(\mathsf{Comon}(\Mm),\ot,\mathbf{1})$ may fail to be braided, and the category $\mathsf{Bimon}(\Mm)$ may fail to be monoidal. However, if $\sigma$ is a \textit{symmetry}, i.e.\ $\sigma^{-1}_{A,B}=\sigma_{B,A}$ for all objects $A,B$ in $\Mm$, then $\sigma_{A,B}$ is a morphism of monoids and comonoids. It follows that both $(\mathsf{Mon}(\Mm),\ot,\mathbf{1},\sigma)$ and $(\mathsf{Comon}(\Mm),\ot,\mathbf{1},\sigma)$ are symmetric monoidal categories. Iterating these results and applying \eqref{equivalencescat}, one can deduce that $(\mathsf{Bimon}(\Mm),\ot,\mathbf{1},\sigma)$ and $(\mathsf{Comon}_{\mathrm{coc}}(\Mm),\ot,\mathbf{1},\sigma)$ are symmetric monoidal categories as well. Moreover, if $(\Mm,\ot,\mathbf{1},\sigma)$ is a symmetric monoidal category, then also $(\mathsf{Hopf}(\Mm),\ot,\mathbf{1},\sigma)$ and $(\Hopf,\ot,\mathbf{1},\sigma)$ are symmetric monoidal, see e.g.\ \cite[page 12]{AM}. \medskip

In the sequel, we will see that the hypotheses that the braided monoidal category $(\Mm,\ot,\mathbf{1},\sigma)$ is abelian and symmetric are crucial in proving that $\Hopf$ is semi-abelian. Given the abelian symmetric monoidal category $(\mathsf{Vec}_{\Bbbk},\ot_{\Bbbk},\Bbbk,\tau)$ of $\Bbbk$-vector spaces, where $\tau$ is the canonical flip map or, more generally, $(\mathsf{Vec}_{G},\ot_{\Bbbk},\Bbbk,\sigma)$ of $G$-graded vector spaces with $G$ a finitely generated abelian group and $\mathrm{char}(\Bbbk)\not=2$ (not needed if $G$ is finite of odd cardinality), we already know that the category $\mathsf{Hopf}_{\mathrm{coc}}(\mathsf{Vec}_{\Bbbk})=\mathsf{Hopf}_{\Bbbk,\mathrm{coc}}$ of cocommutative Hopf algebras over an arbitrary field $\Bbbk$ and, more generally, the category $\mathsf{Hopf}_{\mathrm{coc}}(\mathsf{Vec}_{G})$ of cocommutative color Hopf algebras are semi-abelian as proven in \cite{GSV} and \cite{AS}, respectively. 

\section{Pointedness, limits, and protomodularity}\label{sec:pointed,limits,protomodularity}

First, the following result shows that the category $\Hopf$ is pointed. 

\begin{lemma}\label{lem:Hopfpointed}
Let $(\Mm,\ot,\mathbf{1},\sigma)$ be a braided monoidal category. The category $\mathsf{Bimon}(\Mm)$ is pointed, with zero object $\mathbf{1}$. Consequently, $\mathsf{Hopf}(\Mm)$ and $\Hopf$ are also pointed.
\end{lemma}

\begin{proof}
     Clearly, the unit object $\mathbf{1}$ of $\Mm$ is an object in $\mathsf{Hopf}(\Mm)$. For any object $H$ in $\mathsf{Bimon}(\Mm)$ and any morphism $f:H \to \mathbf{1}$ in $\mathsf{Bimon}(\Mm)$, we have $\varepsilon_{H} = \varepsilon_{\mathbf{1}}f = f$. This means the counit $\varepsilon_{H}:H\to\mathbf{1}$ is the unique morphism in $\mathsf{Bimon}(\Mm)$ from $H$ to $\mathbf{1}$ (in fact, also the unique morphism in $\mathsf{Comon}(\Mm)$). Similarly, the unit $u_{H}:\mathbf{1}\to H$ is the unique morphism in $\mathsf{Bimon}(\Mm)$ from $\mathbf{1}$ to $H$ (in fact, also the unique morphism in $\mathsf{Mon}(\Mm)$). Therefore, $\mathbf{1}$ is a zero object for the category $\mathsf{Bimon}(\Mm)$. The same can be deduced for $\mathsf{Hopf}(\Mm)$ and $\Hopf$.
\end{proof}

Now, we study limits in the category $\Hopf$. Let $(\Mm,\ot,\mathbf{1},\sigma)$ be a symmetric monoidal category with equalizers such that $(-)\ot X$ and $X\ot(-)$ preserve them, for any $X\in\Mm$. The next goal it to show that the category $\Hopf$ is finitely complete.

\subsection{Binary products in $\Hopf$}\label{subsect:binprod}
\addtocontents{toc}{\protect\setcounter{tocdepth}{1}}
As recalled in Section \ref{sec:prelimiaries}, since $(\Mm,\ot,\mathbf{1},\sigma)$ is a symmetric monoidal category, the category $(\Hopf,\ot,\mathbf{1},\sigma)$ is also a symmetric monoidal category. Given $A,B$ in $\Hopf$, it is known that the following
\begin{equation}\label{def:binaryproduct}
(A\ot B,\pi_{A}:=\mathrm{Id}_{A}\ot\varepsilon_{B},\pi_{B}:=\varepsilon_{A}\ot\mathrm{Id}_{B})
\end{equation}
is the binary product of $A$ and $B$ in $\Hopf$, see e.g.\ \cite[Proposition 1.2]{Sterck} where the result is stated for the category $\mathsf{Bimon}_{\mathrm{coc}}(\Mm)$. \begin{invisible}
We include the proof for the sake of completeness. Given $f:C\to A$ and $g:C\to B$ in $\Hopf$, if there exists a morphism $\phi:C\to A\ot B$ in $\Hopf$ such that $\pi_{A}\phi=f$ and $\pi_{B}\phi=g$, then it is unique. Indeed, since $\phi$ must be in $\mathsf{Comon}(\Mm)$, so $(\phi\ot\phi)\Delta_{C}=\Delta_{A\ot B}\phi$ is satisfied, we get
\[
\begin{split}
\phi&=(\mathrm{Id}_{A}\ot\varepsilon_{B}\ot\varepsilon_{A}\ot\mathrm{Id}_{B})(\mathrm{Id}_{A}\ot\sigma_{A,B}\ot\mathrm{Id}_{B})(\Delta_{A}\ot\Delta_{B})\phi=(\pi_{A}\ot\pi_{B})\Delta_{A\ot B}\phi=(\pi_{A}\phi\ot\pi_{B}\phi)\Delta_{C}\\&=(f\ot g)\Delta_{C}.
\end{split}
\]
The morphism $\phi:=(f\ot g)\Delta_{C}:C\to A\ot B$ is in $\Hopf$ since $f,g$ and $\Delta_{C}$ are morphisms in $\Hopf$. We observe that $\Delta_{C}$ is in $\Hopf$ since $C$ is cocommutative. In fact, we have
\[
\begin{split}
(\Delta_{C}\ot\Delta_{C})\Delta_{C}&=(\mathrm{Id}_{C}\ot\Delta_{C}\ot\mathrm{Id}_{C})(\mathrm{Id}_{C}\ot\Delta_{C})\Delta_{C}=(\mathrm{Id}_{C}\ot\sigma_{C,C}\Delta_{C}\ot\mathrm{Id}_{C})(\mathrm{Id}_{C}\ot\Delta_{C})\Delta_{C}\\&=(\mathrm{Id}_{C}\ot\sigma_{C,C}\ot\mathrm{Id}_{C})(\mathrm{Id}_{C}\ot\Delta_{C}\ot\mathrm{Id}_{C})(\mathrm{Id}_{C}\ot\Delta_{C})\Delta_{C}\\&=(\mathrm{Id}_{C}\ot\sigma_{C,C}\ot\mathrm{Id}_{C})(\Delta_{C}\ot\Delta_{C})\Delta_{C}=\Delta_{C\ot C}\Delta_{C}.
\end{split}
\]
We have that
\[
\pi_{A}\phi=(\mathrm{Id}_{A}\ot\varepsilon_{B})(f\ot g)\Delta_{C}=f(\mathrm{Id}_{C}\ot\varepsilon_{C})\Delta_{C}=f, \quad \pi_{B}\phi=(\varepsilon_{A}\ot\mathrm{Id}_{B})(f\ot g)\Delta_{C}=g.
\]
Therefore, \eqref{def:binaryproduct} is the binary product of $A$ and $B$ in $\Hopf$. 
\end{invisible}
In fact, for two morphisms $f:C\to A$ and $g:C\to B$ in $\Hopf$, the unique morphism $\langle f,g\rangle:C\to A\ot B$ in $\Hopf$ such that $\pi_{A}\langle f,g\rangle=f$ and $\pi_{B}\langle f,g\rangle=g$ is given by $\langle f,g\rangle:=(f\ot g)\Delta_{C}$ and we observe that this morphism is in $\Hopf$ since $C$ is cocommutative. Note that this binary product construction can not be generalized to the non-cocommutative case. In fact, suppose that (\ref{def:binaryproduct}) is a binary product in $\mathsf{Hopf}(\Mm)$. Then, for any object $A$ in $\mathsf{Hopf}(\Mm)$, the morphism $\langle\mathrm{Id}_{A},\mathrm{Id}_{A}\rangle=\Delta_{A}$ is in $\mathsf{Hopf}(\Mm)$. This implies that $A$ is cocommutative.

\begin{remark}\label{cartesian of Hopf}
Recall that a monoidal category $(\Mm,\ot,\mathbf{1})$ is said to be \textit{cartesian} if the tensor product of two objects coincides with their binary product and the unit object is a terminal object.
Since the unit object of $\Hopf$ is a terminal object, the monoidal category $(\Hopf,\ot,\mathbf{1})$ is cartesian. We point out that also the monoidal category $(\mathsf{Comon}_{\mathrm{coc}}(\Mm),\ot,\mathbf{1})$ is cartesian, for any symmetric monoidal category $(\Mm,\ot,\mathbf{1},\sigma)$, see e.g.\ \cite[Corollary 2.24]{KM}. 
\end{remark}

The category $\mathsf{Comon}_{\mathrm{coc}}(\Mm)$ could not have equalizers. This happens under some suitable assumptions on $\Mm$, as a consequence of the following result:

\begin{proposition}[{cf.\ dual of \cite[Theorem 2.3]{Porstequalizer}}]\label{lem2.3Porst}
Let $(\Mm,\otimes,\mathbf{1})$ be a monoidal category such that $\Mm$ has equalizers which are preserved by the functors $X\ot(-)$ and $(-)\ot X$, for any $X$ in $\Mm$. For any pair of morphisms $\alpha, \beta: A\to X$ in $\mathsf{Comon}(\Mm)$, consider the following equalizer in $\Mm$:
\begin{equation}\label{eq:Lambda}
\begin{tikzcd}
	E & A &&&&& A\ot X\ot A
	\arrow[from=1-1, to=1-2, "e"]
	\arrow[shift left, from=1-2, to=1-7, "\Lambda_{\alpha}:=(\mathrm{Id}_{A}\ot\alpha\ot\mathrm{Id}_{A})(\Delta_{A}\ot\mathrm{Id}_{A})\Delta_{A}"]
	\arrow[shift right, from=1-2, to=1-7, "\Lambda_{\beta}:=(\mathrm{Id}_{A}\ot\beta\ot\mathrm{Id}_{A})(\Delta_{A}\ot\mathrm{Id}_{A})\Delta_{A}"']
\end{tikzcd}
\end{equation}
Then, $E$ carries a (unique) comonoid structure such that $e$ becomes a morphism in $\mathsf{Comon}(\Mm)$ and 
\[\begin{tikzcd}
	E & A & X
	\arrow[from=1-1, to=1-2, "e"]
	\arrow[shift left, from=1-2, to=1-3, "\alpha"]
	\arrow[shift right, from=1-2, to=1-3, "\beta"']
\end{tikzcd}\]
is an equalizer in $\mathsf{Comon}(\Mm)$. In particular, the category $\mathsf{Comon}(\Mm)$ has equalizers. 
\end{proposition}

\begin{remark}
    The equalizer \eqref{eq:Lambda} is a coreflexive equalizer, i.e.\ an equalizer of parallel morphisms that have common retraction (in this case $\varepsilon_{A}\ot\varepsilon_{X}\ot\mathrm{Id}_{A}$). In fact, for the previous result to hold, it is sufficient that $\Mm$ has coreflexive equalizers which are preserved by the functors $X\ot(-)$ and $(-)\ot X$, for any $X$ in $\Mm$. The same assumption is sufficient for Corollary \ref{lemma:limitscocombim}, Corollary \ref{cor:comoncocfincomplete}, Proposition \ref{prop:Hopfcocfintcomplete} and Proposition \ref{prop:protomodularity}.
\end{remark}

Under the cocommutativity assumption, the form of the equalizers can be simplified.

\begin{corollary}\label{lemma:limitscocombim}
    Let $(\Mm,\ot,\mathbf{1},\sigma)$ be a braided monoidal category such that $\Mm$ has equalizers and the functors $X\ot(-)$ and $(-)\ot X$ preserve them, for any $X$ in $\Mm$. Then, the category $\mathsf{Comon}_{\mathrm{coc}}(\Mm)$ has equalizers. In particular, this happens when $(\Mm,\ot,\mathbf{1},\sigma)$ is an abelian braided monoidal category. 
    
    More explicitly, given morphisms $\alpha,\beta:A\to X$ in $\mathsf{Comon}_{\mathrm{coc}}(\Mm)$, the equalizer of the pair $(\alpha,\beta)$ in $\mathsf{Comon}_{\mathrm{coc}}(\Mm)$ is given by the equalizer of the pair $((\alpha\ot\mathrm{Id}_{A})\Delta_{A},(\beta\ot\mathrm{Id}_{A})\Delta_{A})$ in $\Mm$.
\end{corollary}

\begin{proof}
    By Proposition \ref{lem2.3Porst}, the equalizer of the pair $(\alpha,\beta)$ in $\mathsf{Comon}(\Mm)$ is given by the equalizer \eqref{eq:Lambda} of the pair $(\Lambda_{\alpha},\Lambda_{\beta})$ in $\Mm$. To obtain that the latter is also the equalizer of $(\alpha,\beta)$ in $\mathsf{Comon}_{\mathrm{coc}}(\Mm)$, it suffices to show that $E$ is in $\mathsf{Comon}_{\mathrm{coc}}(\Mm)$. Since $\sigma_{A,A}\Delta_{A}=\Delta_{A}$, we have
\[
(e\ot e)\sigma_{E,E}\Delta_{E}=\sigma_{A,A}(e\ot e)\Delta_{E}=\sigma_{A,A}\Delta_{A}e=\Delta_{A}e=(e\ot e)\Delta_{E},
\]
hence $\sigma_{E,E}\Delta_{E}=\Delta_{E}$ since $e\ot e$ is a monomorphism in $\Mm$. Thus, the category $\mathsf{Comon}_{\mathrm{coc}}(\Mm)$ has equalizers as well. Moreover, the equalizer of the pair $(\Lambda_{\alpha},\Lambda_{\beta})$ in $\Mm$ is isomorphic to the equalizer of the pair $((\mathrm{Id}_{A}\ot\sigma_{X,A})\Lambda_{\alpha},(\mathrm{Id}_{A}\ot\sigma_{X,A})\Lambda_{\beta})$ in $\Mm$, since $\mathrm{Id}_{A}\ot\sigma_{X,A}$ is an isomorphism in $\Mm$. In addition, since $A$ is cocommutative, we get
\[
\begin{split}
(\mathrm{Id}_{A}\ot\sigma_{X,A})\Lambda_\alpha&=(\mathrm{Id}_{A}\ot\sigma_{X,A})(\mathrm{Id}_{A}\ot\alpha\ot\mathrm{Id}_{A})(\Delta_{A}\ot\mathrm{Id}_{A})\Delta_{A}\\&=(\mathrm{Id}_{A}\ot \mathrm{Id}_{A}\ot\alpha)(\mathrm{Id}_{A}\ot\sigma_{A,A})(\mathrm{Id}_{A}\ot\Delta_A)\Delta_A\\&=(\mathrm{Id}_{A}\ot \mathrm{Id}_{A}\ot\alpha)(\mathrm{Id}_{A}\ot\Delta_A)\Delta_A\\&=(\Delta_{A}\ot\mathrm{Id}_{X})(\mathrm{Id}_{A}\ot\alpha)\Delta_{A}.
\end{split}
\]
Therefore, since $\Delta_{A}\ot\mathrm{Id}_{X}$ has retraction $\varepsilon_{A}\ot\mathrm{Id}_{A}\ot\mathrm{Id}_{X}$ in $\Mm$, the equalizer of $((\mathrm{Id}_{A}\ot\sigma_{X,A})\Lambda_{\alpha},(\mathrm{Id}_{A}\ot\sigma_{X,A})\Lambda_{\beta})$ in $\Mm$ is isomorphic to the equalizer of  $((\mathrm{Id}_{A}\ot\alpha)\Delta_{A},(\mathrm{Id}_{A}\ot\beta)\Delta_{A})$ in $\Mm$, and then also to that of the pair $((\alpha\ot\mathrm{Id}_{A})\Delta_{A},(\beta\ot\mathrm{Id}_{A})\Delta_{A})$ in $\Mm$.
\end{proof}

Recall that a category is finitely complete if and only if it has a terminal object, binary products and equalizers, see e.g. \cite[Proposition 2.8.2]{Borceux}. As recalled in Remark \ref{cartesian of Hopf}, $(\mathsf{Comon}_{\mathrm{coc}}(\Mm),\ot,\mathbf{1})$ is cartesian monoidal when $(\Mm,\ot,\mathbf{1},\sigma)$ is symmetric. By Corollary \ref{lemma:limitscocombim}, one obtains the following result.

\begin{corollary}\label{cor:comoncocfincomplete}
    Let $(\Mm,\ot,\mathbf{1},\sigma)$ be a symmetric monoidal category which has equalizers that are preserved by $X\ot(-)$ and $(-)\ot X$, for any $X\in\Mm$. Then, the category $\mathsf{Comon}_{\mathrm{coc}}(\Mm)$ is finitely complete. In particular, this happens when $(\Mm,\ot,\mathbf{1},\sigma)$ is an abelian symmetric monoidal category.
\end{corollary}

Now, we turn our attention to $\Hopf$. As observed in Remark \ref{cartesian of Hopf}, $(\Hopf,\ot,\mathbf{1})$ is cartesian monoidal. 
We then have the following result concerning limits in $\Hopf$:

\begin{proposition}\label{prop:Hopfcocfintcomplete}
Let $(\Mm,\ot,\mathbf{1},\sigma)$ be a symmetric monoidal category with equalizers which are preserved by $X\ot(-)$ and $(-)\ot X$, for any $X\in\Mm$. Then, the category $\mathsf{Hopf}_{\mathrm{coc}}(\Mm)$ is finitely complete. In particular, this happens when $(\Mm,\ot,\mathbf{1},\sigma)$ is an abelian symmetric monoidal category.
\end{proposition}

\begin{proof}
By Corollary \ref{cor:comoncocfincomplete}, we know that $\mathsf{Comon}_{\mathrm{coc}}(\Mm)$ is finitely complete. Since $\Hopf=\mathsf{Grp}(\mathsf{Comon}_{\mathrm{coc}}(\Mm))$, see e.g.\ \cite[Remark 3.3]{Porst}, we obtain that $\Hopf$ is finitely complete, see e.g. \cite[Exercise 3.1.2]{Bourn-book}. 
\end{proof}

\subsection{Equalizers in $\Hopf$}\label{subsec:equalizers} 
\addtocontents{toc}{\protect\setcounter{tocdepth}{1}}
Suppose that the braided monoidal category $(\Mm,\ot,\mathbf{1},\sigma)$ has equalizers which are preserved by $X\ot(-)$ and $(-)\ot X$, for any $X\in\Mm$. 

By Proposition \ref{prop:Hopfcocfintcomplete} and \cite[Proposition 2.8.2]{Borceux}, we know that $\Hopf$ has equalizers for any pair of morphisms. Moreover, since the forgetful functor $\Hopf=\mathsf{Grp}(\mathsf{Comon}_{\mathrm{coc}}(\Mm))\to\mathsf{Comon}_{\mathrm{coc}}(\Mm)$ preserves limits, we can write down the explicit description of equalizers and kernels in $\Hopf$. Note that the symmetric assumption for the braided monoidal category $(\Mm,\ot,\mathbf{1},\sigma)$ in Proposition \ref{prop:Hopfcocfintcomplete} is used for the binary products, while it is not necessary for the equalizers. We recall the construction of equalizers and kernels in $\Hopf$. In fact, the construction of equalizers in $\mathsf{Bimon}_{\mathrm{coc}}(\Mm)$ is the same, which is given in e.g.\ \cite{Sterck}.

\begin{lemma}\label{lem:eqkernel}
For morphisms $f,g:A\to B$ in $\Hopf$, the equalizer of the pair $(f,g)$ in $\Hopf$ is given by the equalizer of the pair $((g\ot\mathrm{Id})\Delta_{A},(f\ot\mathrm{Id})\Delta_{A})$ in $\Mm$. Consequently, the kernel of $f$ in $\Hopf$, which we denote by $\mathsf{hker}(f):\mathsf{HKer}(f)\to A$, is the equalizer of the pair $((f\ot\mathrm{Id}_{A})\Delta_{A},u_{B}\ot\mathrm{Id}_{A})$ in $\Mm$.
\end{lemma}

\begin{invisible}
\begin{proof}
Denote $i:K\to A$ by the equalizer of the pair $(f,g)$ in $\mathsf{Comon}_{\mathrm{coc}}(\Mm)$ as in Corollary \ref{lemma:limitscocombim}, we claim that $K$ is in $\Hopf$ and it is the equalizer of the pair $(f,g)$ in $\Hopf$. In fact, since
\[
fm_{A}(i\ot i)=m_{B}(f\ot f)(i\ot i)=m_{B}(g\ot g)(i\ot i)=gm_{A}(i\ot i), \qquad fu_{A}=u_{B}=gu_{A},
\]
there exist unique morphisms $m_{K}:K\ot K\to K$ and $u_{K}:\mathbf{1}\to K$ in $\Mm$ such that $m_{A}(i\ot i)=im_{K}$ and $iu_{K}=u_{A}$. One can verify that $(K,m_{K},u_{K},\Delta_{K},\varepsilon_{K})$ is in $\mathsf{Bimon}_{\mathrm{coc}}(\Mm)$. Moreover, since $fS_{A}i=S_{B}fi=S_{B}gi=gS_{A}i$, there exists a unique morphism $S_{K}:K\to K$ in $\Mm$ such that $S_{A}i=iS_{K}$. One can verify that $(K,m_{K},u_{K},\Delta_{K},\varepsilon_{K},S_{K})$ is in $\Hopf$. Furthermore, for any morphism $l:C\to A$ in $\Hopf$ such that $fl=gl$, there exists a unique morphism $p:C\to K$ in $\mathsf{Comon}_{\mathrm{coc}}(\Mm)$ such that $ip=l$. We have
\[
ipu_{C}=lu_{C}=u_{A}=iu_{K},\quad ipm_{C}=lm_{C}=m_{A}(l\ot l)=m_{A}(ip\ot ip)=im_{K}(p\ot p).
\]
Since $i$ is a monomorphism in $\Mm$, we obtain that $p$ is a morphism in $\mathsf{Mon}(\Mm)$. Hence, $p$ is the unique morphism in $\Hopf$ such that $ip=l$, then $i:K\to A$ is the equalizer of the pair $(f,g)$ in $\Hopf$. Therefore, the equalizer of $(f,g)$ in $\Hopf$ is given by the equalizer of $((g\ot\mathrm{Id})\Delta_{A},(f\ot\mathrm{Id})\Delta_{A})$ in $\Mm$. 

By Lemma \ref{lem:Hopfpointed}, the zero morphism from $A$ to $B$ is $u_B \varepsilon_A:A \to B$. Hence, the kernel of $f$ in $\Hopf$ is the equalizer of the pair $(u_{B}\ot\mathrm{Id}_{A},(f\ot\mathrm{Id}_{A})\Delta_{A})$ in $\Mm$.
\end{proof}
\end{invisible}

\begin{remark}
We recall that a category is \textit{locally presentable} if and only if it is \textit{accessible} and complete, if and only if it is accessible and cocomplete, see e.g.\ \cite[Corollary 2.47]{AR} (we refer the reader to \cite{AR} for the definitions of local presentability and accessibility of a category). By \cite[Proposition 4.1.1]{Porst} we know that $\Hopf$ is always accessible for any symmetric monoidal category $(\Mm,\ot,\mathbf{1},\sigma)$. Therefore, $\Hopf$ is locally presentable if and only if it is complete, if and only if it is cocomplete.  
\end{remark}

\begin{remark}\label{rmk:comonlocpres}
If $(\Mm,\ot,\mathbf{1},\sigma)$ is a symmetric monoidal category such that $\Mm$ is locally presentable (this is called \textit{admissible monoidal} in \cite{Porst3}), we get that $\mathsf{Comon}(\Mm)$ is locally presentable by \cite[Proposition 2.9]{HLFV}. Thus, $\mathsf{Comon}_{\mathrm{coc}}(\Mm)=\mathsf{Comon}(\mathsf{Comon}(\Mm))$ is also locally presentable. In particular, $\mathsf{Comon}_{\mathrm{coc}}(\Mm)$ is (finitely) complete and so $\Hopf$ is finitely complete, see e.g.\ \cite[Exercise 3.1.2]{Bourn-book}.
\end{remark}

\subsection{Protomodularity}

Under the same assumptions on $\Mm$ of the previous results, we get that $\Hopf$ is (Bourn)-protomodular.

\begin{proposition}\label{prop:protomodularity}
    Let $(\Mm,\ot,\mathbf{1},\sigma)$ be a symmetric monoidal category with equalizers which are preserved by $X\ot(-)$ and $(-)\ot X$, for any $X\in\Mm$. Then, the category $\Hopf$ is (Bourn)-protomodular. In particular, this happens when $(\Mm,\ot,\mathbf{1},\sigma)$ is an abelian symmetric monoidal category.
\end{proposition}

\begin{proof}
Recall that $\mathsf{Grp}(\Cc)$ is protomodular if $\Cc$ has finite limits, see \cite[Proposition 3.24]{BG}. By Corollary \ref{cor:comoncocfincomplete}, we know that $\mathsf{Comon}_{\mathrm{coc}}(\Mm)$ is finitely complete. Since $\Hopf=\mathsf{Grp}(\mathsf{Comon}_{\mathrm{coc}}(\Mm))$, we get that $\Hopf$ is protomodular.
\end{proof}

By Remark \ref{rmk:comonlocpres} and \cite[Proposition 3.24]{BG}, we also obtain:

\begin{corollary}
Let $(\Mm,\ot,\mathbf{1},\sigma)$ be a symmetric monoidal category such that $\Mm$ is locally presentable. Then, $\Hopf$ is (Bourn)-protomodular.
\end{corollary}

\begin{remark}
    As recalled in Section \ref{sec:prelimiaries}, for a pointed finitely complete category $\Cc$, the protomodularity is equivalent to the validity of the Split Short Five Lemma in $\Cc$. Hence, by Proposition \ref{prop:protomodularity}, the Split Short Five Lemma holds in $\Hopf$. This could also be proven explicitly using \cite[Theorem 3.10.4]{HS}, given for arbitrary braided monoidal categories, which goes back to the preprint version of \cite{Bespalov} (see also \cite{BeDr}).
\end{remark}

We recall that, in any pointed finitely complete protomodular category $\Cc$, regular epimorphisms are exactly those morphisms $f$ in $\Cc$ such that $f=\mathsf{coker}(\mathsf{ker}(f))$, see e.g.\ \cite[Proposition 3.1.23]{BB}, so regular epimorphisms coincide with cokernels in $\Cc$. Hence, we obtain the following result:

\begin{corollary}\label{coeqequalcoker}
    Let $(\Mm,\ot,\mathbf{1},\sigma)$ be a symmetric monoidal category with equalizers which are preserved by $X\ot(-)$ and $(-)\ot X$, for any $X\in\Mm$. Regular epimorphisms coincide with cokernels in $\Hopf$.
\end{corollary}

\begin{proof}
    This follows since $\Hopf$ is pointed (Lemma \ref{lem:Hopfpointed}), finitely complete (Proposition \ref{prop:Hopfcocfintcomplete}) and protomodular (Proposition \ref{prop:protomodularity}).
\end{proof}

Since we know the explicit form of equalizers and binary products in $\Hopf$, we can compute pullbacks in $\Hopf$ \cite[Proposition 2.8.2]{Borceux}. We describe them for the sake of completeness. 


\subsection{Pullbacks in $\Hopf$}\label{subsec:pullbacksHopf}
Let $f:A\to C$ and $g:B \to C$ be morphisms in $\Hopf$. The pullback object of $A$ and $B$ over $C$ is given by the equalizer in $\Hopf$ of the pair $(f\pi_{A},g\pi_{B})$, 
where $(A\ot B,\pi_{A},\pi_{B})$ is the binary product of $A$ and $B$ in $\Hopf$ constructed in \eqref{def:binaryproduct} (see e.g. the proof of \cite[Proposition 2.8.2]{Borceux}). We know that the equalizer of $(f\pi_{A},g\pi_{B})$ in $\Hopf$ is given by the equalizer of the pair $((f\pi_{A}\ot\mathrm{Id}_{A\ot B})\Delta_{A\ot B},(g\pi_{B}\ot\mathrm{Id}_{A\ot B})\Delta_{A\ot B})$ in $\Mm$. We compute
\[
\begin{split}
(f\pi_{A}\ot\mathrm{Id}_{A\ot B})\Delta_{A\ot B}&=(f(\mathrm{Id}_{A}\ot\varepsilon_{B})\ot\mathrm{Id}_{A\ot B})(\mathrm{Id}_{A}\ot\sigma_{A,B}\ot\mathrm{Id}_{B})(\Delta_{A}\ot\Delta_{B})\\&=(f\ot\mathrm{Id}_{A\ot B})(\mathrm{Id}_{A\ot A}\ot\varepsilon_{B}\ot\mathrm{Id}_{B})(\Delta_{A}\ot\Delta_{B})\\&=(f\ot\mathrm{Id}_{A\ot B})(\Delta_{A}\ot\mathrm{Id}_{B})
\end{split}
\]
and, similarly, 
\[
\begin{split}
(g\pi_{B}\ot\mathrm{Id}_{A\ot B})\Delta_{A\ot B}&=(g(\varepsilon_{A}\ot\mathrm{Id}_{B})\ot\mathrm{Id}_{A\ot B})(\mathrm{Id}_{A}\ot\sigma_{A,B}\ot\mathrm{Id}_{B})(\Delta_{A}\ot\Delta_{B})\\&=(g\ot\mathrm{Id}_{A\ot B})(\sigma_{A,B}\ot\mathrm{Id}_{B})(\mathrm{Id}_{A}\ot\Delta_{B})\\&=(\sigma_{A,C}\ot\mathrm{Id}_{B})(\mathrm{Id}_{A}\ot g\ot\mathrm{Id}_{B})(\mathrm{Id}_{A}\ot\Delta_{B}),
\end{split}
\]
so the pullback object of the pair $(f,g)$ in $\Hopf$ is given by the equalizer of the pair $((f\ot\mathrm{Id}_{A})\Delta_{A}\ot\mathrm{Id}_{B},(\sigma_{A,C}\ot\mathrm{Id}_{B})(\mathrm{Id}_{A}\ot(g\ot\mathrm{Id}_{B})\Delta_{B})$ in $\Mm$. Since $\sigma_{A,C}\ot\mathrm{Id}_{B}$ has inverse in $\Mm$ given by $\sigma_{C,A}\ot\mathrm{Id}_{B}$ and $\sigma_{C,A}(f\ot\mathrm{Id}_{A})\Delta_{A}=(\mathrm{Id}_{A}\ot f)\sigma_{A,A}\Delta_{A}=(\mathrm{Id}_{A}\ot f)\Delta_{A}$, the previous equalizer is also the equalizer of the pair $((\mathrm{Id}_{A}\ot f)\Delta_{A}\ot\mathrm{Id}_{B},\mathrm{Id}_{A}\ot(g\ot\mathrm{Id}_{B})\Delta_{B})$ in $\Mm$. Therefore, the pullback object is the cotensor product $A\square_{C}B$ described as in \eqref{def:cotensor}, for $\rho_{A}:=(\mathrm{Id}_{A}\ot f)\Delta_{A}$ and $\lambda_{B}:=(g\ot\mathrm{Id}_{B})\Delta_{B}$.

\section{Coequalizers in $\Hopf$}\label{sec:colimits}

In this section, $(\Mm,\ot,\mathbf{1},\sigma)$ will denote a braided monoidal category with equalizers and coequalizers that are preserved by $\ot$. We construct coequalizers in $\Hopf$, which will be used to prove the regularity of $\Hopf$. More generally, we describe coequalizers in $\mathsf{Hopf}(\Mm)$. \medskip

Given two parallel morphisms $f,g:A\to B$ in $\mathsf{Hopf}(\Mm)$, we consider the morphisms $\phi_{f},\phi_{g}:B\ot A\ot B\to B$ in $\Mm$ defined by
\begin{equation}\label{def:defphi}
\phi_f:=m_{B}(m_{B}\ot\mathrm{Id}_{B})(\mathrm{Id}_{B}\ot f\ot\mathrm{Id}_{B}),\qquad \phi_g:=m_{B}(m_{B}\ot\mathrm{Id}_{B})(\mathrm{Id}_{B}\ot g\ot\mathrm{Id}_{B})
\end{equation}
and we define $\pi:=\mathsf{coeq}(\phi_f,\phi_g):B\to\mathsf{Coeq}(\phi_f,\phi_g)$ in $\Mm$.
Observe that 
\[
\begin{split}
\phi_{f}(u_{B}\ot\mathrm{Id}_{A}\ot u_{B})&=m_{B}(m_{B}\ot\mathrm{Id}_{B})(\mathrm{Id}_{B}\ot f\ot\mathrm{Id}_{B})(u_{B}\ot\mathrm{Id}_{A}\ot u_{B})\\&=m_{B}(m_{B}\ot\mathrm{Id}_{B})(u_{B}\ot\mathrm{Id}_{B}\ot u_{B})f=f,
\end{split}
\]
which implies $\pi f=\pi\phi_{f}(u_{B}\ot\mathrm{Id}_{A}\ot u_{B})=\pi\phi_{g}(u_{B}\ot\mathrm{Id}_{A}\ot u_{B})=\pi g$. 

\begin{lemma}\label{lem:pihopf}
The morphism $\pi$ is in $\mathsf{Hopf}(\Mm)$.
\end{lemma}
 
\begin{proof}
 Since $(-)\ot B$ preserves coequalizers in $\Mm$, we have that $\pi\ot\mathrm{Id}_{B}$ is the coequalizer of the pair $(\phi_{f}\ot\mathrm{Id}_{B},\phi_{g}\ot\mathrm{Id}_{B})$ in $\Mm$. We compute
\[
\begin{split}
    \pi m_{B}(\phi_{f}\ot\mathrm{Id}_{B})&=\pi m_{B}(m_{B}\ot\mathrm{Id}_{B})(m_{B}\ot\mathrm{Id}_{B\ot B})(\mathrm{Id}_{B}\ot f\ot\mathrm{Id}_{B\ot B})\\&=\pi m_{B}(\mathrm{Id}_{B}\ot m_{B})(m_{B}\ot\mathrm{Id}_{B\ot B})(\mathrm{Id}_{B}\ot f\ot\mathrm{Id}_{B\ot B})\\&=\pi m_{B}(m_{B}\ot\mathrm{Id}_{B})(\mathrm{Id}_{B}\ot f\ot\mathrm{Id}_{B})(\mathrm{Id}_{B}\ot\mathrm{Id}_{A}\ot m_{B})\\&=\pi m_{B}(m_{B}\ot\mathrm{Id}_{B})(\mathrm{Id}_{B}\ot g\ot\mathrm{Id}_{B})(\mathrm{Id}_{B}\ot\mathrm{Id}_{A}\ot m_{B})\\&=\pi m_{B}(\phi_{g}\ot\mathrm{Id}_{B}),
\end{split}
\]
so there exists a unique morphism $\psi:\mathsf{Coeq}(\phi_{f},\phi_{g})\ot B\to\mathsf{Coeq}(\phi_{f},\phi_{g})$ in $\Mm$ such that $\psi(\pi\ot\mathrm{Id}_{B})=\pi m_{B}$. Moreover, since $\mathsf{Coeq}(\phi_{f},\phi_{g})\ot(-)$ preserves coequalizers in $\Mm$, we have that $\mathrm{Id}_{\mathsf{Coeq}}\ot\pi$ is the coequalizer of the pair $(\mathrm{Id}_{\mathsf{Coeq}}\ot\phi_{f},\mathrm{Id}_{\mathsf{Coeq}}\ot\phi_{g})$ in $\Mm$. Then, since
\[
\begin{split}
\pi m_{B}(\mathrm{Id}_{B}\ot\phi_f)&=\pi m_{B}(\mathrm{Id}_{B}\ot m_{B})(\mathrm{Id}_{B}\ot m_{B}\ot\mathrm{Id}_{B})(\mathrm{Id}_{B\ot B}\ot f\ot\mathrm{Id}_{B})\\&=\pi m_{B}(m_{B}\ot\mathrm{Id}_{B})(m_{B}\ot\mathrm{Id}_{B\ot B})(\mathrm{Id}_{B\ot B}\ot f\ot\mathrm{Id}_{B})\\&=\pi m_{B}(m_{B}\ot\mathrm{Id}_{B})(\mathrm{Id}_{B}\ot f\ot\mathrm{Id}_{B})(m_{B}\ot\mathrm{Id}_{A\ot B})\\&=\pi m_{B}(m_{B}\ot\mathrm{Id}_{B})(\mathrm{Id}_{B}\ot g\ot\mathrm{Id}_{B})(m_{B}\ot\mathrm{Id}_{A\ot B})\\&=\pi m_{B}(\mathrm{Id}_{B}\ot\phi_g),
\end{split}
\]
we obtain
\[
\begin{split}
\psi(\mathrm{Id}_{\mathsf{Coeq}}\ot\phi_f)(\pi\ot\mathrm{Id}_{B\ot A\ot B})&=\psi(\pi\ot\mathrm{Id}_{B})(\mathrm{Id}_{B}\ot\phi_f)=\pi m_{B}(\mathrm{Id}_{B}\ot\phi_f)=\pi m_{B}(\mathrm{Id}_{B}\ot\phi_g)\\&=\psi(\mathrm{Id}_{\mathsf{Coeq}}\ot\phi_g)(\pi\ot\mathrm{Id}_{B\ot A\ot B}).
\end{split}
\]
Since $\pi\ot\mathrm{Id}_{B\ot A\ot B}$ is an epimorphism in $\Mm$, we get $\psi(\mathrm{Id}_{\mathsf{Coeq}}\ot\phi_f)=\psi(\mathrm{Id}_{\mathsf{Coeq}}\ot\phi_g)$. Hence, there exists a unique morphism $m_{\mathsf{Coeq}}:\mathsf{Coeq}(\phi_{f},\phi_{g})\ot\mathsf{Coeq}(\phi_{f},\phi_{g})\to\mathsf{Coeq}(\phi_{f},\phi_{g})$ in $\Mm$ such that $m_{\mathsf{Coeq}}(\mathrm{Id}_{\mathsf{Coeq}}\ot\pi)=\psi$. By defining $u_{\mathsf{Coeq}}:=\pi u_{B}$, one can prove that $(\mathsf{Coeq}(\phi_{f},\phi_{g}),m_{\mathsf{Coeq}},u_{\mathsf{Coeq}})$ is in $\mathsf{Mon}(\Mm)$ and $\pi m_{B}=\psi(\pi\ot\mathrm{Id}_{B})=m_{\mathsf{Coeq}}(\pi\ot\pi)$ and $\pi u_{B}=u_{\mathsf{Coeq}}$, i.e.\ $\pi$ is in $\mathsf{Mon}(\Mm)$.

Since
\[
\begin{split}
&\Delta_{B}\phi_{f}=\Delta_{B}m_{B}(m_{B}\ot\mathrm{Id}_{B})(\mathrm{Id}_{B}\ot f\ot\mathrm{Id}_{B})\\
= &(m_{B}\ot m_{B})(\mathrm{Id}_{B}\ot\sigma_{B,B}\ot\mathrm{Id}_{B})(\Delta_{B}\otimes\Delta_{B})(m_{B}\ot\mathrm{Id}_{B})(\mathrm{Id}_{B}\ot f\ot\mathrm{Id}_{B})\\
=&(m_{B}\ot m_{B})(\mathrm{Id}_{B}\ot\sigma_{B,B}\ot\mathrm{Id}_{B})(m_{B}\ot m_{B}\ot\mathrm{Id}_{B\ot B})(\mathrm{Id}_{B}\ot\sigma_{B,B}\ot\mathrm{Id}_{B}\ot\mathrm{Id}_{B\ot B})\\
&(\Delta_{B}\ot\Delta_{B}\ot\Delta_{B})(\mathrm{Id}_{B}\ot f\ot\mathrm{Id}_{B})\\
=&(m_{B}\ot m_{B})(\mathrm{Id}_{B}\ot\sigma_{B,B}\ot\mathrm{Id}_{B})(m_{B}\ot m_{B}\ot\mathrm{Id}_{B\ot B})(\mathrm{Id}_{B}\ot\sigma_{B,B}\ot\mathrm{Id}_{B}\ot\mathrm{Id}_{B\ot B})\\
&(\Delta_{B}\ot(f\ot f)\Delta_{A}\ot\Delta_{B}),
\end{split}
\]
we obtain that
\begin{align*}
&\hspace{0.5cm}(\pi\ot\pi)\Delta_{B}\phi_{f}\\
&=
(\pi\ot\pi)(m_{B}\ot m_{B})(\mathrm{Id}_{B}\ot\sigma_{B,B}\ot\mathrm{Id}_{B})(m_{B}\ot m_{B}\ot\mathrm{Id}_{B\ot B})(\mathrm{Id}_{B}\ot\sigma_{B,B}\ot\mathrm{Id}_{B}\ot\mathrm{Id}_{B\ot B})\\&\hspace{0.5cm}(\Delta_{B}\ot(f\ot f)\Delta_{A}\ot\Delta_{B})\\&=(m_{\mathsf{Coeq}}\ot m_{\mathsf{Coeq}})(\pi\ot\pi\ot\pi\ot\pi)(\mathrm{Id}_{B}\ot\sigma_{B,B}\ot\mathrm{Id}_{B})(m_{B}\ot m_{B}\ot\mathrm{Id}_{B\ot B})\\&\hspace{0.5cm}(\mathrm{Id}_{B}\ot\sigma_{B,B}\ot\mathrm{Id}_{B\ot B\ot B})(\Delta_{B}\ot(f\ot f)\Delta_{A}\ot\Delta_{B})\\&=(m_{\mathsf{Coeq}}\ot m_{\mathsf{Coeq}})(\mathrm{Id}_{\mathsf{Coeq}}\ot\sigma_{\mathsf{Coeq},\mathsf{Coeq}}\ot\mathrm{Id}_{\mathsf{Coeq}})(\pi\ot\pi\ot\pi\ot\pi)(m_{B}\ot m_{B}\ot\mathrm{Id}_{B\ot B})\\&\hspace{0.5cm}(\mathrm{Id}_{B}\ot\sigma_{B,B}\ot\mathrm{Id}_{B\ot B\ot B})(\Delta_{B}\ot(f\ot f)\Delta_{A}\ot\Delta_{B})\\&=(m_{\mathsf{Coeq}}\ot m_{\mathsf{Coeq}})(\mathrm{Id}_{\mathsf{Coeq}}\ot\sigma_{\mathsf{Coeq},\mathsf{Coeq}}\ot\mathrm{Id}_{\mathsf{Coeq}})(m_{\mathsf{Coeq}}\ot m_{\mathsf{Coeq}}\ot\mathrm{Id}_{\mathsf{Coeq}\ot \mathsf{Coeq}})\\&\hspace{0.5cm}(\pi\ot\pi\ot\pi\ot\pi\ot\pi\ot\pi)(\mathrm{Id}_{B}\ot\sigma_{B,B}\ot\mathrm{Id}_{B}\ot\mathrm{Id}_{B\ot B})(\Delta_{B}\ot(f\ot f)\Delta_{A}\ot\Delta_{B})\\&=(m_{\mathsf{Coeq}}\ot m_{\mathsf{Coeq}})(\mathrm{Id}_{\mathsf{Coeq}}\ot\sigma_{\mathsf{Coeq},\mathsf{Coeq}}\ot\mathrm{Id}_{\mathsf{Coeq}})(m_{\mathsf{Coeq}}\ot m_{\mathsf{Coeq}}\ot\mathrm{Id}_{\mathsf{Coeq}\ot \mathsf{Coeq}})\\&\hspace{0.5cm}(\mathrm{Id}_{\mathsf{Coeq}}\ot\sigma_{\mathsf{Coeq},\mathsf{Coeq}}\ot\mathrm{Id}_{\mathsf{Coeq}\ot \mathsf{Coeq}\ot \mathsf{Coeq}})(\pi\ot\pi\ot\pi\ot\pi\ot\pi\ot\pi)(\Delta_{B}\ot(f\ot f)\Delta_{A}\ot\Delta_{B})\\&=(m_{\mathsf{Coeq}}\ot m_{\mathsf{Coeq}})(\mathrm{Id}_{\mathsf{Coeq}}\ot\sigma_{\mathsf{Coeq},\mathsf{Coeq}}\ot\mathrm{Id}_{\mathsf{Coeq}})(m_{\mathsf{Coeq}}\ot m_{\mathsf{Coeq}}\ot\mathrm{Id}_{\mathsf{Coeq}\ot \mathsf{Coeq}})\\&\hspace{0.5cm}(\mathrm{Id}_{\mathsf{Coeq}}\ot\sigma_{\mathsf{Coeq},\mathsf{Coeq}}\ot\mathrm{Id}_{\mathsf{Coeq}\ot \mathsf{Coeq}\ot \mathsf{Coeq}})(\pi\ot\pi\ot\pi\ot\pi\ot\pi\ot\pi)(\Delta_{B}\ot(g\ot g)\Delta_{A}\ot\Delta_{B})\\&=(\pi\ot\pi)\Delta_{B}\phi_{g}.
\end{align*}
Hence, there exists a unique morphism $\Delta_{\mathsf{Coeq}}:\mathsf{Coeq}(\phi_{f},\phi_{g})\to\mathsf{Coeq}(\phi_{f},\phi_{g})\ot\mathsf{Coeq}(\phi_{f},\phi_{g})$ in $\Mm$ such that $\Delta_{\mathsf{Coeq}}\pi=(\pi\ot\pi)\Delta_{B}$. Moreover, we have
\[
\begin{split}
\varepsilon_{B}\phi_{f}&=\varepsilon_{B}m_{B}(m_{B}\ot\mathrm{Id}_{B})(\mathrm{Id}_{B}\ot f\ot\mathrm{Id}_{B})=(\varepsilon_{B}\ot\varepsilon_{B})(m_{B}\ot\mathrm{Id}_{B})(\mathrm{Id}_{B}\ot f\ot\mathrm{Id}_{B})\\&=\varepsilon_{B}\ot\varepsilon_{B}f\ot\varepsilon_{B}=\varepsilon_{B}\ot\varepsilon_{A}\ot\varepsilon_{B}=\varepsilon_{B}\ot\varepsilon_{B}g\ot\varepsilon_{B}=\varepsilon_{B}\phi_{g},
\end{split}
\]
so there exists a unique morphism $\varepsilon_{\mathsf{Coeq}}:\mathsf{Coeq}(\phi_{f},\phi_{g})\to\mathbf{1}$ in $\Mm$ such that $\varepsilon_{\mathsf{Coeq}}\pi=\varepsilon_{B}$. One can prove that $(\mathsf{Coeq}(\phi_{f},\phi_{g}),\Delta_{\mathsf{Coeq}},\varepsilon_{\mathsf{Coeq}})$ is in $\mathsf{Comon}(\Mm)$ so that $\pi$ is in $\mathsf{Comon}(\Mm)$.
Furthermore, since $B$ is in $\mathsf{Bimon}(\Mm)$, also $\mathsf{Coeq}(\phi_{f},\phi_{g})$ is in $\mathsf{Bimon}(\Mm)$ so that $\pi$ is a morphism in $\mathsf{Bimon}(\Mm)$.

Finally, we compute
\[
\begin{split}
S_{B}\phi_{f}&=S_{B}m_{B}(m_{B}\ot\mathrm{Id}_{B})(\mathrm{Id}_{B}\ot f\ot\mathrm{Id}_{B})
=m_{B}\sigma_{B,B}(S_{B}\ot S_{B})(m_{B}\ot\mathrm{Id}_{B})(\mathrm{Id}_{B}\ot f\ot\mathrm{Id}_{B})
\\&=m_{B}\sigma_{B,B}(m_{B}\ot\mathrm{Id}_{B})(\sigma_{B,B}\ot\mathrm{Id}_{B})(S_{B}\ot S_{B}\ot S_{B})(\mathrm{Id}_{B}\ot f\ot\mathrm{Id}_{B})\\&=m_{B}\sigma_{B,B}(m_{B}\ot\mathrm{Id}_{B})(\sigma_{B,B}\ot\mathrm{Id}_{B})(\mathrm{Id}_{B}\ot f\ot\mathrm{Id}_{B})(S_{B}\ot S_{A}\ot S_{B})
\end{split}
\]
and then 
\[
\begin{split}
&\hspace{0.5cm} \pi S_{B}\phi_{f}=\pi m_{B}\sigma_{B,B}(m_{B}\ot\mathrm{Id}_{B})(\sigma_{B,B}\ot\mathrm{Id}_{B})(\mathrm{Id}_{B}\ot f\ot\mathrm{Id}_{B})(S_{B}\ot S_{A}\ot S_{B})\\&=m_{\mathsf{Coeq}}(\pi\ot\pi)\sigma_{B,B}(m_{B}\ot\mathrm{Id}_{B})(\sigma_{B,B}\ot\mathrm{Id}_{B})(\mathrm{Id}_{B}\ot f\ot\mathrm{Id}_{B})(S_{B}\ot S_{A}\ot S_{B})\\&=m_{\mathsf{Coeq}}\sigma_{\mathsf{Coeq},\mathsf{Coeq}}(\pi\ot\pi)(m_{B}\ot\mathrm{Id}_{B})(\sigma_{B,B}\ot\mathrm{Id}_{B})(\mathrm{Id}_{B}\ot f\ot\mathrm{Id}_{B})(S_{B}\ot S_{A}\ot S_{B})\\&=m_{\mathsf{Coeq}}\sigma_{\mathsf{Coeq},\mathsf{Coeq}}(m_{\mathsf{Coeq}}\ot\mathrm{Id}_{\mathsf{Coeq}})(\pi\ot\pi\ot\pi)(\sigma_{B,B}\ot\mathrm{Id}_{B})(\mathrm{Id}_{B}\ot f\ot\mathrm{Id}_{B})(S_{B}\ot S_{A}\ot S_{B})\\&=m_{\mathsf{Coeq}}\sigma_{\mathsf{Coeq},\mathsf{Coeq}}(m_{\mathsf{Coeq}}\ot\mathrm{Id}_{\mathsf{Coeq}})(\sigma_{\mathsf{Coeq},\mathsf{Coeq}}\ot\mathrm{Id}_{\mathsf{Coeq}})(\pi\ot\pi\ot\pi)(\mathrm{Id}_{B}\ot f\ot\mathrm{Id}_{B})(S_{B}\ot S_{A}\ot S_{B})\\&=m_{\mathsf{Coeq}}\sigma_{\mathsf{Coeq},\mathsf{Coeq}}(m_{\mathsf{Coeq}}\ot\mathrm{Id}_{\mathsf{Coeq}})(\sigma_{\mathsf{Coeq},\mathsf{Coeq}}\ot\mathrm{Id}_{\mathsf{Coeq}})(\pi\ot\pi\ot\pi)(\mathrm{Id}_{B}\ot g\ot\mathrm{Id}_{B})(S_{B}\ot S_{A}\ot S_{B})\\&=\pi S_{B}\phi_{g}.
\end{split}
\]
By the universal property, there exists a unique morphism $S_{\mathsf{Coeq}}:\mathsf{Coeq}(\phi_{f},\phi_{g})\to\mathsf{Coeq}(\phi_{f},\phi_{g})$ in $\Mm$ such that $\pi S_{B}=S_{\mathsf{Coeq}}\pi$. One can check that $(\mathsf{Coeq}(\phi_{f},\phi_{g}),m_{\mathsf{Coeq}},u_{\mathsf{Coeq}},\Delta_{\mathsf{Coeq}},\varepsilon_{\mathsf{Coeq}},S_{\mathsf{Coeq}})$ is in $\mathsf{Hopf}(\Mm)$ so that $\pi$ is in $\mathsf{Hopf}(\Mm)$.
\end{proof}

\begin{proposition}\label{prop:coequalizerHopf}
Given morphisms $f, g:A \to B$ in $\mathsf{Hopf}(\Mm)$, $\pi:B\to\mathsf{Coeq}(\phi_{f},\phi_{g})$ is the coequalizer of $(f,g)$ in $\mathsf{Hopf}(\Mm)$. The same happens for the category $\Hopf$.
\end{proposition}

\begin{proof}
By Lemma \ref{lem:pihopf}, we already know that $\pi$ is in $\mathsf{Hopf}(\Mm)$ and coequalizes the pair $(f,g)$ in $\mathsf{Hopf}(\Mm)$. Suppose there is a morphism $l:B\to D$ in $\mathsf{Hopf}(\Mm)$ such that $lf=lg$. 
Since $l$ is a morphism in $\mathsf{Mon}(\Mm)$, we get that $l\phi_{f}=l\phi_{g}$. Thus, there exists a unique morphism $p:\mathsf{Coeq}(\phi_{f},\phi_{g})\to D$ in $\Mm$ such that $p\pi=l$. Since $l$ and $\pi$ are morphisms in $\mathsf{Mon}(\Mm)$ and $\pi\ot \pi$ is an epimorphism in $\Mm$, we get that $p$ is a morphism in $\mathsf{Mon}(\Mm)$ (hence the unique morphism in $\mathsf{Mon}(\Mm)$ such that $p\pi=l$). Moreover, since $l$ and $\pi$ are morphisms in $\mathsf{Comon}(\Mm)$, we have
\[
\begin{split}
    \Delta_{D}p\pi=\Delta_{D}l=(l\ot l)\Delta_{B}=(p\ot p)(\pi\ot\pi)\Delta_{B}=(p\ot p)\Delta_{\mathsf{Coeq}}\pi,\ \varepsilon_{D}p\pi=\varepsilon_{D}l=\varepsilon_{B}=\varepsilon_{\mathsf{Coeq}}\pi,
\end{split}
\]
so that the morphism $p$ is in $\mathsf{Comon}(\Mm)$ as $\pi$ is an epimorphism in $\Mm$.
Therefore, $p$ is the unique morphism in $\mathsf{Hopf}(\Mm)$ such that $p\pi=l$, which means that $(\mathsf{Coeq}(\phi_{f},\phi_{g}),\pi)$ is the coequalizer of the pair $(f,g)$ in $\mathsf{Hopf}(\Mm)$. 

If we consider morphisms $f,g:A\to B$ in $\Hopf$ then $\mathsf{Coeq}(\phi_{f},\phi_{g})$ is in $\Hopf$. In fact, from
\[
\sigma_{\mathsf{Coeq},\mathsf{Coeq}}\Delta_{\mathsf{Coeq}}\pi=\sigma_{\mathsf{Coeq},\mathsf{Coeq}}(\pi\ot\pi)\Delta_{B}=(\pi\ot\pi)\sigma_{B,B}\Delta_{B}=(\pi\ot\pi)\Delta_{B}=\Delta_{\mathsf{Coeq}}\pi,
\]
we get $\sigma_{\mathsf{Coeq},\mathsf{Coeq}}\Delta_{\mathsf{Coeq}}=\Delta_{\mathsf{Coeq}}$. Therefore, $(\mathsf{Coeq}(\phi_{f},\phi_{g}),\pi)$ is the coequalizer of the pair $(f,g)$ in $\Hopf$. 
\end{proof}

From the explicit description of coequalizers in $\mathsf{Hopf}(\Mm)$, we can deduce how cokernels in $\mathsf{Hopf}(\Mm)$ are made. 

\begin{corollary}\label{cor:cokernel}
The cokernel of a morphism $f:A\to B$ in $\mathsf{Hopf}(\Mm)$, which we denote by $\mathsf{hcoker}(f):B\to\mathsf{HCoker}(f)$, is given by the coequalizer of the pair $(m_{B}(m_{B}\ot\mathrm{Id}_{B})(\mathrm{Id}_{B}\ot f\ot\mathrm{Id}_{B}), m_{B}(\mathrm{Id}_{B}\ot\varepsilon_{A}\ot\mathrm{Id}_{B}))$ in $\Mm$. The same holds for $\Hopf$.
\end{corollary}

As recalled in Section \ref{sec:prelimiaries}, the existence of the coequalizer for any pair of morphisms is not required in the definition of a semi-abelian category see Definition \ref{def:semi-abeliancategory}), but it is obtained as a consequence. In our strategy, we will use the explicit construction of coequalizers in $\Hopf$ to deduce the regularity of the category. \medskip

On the other hand, the existence of binary coproducts is required in the definition of a semi-abelian category. We make some comments in this regard. It is known that the forgetful functor $\mathsf{Comon}(\Mm)\to\Mm$ creates colimits, for any monoidal category $(\Mm,\ot,\mathbf{1})$ (we refer the reader to \cite[V 1]{MacLane-book} for the definition of creation of colimits). As recalled in Section \ref{sec:prelimiaries}, given a symmetric monoidal category $(\Mm,\ot,\mathbf{1},\sigma)$, the category $\mathsf{Bimon}(\Mm)\cong\mathsf{Comon}(\mathsf{Mon}(\Mm))$ is monoidal. Hence, the forgetful functor
\[
\mathsf{Bimon}_{\mathrm{coc}}(\Mm)\cong\mathsf{Comon}(\mathsf{Bimon}(\Mm))\cong \mathsf{Comon}(\mathsf{Comon}(\mathsf{Mon}(\Mm)))\to \mathsf{Comon}(\mathsf{Mon}(\Mm))
\to\mathsf{Mon}(\Mm)
\]
creates colimits. Therefore, if $\mathsf{Mon}(\Mm)$ has binary coproducts, $\mathsf{Bimon}_{\mathrm{coc}}(\Mm)$ also has binary coproducts. Hence, if in addition $\Hopf$ is closed under binary coproducts in $\mathsf{Bimon}_{\mathrm{coc}}(\Mm)$, then $\Hopf$ also has binary coproducts.

\begin{remark}
If $(\Mm,\ot,\mathbf{1})$ is an abelian monoidal category 
and the forgetful functor $\mathsf{Mon}(\Mm)\to\Mm$ has a left adjoint, then $\mathsf{Mon}(\Mm)$ has binary coproducts, see e.g.\ \cite[Theorem 3.5]{Porstequalizer}. Moreover, if the forgetful functor $\mathsf{Mon}(\Mm)\to\Mm$ is extremally monadic, $\Hopf$ is closed under binary coproducts in $\mathsf{Bimon}_{\mathrm{coc}}(\Mm)$, see e.g.\ \cite[Proposition 52]{Porst2}.
\end{remark}

\begin{remark}\label{rmk:colimitsHopf}
    Recall that if $\Mm$ is locally presentable then $\mathsf{Mon}(\Mm)$ and $\mathsf{Comon}(\Mm)$ are locally presentable by \cite[Proposition 2.9]{HLFV}. Since $\mathsf{Bimon}_{\mathrm{coc}}(\Mm)\cong \mathsf{Comon}(\mathsf{Comon}(\mathsf{Mon}(\Mm)))$, we obtain that $\mathsf{Bimon}_{\mathrm{coc}}(\Mm)$ is locally presentable by repeating the argument. In particular, it is complete and cocomplete. However, in general, $\mathsf{Hopf}(\Mm)$ is not closed under (co)limits in $\mathsf{Bimon}(\Mm)$. 
    If $\Mm$ is locally presentable, the fact that $\mathsf{Hopf}(\Mm)$ is closed under colimits in $\mathsf{Bimon}(\Mm)$ is equivalent to the other conditions given in \cite[Proposition 49]{Porst2}. We also point out that conditions on a braided monoidal category $(\Mm,\ot,\mathbf{1},\sigma)$ such that $\mathsf{Hopf}(\Mm)$ is closed in $\mathsf{Bimon}(\Mm)$ under colimits are given in \cite[Theorem 5.10]{AGV}. 
\end{remark}

\section{Newman's Theorem for $\Hopf$}\label{sec:NewmanTheorem}

In this section, $(\Mm,\ot,\mathbf{1},\sigma)$ will be a braided monoidal category which has equalizers and coequalizers that are preserved by $\ot$.

Recall that there is a well known bijective correspondence between Hopf subalgebras and left ideals which are also two-sided coideals of a cocommutative Hopf algebra, proven by K. Newman in \cite[Theorem 4.1]{Newman}. The goal of this section is to generalize this correspondence in a suitable way to the setting of $\Hopf$, which will be useful to prove the regularity of the category $\Hopf$. To this end, the following propositions provide the precise definitions of the aforementioned bijective correspondence maps.

\begin{proposition}\label{prop:definitionphi}
    Let $A$ be an object in $\Hopf$ and $i:K\to A$ be a monomorphism in $\Hopf$. Let $\pi:A\to Q$ be the coequalizer of the pair of morphisms $(m_{A}(\mathrm{Id}_{A}\ot i),\mathrm{Id}_{A}\ot\varepsilon_{K})$ in $\Mm$. Then, we have the following results.
\begin{itemize}
    \item[1)] The morphism $\pi$ is an epimorphism in $\mathsf{Comon}_{\mathrm{coc}}(_{A}\Mm)$. 
    \item[2)] The following diagram is a coequalizer in $\Mm$

\begin{equation}\label{picoequalizer}
\begin{tikzcd}
	A\square_{Q}A && A & Q
	\arrow[shift left, from=1-1, to=1-3, "(\varepsilon_{A}\ot\mathrm{Id}_{A})e_{A,A}"]
	\arrow[shift right, from=1-1, to=1-3, "(\mathrm{Id}_{A}\ot\varepsilon_{A})e_{A,A}"']
	\arrow[from=1-3, to=1-4, "\pi"]
\end{tikzcd}
\end{equation}
where $e_{A,A}:A\square_{Q}A\to A\ot A$ is the equalizer in $\Mm$ defined as in \eqref{def:cotensor} and $A$ is an object in ${}^{Q}\Mm$ with structure $(\pi\ot\mathrm{Id}_{A})\Delta_{A}$ and in $\Mm^{Q}$ with structure $(\mathrm{Id}_{A}\ot\pi)\Delta_{A}$.
\end{itemize}
\end{proposition}

\begin{proof}
    1). First, we prove that $\pi:A\to Q$ is in $\mathsf{Comon}_{\mathrm{coc}}(_{A}\Mm)$. Since $A\ot(-)$ preserves coequalizers in $\Mm$, we obtain that $\mathrm{Id}_{A}\ot\pi$ is the coequalizer of the pair $(\mathrm{Id}_{A}\ot m_{A}(\mathrm{Id}_{A}\ot i),\mathrm{Id}_{A\ot A}\ot\varepsilon_{K})$ in $\Mm$. Since 
    \begin{align*}
    \pi m_{A}(\mathrm{Id}_{A}\ot m_{A}(\mathrm{Id}_{A}\ot i))&=\pi m_{A}(\mathrm{Id}_{A} \ot i) (m_{A} \ot \mathrm{Id}_{K})=\pi (\mathrm{Id}_{A}\ot\varepsilon_{K})(m_{A} \ot \mathrm{Id}_{K})\\&= \pi m_{A} (\mathrm{Id}_{A\ot A}\ot\varepsilon_{K}),
    \end{align*}
    there exists a unique morphism $\mu_{Q}:A\ot Q\to Q$ in $\Mm$ such that $\mu_{Q}(\mathrm{Id}_{A}\ot\pi)=\pi m_{A}$. One can check that $Q$ becomes an object in $_{A}\Mm$ with action $\mu_{Q}$. This follows since $(A,m_{A},u_{A})$ is in $\mathsf{Mon}(\Mm)$ and $\pi$ is an epimorphism in $\Mm$, which is preserved by $A\ot(-)$. Because $\varepsilon_{A}m_{A}(\mathrm{Id}_{A}\ot i)=\varepsilon_{A}\ot\varepsilon_{A}i=\varepsilon_{A}(\mathrm{Id}_{A}\ot\varepsilon_{K})$, there exists a unique morphism $\varepsilon_{Q}:Q\to\mathbf{1}$ in $\Mm$ such that $\varepsilon_{Q}\pi=\varepsilon_{A}$. Since $\pi m_{A}(\mathrm{Id}_A\ot i)=\pi(\mathrm{Id}_A\ot\varepsilon_{K})$, we have 
    $$
    \pi i=\pi m_{A}(\mathrm{Id}_A\ot i) (u_A \ot \mathrm{Id}_K)=\pi(\mathrm{Id}_A\ot\varepsilon_{K}) (u_A \ot \mathrm{Id}_K)= \pi u_{A}\varepsilon_{K}. 
    $$
    Thus, we get
\[
\begin{split}
&\hspace{0.5cm} (\pi\ot\pi)\Delta_{A}m_{A}(\mathrm{Id}_{A}\ot i)=(\pi\ot\pi)(m_{A}\ot m_{A})(\mathrm{Id}_{A}\ot\sigma_{A,A}\ot\mathrm{Id}_{A})(\Delta_{A}\ot\Delta_{A})(\mathrm{Id}_{A}\ot i)\\&=(\mu_{Q}\ot\mu_{Q})(\mathrm{Id}_{A}\ot\pi\ot\mathrm{Id}_{A}\ot\pi)(\mathrm{Id}_{A}\ot\sigma_{A,A}\ot\mathrm{Id}_{A})(\Delta_{A}\ot\Delta_{A})(\mathrm{Id}_{A}\ot i)\\&=(\mu_{Q}\ot\mu_{Q})(\mathrm{Id}_{A}\ot\sigma_{A,Q}\ot\mathrm{Id}_{Q})(\mathrm{Id}_{A\ot A}\ot\pi\ot\pi)(\Delta_{A}\ot\Delta_{A})(\mathrm{Id}_{A}\ot i)\\&=(\mu_{Q}\ot\mu_{Q})(\mathrm{Id}_{A}\ot\sigma_{A,Q}\ot\mathrm{Id}_{Q})(\mathrm{Id}_{A\ot A}\ot\pi\ot\pi)(\Delta_{A}\ot(i\ot i)\Delta_{K})\\&=(\mu_{Q}\ot\mu_{Q})(\mathrm{Id}_{A}\ot\sigma_{A,Q}\ot\mathrm{Id}_{Q})(\Delta_{A}\ot(\pi u_{A}\varepsilon_{K}\ot\pi u_{A}\varepsilon_{K})\Delta_{K})\\&=(\mu_{Q}\ot\mu_{Q})(\mathrm{Id}_{A}\ot \pi u_{A}\ot\mathrm{Id}_{A}\ot \pi u_{A})(\Delta_{A}\ot\varepsilon_{K})\\&=(\pi m_{A}\ot\pi m_{A})(\mathrm{Id}_{A}\ot u_{A}\ot\mathrm{Id}_{A}\ot u_{A})(\Delta_{A}\ot\varepsilon_{K})\\&=(\pi\ot\pi)\Delta_{A}(\mathrm{Id}_{A}\ot\varepsilon_{K}).
\end{split}
\] 
Therefore, there exists a unique morphism $\Delta_{Q}:Q\to Q\ot Q$ in $\Mm$ such that $\Delta_{Q}\pi=(\pi\ot\pi)\Delta_{A}$. One can check that $(Q,\Delta_{Q},\varepsilon_{Q})$ is an object in $\mathsf{Comon}(\Mm)$ since $(A,\Delta_{A},\varepsilon_{A})$ is in $\mathsf{Comon}(\Mm)$ and $\pi$ is an epimorphism in $\Mm$, so that $\pi$ is a morphism in $\mathsf{Comon}(\Mm)$. Notice that the cocommutativity of $Q$ descends from the cocommutativity of $A$. To conclude that $\pi$ is in $\mathsf{Comon}_{\mathrm{coc}}(_{A}\Mm)$, it remains to prove that $Q$ is in $\mathsf{Comon}_{\mathrm{coc}}(_{A}\Mm)$, i.e.\ it remains to prove that $\Delta_{Q}$ and $\varepsilon_{Q}$ are morphisms in $_{A}\Mm$. Since $\mu_{Q\ot Q}=(\mu_{Q}\ot\mu_{Q})(\mathrm{Id}_{A}\ot\sigma_{A,Q}\ot\mathrm{Id}_{Q})(\Delta_{A}\ot\mathrm{Id}_{Q\ot Q})$, we can compute
\[
\begin{split}
\mu_{Q\ot Q}(\mathrm{Id}_{A}\ot\Delta_{Q})(\mathrm{Id}_{A}\ot\pi)&=(\mu_{Q}\ot\mu_{Q})(\mathrm{Id}_{A}\ot\sigma_{A,Q}\ot\mathrm{Id}_{Q})(\Delta_{A}\ot\Delta_{Q})(\mathrm{Id}_{A}\ot\pi)\\&=(\mu_{Q}\ot\mu_{Q})(\mathrm{Id}_{A}\ot\sigma_{A,Q}\ot\mathrm{Id}_{Q})(\Delta_{A}\ot(\pi\ot\pi)\Delta_{A})\\&=(\mu_{Q}\ot\mu_{Q})(\mathrm{Id}_{A}\ot\pi\ot\mathrm{Id}_{A}\ot\pi)(\mathrm{Id}\ot\sigma_{A,A}\ot\mathrm{Id}_{A})(\Delta_{A}\ot\Delta_{A})\\&=(\pi\ot\pi)(m_{A}\ot m_{A})(\mathrm{Id}_{A}\ot\sigma_{A,A}\ot\mathrm{Id}_{A})(\Delta_{A}\ot\Delta_{A})\\&=(\pi\ot\pi)\Delta_{A}m_{A}=\Delta_{Q}\pi m_{A}=\Delta_{Q}\mu_{Q}(\mathrm{Id}_{A}\ot\pi).
\end{split}
\]
Since $\mathrm{Id}_{A}\ot\pi$ is an epimorphism in $\Mm$, we obtain $\mu_{Q\ot Q}(\mathrm{Id}_{A}\ot\Delta_{Q})=\Delta_{Q}\mu_{Q}$. Furthermore, by considering the trivial left $A$-module structure $\mu_{\mathbf{1}}:= \varepsilon_{A} \ot \mathrm{Id}_{\mathbf{1}}:A \ot \mathbf{1} \to \mathbf{1}$ of the unit object $\mathbf{1}$, we have
\[
\mu_{\mathbf{1}}(\mathrm{Id}_{A}\ot\varepsilon_{Q})(\mathrm{Id}_{A}\ot\pi)=\mu_{\mathbf{1}} (\mathrm{Id}_{A}\ot \varepsilon_{A})=\varepsilon_{A}\ot\varepsilon_{A}=\varepsilon_{A}m_{A}=\varepsilon_{Q}\pi m_{A}=\varepsilon_{Q}\mu_{Q}(\mathrm{Id}_{A}\ot\pi).
\]
It follows that $\mu_{\mathbf{1}}(\mathrm{Id}_A \ot\varepsilon_{Q})=\varepsilon_{Q}\mu_{Q}$. Hence, $\pi$ is in $\mathsf{Comon}_{\mathrm{coc}}(_{A}\Mm)$. Since $\pi$ is an epimorphism in $\Mm$, it is also an epimorphism in $\mathsf{Comon}_{\mathrm{coc}}(_{A}\Mm)$.

2). We verify that $\pi$ is the coequalizer of the pair $((\varepsilon_{A}\ot\mathrm{Id}_{A})e_{A,A},(\mathrm{Id}_{A}\ot\varepsilon_{A})e_{A,A})$ in $\Mm$. Since $((\mathrm{Id}_{A}\ot\pi)\Delta_{A}\ot\mathrm{Id}_{A})e_{A,A}=(\mathrm{Id}_{A}\ot(\pi\ot\mathrm{Id}_{A})\Delta_{A})e_{A,A}$, we obtain 
\[
\begin{split}
\pi(\mathrm{Id}_{A}\ot\varepsilon_{A})e_{A,A}&=(\varepsilon_{A}\ot\mathrm{Id}_{Q}\ot\varepsilon_{A})((\mathrm{Id}_{A}\ot\pi)\Delta_{A}\ot\mathrm{Id}_{A})e_{A,A}\\&=(\varepsilon_{A}\ot\mathrm{Id}_{Q}\ot\varepsilon_{A})(\mathrm{Id}_{A}\ot(\pi\ot\mathrm{Id}_{A})\Delta_{A})e_{A,A}=\pi(\varepsilon_{A}\ot\mathrm{Id}_{A})e_{A,A},
\end{split}
\]
i.e.\ $\pi$ coequalizes the pair $((\varepsilon_{A}\ot\mathrm{Id}_{A})e_{A,A},(\mathrm{Id}_{A}\ot\varepsilon_{A})e_{A,A})$. Define the following morphism in $\Mm$
\[
\zeta:=(\mathrm{Id}_{A}\ot m_{A})(\Delta_{A}\ot i):A\ot K\to A\ot A.
\]
We compute
\begin{align*}
&(\mathrm{Id}_{A}\ot(\pi\ot\mathrm{Id}_{A})\Delta_{A})\zeta=(\mathrm{Id}_{A}\ot(\pi\ot\mathrm{Id}_{A})\Delta_{A})(\mathrm{Id}_{A}\ot m_{A})(\Delta_{A}\ot i)\\&=(\mathrm{Id}_{A}\ot\pi\ot\mathrm{Id}_{A})(\mathrm{Id}_{A}\ot m_{A}\ot m_{A})(\mathrm{Id}_{A\ot A}\ot\sigma_{A,A}\ot\mathrm{Id}_{A})(\mathrm{Id}_{A}\ot\Delta_{A}\ot\Delta_{A})(\Delta_{A}\ot i)\\&=(\mathrm{Id}_{A}\ot\pi\ot\mathrm{Id}_{A})(\mathrm{Id}_{A}\ot m_{A}\ot m_{A})(\mathrm{Id}_{A\ot A}\ot\sigma_{A,A}\ot\mathrm{Id}_{A})(\mathrm{Id}_{A}\ot\Delta_{A}\ot i\ot i)(\Delta_{A}\ot \Delta_{K})\\&=(\mathrm{Id}_{A}\ot\mu_{Q}\ot m_{A})(\mathrm{Id}_{A}\ot\mathrm{Id}_{A}\ot\pi\ot \mathrm{Id}_{A\ot A})(\mathrm{Id}_{A\ot A}\ot\sigma_{A,A}\ot\mathrm{Id}_{A})(\mathrm{Id}_{A}\ot\Delta_{A}\ot i\ot i)(\Delta_{A}\ot \Delta_{K})\\&=(\mathrm{Id}_{A}\ot\mu_{Q}\ot m_{A})(\mathrm{Id}_{A\ot A}\ot\sigma_{A,Q}\ot\mathrm{Id}_{A})(\mathrm{Id}_{A}\ot\mathrm{Id}_{A\ot A}\ot\pi\ot \mathrm{Id}_{A})(\mathrm{Id}_{A}\ot\Delta_{A}\ot i\ot i)(\Delta_{A}\ot \Delta_{K})\\&=(\mathrm{Id}_{A}\ot\mu_{Q}\ot m_{A})(\mathrm{Id}_{A\ot A}\ot\sigma_{A,Q}\ot\mathrm{Id}_{A})(\mathrm{Id}_{A}\ot\mathrm{Id}_{A\ot A}\ot\pi\ot \mathrm{Id}_{A})(\mathrm{Id}_{A}\ot\Delta_{A}\ot u_{A}\varepsilon_{K}\ot i)(\Delta_{A}\ot \Delta_{K})
\\&=(\mathrm{Id}_{A}\ot\mu_{Q}\ot m_{A})(\mathrm{Id}_{A\ot A}\ot\sigma_{A,Q}\ot\mathrm{Id}_{A})(\mathrm{Id}_{A}\ot\mathrm{Id}_{A\ot A}\ot\pi\ot \mathrm{Id}_{A})(\mathrm{Id}_{A}\ot\Delta_{A}\ot u_{A}\ot\mathrm{Id}_{A})(\Delta_{A}\ot i)\\&=(\mathrm{Id}_{A}\ot\mu_{Q}\ot m_{A})(\mathrm{Id}_{A\ot A}\ot\pi\ot \mathrm{Id}_{A\ot A})(\mathrm{Id}_{A\ot A}\ot\sigma_{A,A}\ot\mathrm{Id}_{A})(\Delta_{A}\ot\mathrm{Id}_{A}\ot u_{A}\ot\mathrm{Id}_{A})(\Delta_{A}\ot i)\\&=(\mathrm{Id}_{A}\ot\mu_{Q}\ot m_{A})(\mathrm{Id}_{A\ot A}\ot\pi\ot \mathrm{Id}_{A\ot A})(\Delta_{A}\ot u_{A}\ot\mathrm{Id}_{A\ot A})(\Delta_{A}\ot i)\\&=(\mathrm{Id}_{A}\ot\pi\ot m_{A})(\mathrm{Id}_{A}\ot m_{A}\ot \mathrm{Id}_{A\ot A})(\Delta_{A}\ot u_{A}\ot\mathrm{Id}_{A\ot A})(\Delta_{A}\ot i)\\&=(\mathrm{Id}_{A}\ot\pi\ot m_{A})(\Delta_{A}\ot\mathrm{Id}_{A\ot A})(\Delta_{A}\ot i)=((\mathrm{Id}_{A}\ot\pi)\Delta_{A}\ot\mathrm{Id}_{A})(\mathrm{Id}_{A}\ot m_{A})(\Delta_{A}\ot i)\\&=((\mathrm{Id}_{A}\ot\pi)\Delta_{A}\ot\mathrm{Id}_{A})\zeta.
\end{align*}
By the universal property, there exists a unique morphism $\xi:A\ot K\to A\square_{Q}A$ in $\Mm$ such that $e_{A,A}\xi=\zeta$. Suppose we have a morphism $f:A\to C$ in $\Mm$ such that $f(\varepsilon_{A}\ot\mathrm{Id}_{A})e_{A,A}=f(\mathrm{Id}_{A}\ot\varepsilon_{A})e_{A,A}$. We have
\[
\begin{split}
fm_{A}(\mathrm{Id}_{A}\ot i)&=f(\varepsilon_{A}\ot\mathrm{Id}_{A})(\mathrm{Id}_{A}\ot m_{A})(\Delta_{A}\ot i)=f(\varepsilon_{A}\ot\mathrm{Id}_{A})\zeta=f(\varepsilon_{A}\ot\mathrm{Id}_{A})e_{A,A}\xi\\&=f(\mathrm{Id}_{A}\ot\varepsilon_{A})e_{A,A}\xi=f(\mathrm{Id}_{A}\ot\varepsilon_{A})\zeta=f(\mathrm{Id}_{A}\ot\varepsilon_{A})(\mathrm{Id}_{A}\ot m_{A})(\Delta_{A}\ot i)\\&=f(\mathrm{Id}_{A}\ot\varepsilon_{A}\ot\varepsilon_{A})(\Delta_{A}\ot i)=f(\mathrm{Id}_{A}\ot\varepsilon_{K}).
\end{split}
\]
Since $\pi$ is the coequalizer of the pair $(m_{A}(\mathrm{Id}_{A}\ot i),\mathrm{Id}_{A}\ot\varepsilon_{K})$ in $\Mm$, there exists a unique morphism $p:Q\to C$ such that $p\pi=f$. Therefore, $\pi$ is the coequalizer of the pair $((\varepsilon_{A}\ot\mathrm{Id}_{A})e_{A,A},(\mathrm{Id}_{A}\ot\varepsilon_{A})e_{A,A})$ in $\Mm$.
\end{proof}

Proposition \ref{prop:definitionphi} maps a monomorphism $i:K \to A$ in $\Hopf$ to an epimorphism $\pi:A \to Q$ in $\mathsf{Comon}_{\mathrm{coc}}(_{A}\Mm)$. The next goal is to find the inverse assignment. By Corollary \ref{lemma:limitscocombim}, we know that $\mathsf{Comon}_{\mathrm{coc}}(\Mm)$ has equalizers. Thus, we can consider the equalizer of two particular morphisms in $\mathsf{Comon}_{\mathrm{coc}}(\Mm)$. As is usually done, we denote by $A^{\mathrm{co}Q}$ the equalizer of the pair $(\pi,\pi u_{A}\varepsilon_{A})$ in $\mathsf{Comon}_{\mathrm{coc}}(\Mm)$, which is the equalizer of the pair $((\pi\ot\mathrm{Id}_{A})\Delta_{A},\pi u_{A}\ot\mathrm{Id}_{A})$ in $\Mm$ (see Corollary \ref{lemma:limitscocombim}).   


\begin{proposition}\label{prop:definitionpsi}
    Let $A$ be an object in $\Hopf$ and $\pi:A\to Q$ be an epimorphism in $\mathsf{Comon}_{\mathrm{coc}}(_{A}\Mm)$, where $A$ is in $_{A}\Mm$ with action given by $m_{A}$. Let $i:A^{\mathrm{co}Q}\to A$ be the equalizer of the pair of morphisms $(\pi,\pi u_{A}\varepsilon_{A})$ in $\mathsf{Comon}_{\mathrm{coc}}(\Mm)$, i.e.\ the equalizer of the pair $((\pi\ot\mathrm{Id}_{A})\Delta_{A},\pi u_{A}\ot\mathrm{Id}_{A})$ in $\Mm$. Then, we have the following results.
\begin{itemize}
    \item[1)] The morphism $i$ is a monomorphism in $\Hopf$. 
    \item[2)] The following diagram is an equalizer in $\Mm$
\begin{equation}\label{iequalizer}
\begin{tikzcd}
	A^{\mathrm{co}Q} & A && A\ot_{A^{\mathrm{co}Q}}A
	\arrow[from=1-1, to=1-2, "i"]
	\arrow[shift left, from=1-2, to=1-4, "q_{A,A}(u_{A}\ot\mathrm{Id}_{A})"]
	\arrow[shift right, from=1-2, to=1-4, "q_{A,A}(\mathrm{Id}_{A}\ot u_{A})"']
\end{tikzcd}
\end{equation}
where $q_{A,A}:A\ot A\to A\ot_{A^{\mathrm{co}Q}}A$ is the coequalizer in $\Mm$ defined as in \eqref{def:balancedtensor} and $A$ is an object in $_{A^{\mathrm{co}Q}}\Mm$ with structure $m_{A}(i\ot\mathrm{Id}_{A})$ and in $\Mm_{A^{\mathrm{co}Q}}$ with structure $m_{A}(\mathrm{Id}_{A}\ot i)$. 
\end{itemize}
\end{proposition}

\begin{proof}
    1).  
    We already know that $i:A^{\mathrm{co}Q}\to A$ is in $\mathsf{Comon}_{\mathrm{coc}}(\Mm)$. It suffices to show that $A^{\mathrm{co}Q}$ is in $\mathsf{Mon}(\Mm)$ and it has an antipode. Consider the following diagram:
\[\begin{tikzcd}
	A^{\mathrm{co}Q}\ot A^{\mathrm{co}Q} &  A\ot A &&& A\ot Q\ot A \\
	A^{\mathrm{co}Q} & A &&& Q\ot A
	\arrow[from=1-1, to=1-2, "i\ot i"]
	\arrow[dashed, from=1-1, to=2-1,"m_{A^{\mathrm{co}Q}}"]
	\arrow[shift left, from=1-2, to=1-5,"\mathrm{Id}_{A}\ot(\pi\ot\mathrm{Id}_{A})\Delta_{A}"]
	\arrow[shift right, from=1-2, to=1-5, "\mathrm{Id}_{A}\ot\pi u_{A}\ot\mathrm{Id}_{A}"']
	\arrow[from=1-2, to=2-2,"m_A"]
	\arrow[from=1-5, to=2-5,"\mu_{Q\ot A}"]
	\arrow[from=2-1, to=2-2,"i"']
	\arrow[shift left, from=2-2, to=2-5,"(\pi\ot\mathrm{Id}_{A})\Delta_{A}"]
	\arrow[shift right, from=2-2, to=2-5,"\pi u_{A}\ot\mathrm{Id}_{A}"']
\end{tikzcd}\]
For the sake of brevity, we first compute
\[
\begin{split}
(\pi\ot\mathrm{Id}_{A})\Delta_{A}m_{A}&=(\pi\ot\mathrm{Id}_{A})(m_{A}\ot m_{A})(\mathrm{Id}_{A}\ot\sigma_{A,A}\ot\mathrm{Id}_{A})(\Delta_{A}\ot\Delta_{A})\\&=(\mu_{Q}\ot m_{A})(\mathrm{Id}_{A}\ot\pi\ot\mathrm{Id}_{A\ot A})(\mathrm{Id}_{A}\ot\sigma_{A,A}\ot\mathrm{Id}_{A})(\Delta_{A}\ot\Delta_{A})\\&=(\mu_{Q}\ot m_{A})(\mathrm{Id}_{A}\ot\sigma_{A,Q}\ot\mathrm{Id}_{A})(\Delta_{A}\ot\mathrm{Id}_{Q\ot A})
(\mathrm{Id}_{A}\ot(\pi\ot\mathrm{Id}_{A})\Delta_{A})\\&=\mu_{Q\ot A}(\mathrm{Id}_{A}\ot(\pi\ot\mathrm{Id}_{A})\Delta_{A}).
\end{split}
\]
Therefore,
\[
\begin{split}
(\pi\ot\mathrm{Id}_{A})&\Delta_{A}m_{A}(i\ot i)=\mu_{Q\ot A}(\mathrm{Id}_{A}\ot(\pi\ot\mathrm{Id}_{A})\Delta_{A})(i\ot i)=\mu_{Q\ot A}(i\ot(\pi i\ot i)\Delta_{A^{\mathrm{co}Q}})\\&=\mu_{Q\ot A}(i\ot(\pi u_{A}\varepsilon_{A}i\ot i)\Delta_{A^{\mathrm{co}Q}})=\mu_{Q\ot A}(i\ot(\pi u_{A}\varepsilon_{A}\ot \mathrm{Id}_{A})\Delta_{A}i)\\&=\mu_{Q\ot A}(i\ot\pi u_{A}\ot i)=(\mu_{Q}\ot m_{A})(\mathrm{Id}_{A}\ot\sigma_{A,Q}\ot\mathrm{Id}_{A})(\Delta_{A}\ot\mathrm{Id}_{Q\ot A})(i\ot\pi u_{A}\ot i)\\&=(\mu_{Q}\ot m_{A})(\mathrm{Id}_{A}\ot\sigma_{A,Q}\ot\mathrm{Id}_{A})(i\ot i\ot\pi u_{A}\ot i)(\Delta_{A^{\mathrm{co}Q}}\ot\mathrm{Id}_{A^{\mathrm{co}Q}})\\&=(\mu_{Q}\ot m_{A})(i\ot \pi u_{A}\ot i\ot i)(\Delta_{A^{\mathrm{co}Q}}\ot\mathrm{Id}_{A^{\mathrm{co}Q}})\\&=(\pi m_{A}(\mathrm{Id}_{A}\ot u_{A})\ot m_{A})
(i\ot i\ot i)(\Delta_{A^{\mathrm{co}Q}}\ot\mathrm{Id}_{A^{\mathrm{co}Q}})\\&=
(\pi i\ot m_{A}(i\ot i))(\Delta_{A^{\mathrm{co}Q}}\ot\mathrm{Id}_{A^{\mathrm{co}Q}})=
(\pi u_{A}\varepsilon_{A} i\ot m_{A}(i\ot i))(\Delta_{A^{\mathrm{co}Q}}\ot\mathrm{Id}_{A^{\mathrm{co}Q}})\\&=(\pi u_{A}\varepsilon_{A} \ot m_{A})(\Delta_{A}i\ot i)=(\pi u_{A}\ot\mathrm{Id}_{A})m_{A}(i\ot i).
\end{split}
\]
By Corollary \ref{lemma:limitscocombim}, we know that $i:A^{\mathrm{co}Q}\to A$ is the equalizer of the pair $((\pi\ot\mathrm{Id}_{A})\Delta_{A},\pi u_{A}\ot\mathrm{Id}_{A})$ in $\Mm$.
Hence, by the universal property, there exists a unique morphism $m_{A^{\mathrm{co}Q}}:A^{\mathrm{co}Q}\ot A^{\mathrm{co}Q}\to A^{\mathrm{co}Q}$ in $\Mm$ such that $m_{A}(i\ot i)=im_{A^{\mathrm{co}Q}}$. Moreover, since $(\pi\ot\mathrm{Id}_{A})\Delta_{A}u_{A}=(\pi u_{A}\ot\mathrm{Id}_{A})u_{A}$, there exists a unique morphism $u_{A^{\mathrm{co}Q}}:\mathbf{1}\to A^{\mathrm{co}Q}$ in $\Mm$ such that 
$iu_{A^{\mathrm{co}Q}}=u_{A}$. One can easily check that $(A^{\mathrm{co}Q},m_{A^{\mathrm{co}Q}},u_{A^{\mathrm{co}Q}})$ is in $\mathsf{Mon}(\Mm)$ since $(A,m_{A},u_{A})$ is in $\mathsf{Mon}(\Mm)$ and $i$ is a monomorphism in $\Mm$, so that $i:A^{\mathrm{co}Q}\to A$ becomes a morphism in $\mathsf{Mon}(\Mm)$. It follows that $(A^{\mathrm{co}Q},m_{A^{\mathrm{co}Q}},u_{A^{\mathrm{co}Q}},\Delta_{A^{\mathrm{co}Q}},\varepsilon_{A^{\mathrm{co}Q}})$ is in $\mathsf{Bimon}_{\mathrm{coc}}(\Mm)$ (and hence $i$ is a morphism in $\mathsf{Bimon}_{\mathrm{coc}}(\Mm)$) since $(A,m_{A},u_{A},\Delta_{A},\varepsilon_{A})$ is in $\mathsf{Bimon}(\Mm)$ and $i$ is a monomorphism in $\Mm$ which is preserved by $\ot$. Furthermore, we compute
\[
\begin{split}
    \pi i=\pi u_{A}\varepsilon_{A}i&=\pi m_{A}(S_{A}\ot\mathrm{Id}_{A})\Delta_{A}i=\mu_{Q}(\mathrm{Id}_{A}\ot\pi)(S_{A}\ot\mathrm{Id}_{A})\Delta_{A}i=\mu_{Q}(S_{A}\ot\mathrm{Id}_{Q})(\mathrm{Id}_{A}\ot\pi)\Delta_{A}i\\&=\mu_{Q}(S_{A}\ot\mathrm{Id}_{Q})(\mathrm{Id}_{A}\ot\pi u_{A})i=\mu_{Q}(\mathrm{Id}_{A}\ot\pi u_{A})S_{A}i=\pi m_{A}(\mathrm{Id}_{A}\ot u_{A})S_{A}i=\pi S_{A}i
\end{split}
\]
and, since $\Delta_{A}S_{A}=(S_{A}\ot S_{A})\Delta_{A}$ as $A$ is cocommutative, we also have 
\[
\begin{split}
(\pi\ot\mathrm{Id}_{A})\Delta_{A}S_{A}i&=(\pi\ot\mathrm{Id}_{A})(S_{A}\ot S_{A})\Delta_{A}i=(\pi S_{A}i\ot S_{A}i)\Delta_{A^{\mathrm{co}Q}}=(\pi i\ot S_{A}i)\Delta_{A^{\mathrm{co}Q}}\\&=(\pi u_{A}\varepsilon_{A}i\ot S_{A}i)\Delta_{A^{\mathrm{co}Q}}=(\pi u_{A}\varepsilon_{A}\ot S_{A})\Delta_{A}i=(\pi u_{A}\ot\mathrm{Id}_{A})S_{A}i.
\end{split}
\]
Hence, there exists a unique morphism $S_{A^{\mathrm{co}Q}}:A^{\mathrm{co}Q}\to A^{\mathrm{co}Q}$ in $\Mm$ such that $iS_{A^{\mathrm{co}Q}}=S_{A}i$. Since $S_{A}$ is the antipode of $A$ and $i$ is in $\mathsf{Bimon}_{\mathrm{coc}}(\Mm)$, we obtain that $S_{A^{\mathrm{co}Q}}$ is the antipode of $A^{\mathrm{co}Q}$. Therefore, the morphism $i$ is in $\Hopf$. Since $i$ is a monomorphism in $\Mm$, it is also a monomorphism in $\Hopf$. 

2). We now prove that $i$ is the equalizer of the pair $(q_{A,A}(u_{A}\ot\mathrm{Id}_{A}),q_{A,A}(\mathrm{Id}_{A}\ot u_{A}))$ in $\Mm$. We have that
\[
\begin{split}
q_{A,A}(u_{A}\ot \mathrm{Id}_{A})i&=q_{A,A}(\mathrm{Id}_{A}\ot m_{A}(i\ot\mathrm{Id}_{A}))(u_{A}\ot\mathrm{Id}_{K}\ot u_{A})\\&=q_{A,A}(m_{A}(\mathrm{Id}_{A}\ot i)\ot\mathrm{Id}_{A})(u_{A}\ot\mathrm{Id}_{K}\ot u_{A})=q_{A,A}(\mathrm{Id}_{A}\ot u_{A})i,
\end{split}
\]
i.e.\ $i$ equalizes the pair $(q_{A,A}(u_{A}\ot\mathrm{Id}_{A}),q_{A,A}(\mathrm{Id}_{A}\ot u_{A}))$. It remains to verify the universal property. By defining the following morphism in $\Mm$
\[
\zeta:=(\pi\ot m_{A})(\Delta_{A}\ot \mathrm{Id}_{A}):A\ot A\to Q\ot A,
\]
we compute
\begin{align*}
\zeta&(m_{A}(\mathrm{Id}_{A}\ot i)\ot\mathrm{Id}_{A})=(\pi\ot m_{A})(\Delta_{A}\ot\mathrm{Id}_{A})(m_{A}(\mathrm{Id}_{A}\ot i)\ot\mathrm{Id}_{A})\\&=(\pi\ot m_{A})(m_{A}\ot m_{A}\ot\mathrm{Id}_{A})(\mathrm{Id}_{A}\ot\sigma_{A,A}\ot\mathrm{Id}_{A\ot A})(\Delta_{A}\ot\Delta_{A}\ot\mathrm{Id}_{A})(\mathrm{Id}_{A}\ot i\ot\mathrm{Id}_{A})\\&=(\mu_{Q}\ot m_{A})(\mathrm{Id}_{A}\ot\pi\ot m_{A}\ot\mathrm{Id}_{A})(\mathrm{Id}_{A}\ot\sigma_{A,A}\ot\mathrm{Id}_{A\ot A})(\mathrm{Id}_{A\ot A}\ot i\ot i\ot\mathrm{Id}_{A})(\Delta_{A}\ot\Delta_{A^{\mathrm{co}Q}}\ot\mathrm{Id}_{A})\\&=(\mu_{Q}\ot m_{A})(\mathrm{Id}_{A}\ot\pi\ot m_{A}\ot\mathrm{Id}_{A})(\mathrm{Id}_{A}\ot i\ot\mathrm{Id}_{A}\ot i\ot\mathrm{Id}_{A})(\mathrm{Id}_{A}\ot\sigma_{A,A^{\mathrm{co}Q}}\ot\mathrm{Id}_{A^{\mathrm{co}Q}\ot A})\\&\hspace{0.5cm}(\Delta_{A}\ot\Delta_{A^{\mathrm{co}Q}}\ot\mathrm{Id}_{A})\\&=(\mu_{Q}\ot m_{A})(\mathrm{Id}_{A}\ot\pi u_{A}\varepsilon_{A}\ot m_{A}\ot\mathrm{Id}_{A})(\mathrm{Id}_{A}\ot i\ot\mathrm{Id}_{A}\ot i\ot\mathrm{Id}_{A})(\mathrm{Id}_{A}\ot\sigma_{A,A^{\mathrm{co}Q}}\ot\mathrm{Id}_{A^{\mathrm{co}Q}\ot A})\\&\hspace{0.5cm}(\Delta_{A}\ot\Delta_{A^{\mathrm{co}Q}}\ot\mathrm{Id}_{A})\\&=(\pi\ot m_{A})(m_{A}(\mathrm{Id}_{A}\ot u_{A}\varepsilon_{A})\ot m_{A}\ot\mathrm{Id}_{A})(\mathrm{Id}_{A}\ot i\ot\mathrm{Id}_{A}\ot i\ot\mathrm{Id}_{A})(\mathrm{Id}_{A}\ot\sigma_{A,A^{\mathrm{co}Q}}\ot\mathrm{Id}_{A^{\mathrm{co}Q}\ot A})\\&\hspace{0.5cm}(\Delta_{A}\ot\Delta_{A^{\mathrm{co}Q}}\ot\mathrm{Id}_{A})\\&=(\pi\ot m_{A})(\mathrm{Id}_{A}\ot \varepsilon_{A}\ot m_{A}\ot\mathrm{Id}_{A})(\mathrm{Id}_{A}\ot\sigma_{A,A}\ot\mathrm{Id}_{A\ot A})(\mathrm{Id}_{A\ot A}\ot i\ot i\ot\mathrm{Id}_{A})(\Delta_{A}\ot\Delta_{A^{\mathrm{co}Q}}\ot\mathrm{Id}_{A})\\&=(\pi\ot m_{A})(\mathrm{Id}_{A}\ot m_{A}\ot\mathrm{Id}_{A})(\mathrm{Id}_{A\ot A}\ot \varepsilon_{A}i\ot i\ot\mathrm{Id}_{A})(\Delta_{A}\ot\Delta_{A^{\mathrm{co}Q}}\ot\mathrm{Id}_{A})\\&=(\pi\ot m_{A})(\mathrm{Id}_{A}\ot m_{A}\ot\mathrm{Id}_{A})(\mathrm{Id}_{A\ot A}\ot \varepsilon_{A^{\mathrm{co}Q}}\ot i\ot\mathrm{Id}_{A})(\Delta_{A}\ot\Delta_{A^{\mathrm{co}Q}}\ot\mathrm{Id}_{A})\\&=(\pi\ot m_{A})(\mathrm{Id}_{A}\ot m_{A}\ot\mathrm{Id}_{A})(\Delta_{A}\ot i\ot\mathrm{Id}_{A})=(\pi\ot m_{A})(\mathrm{Id}_{A\ot A}\ot m_{A})(\Delta_{A}\ot i\ot\mathrm{Id}_{A})\\&=(\pi\ot m_{A})(\Delta_{A}\ot\mathrm{Id}_{A})(\mathrm{Id}_{A}\ot m_{A}(i\ot\mathrm{Id}_{A}))=\zeta(\mathrm{Id}_{A}\ot m_{A}(i\ot\mathrm{Id}_{A})).
\end{align*}
Therefore, there exists a unique morphism $\xi:A\ot_{A^{\mathrm{co}Q}}A\to Q\ot A$ in $\Mm$ such that $\xi q_{A,A}=\zeta$. Suppose that there is a morphism $f:C\to A$ in $\Mm$ such that $q_{A,A}(u_{A}\ot\mathrm{Id}_{A})f=q_{A,A}(\mathrm{Id}_{A}\ot u_{A})f$. We have
\[
\begin{split}
(\pi\ot\mathrm{Id}_{A})&\Delta_{A}f = (\pi\ot m_A) (\id_A \ot \id_A \ot u_A)\Delta_{A}f = (\pi\ot m_{A})(\Delta_{A}\ot\mathrm{Id}_{A})(\mathrm{Id}_{A}\ot u_{A})f\\
&=\zeta(\mathrm{Id}_{A}\ot u_{A})f=\xi q_{A,A}(\mathrm{Id}_{A}\ot u_{A})f=\xi q_{A,A}(u_{A}\ot\mathrm{Id}_{A})f=\zeta(u_{A}\ot\mathrm{Id}_{A})f\\
&=(\pi\ot m_{A})(\Delta_{A}\ot \mathrm{Id}_{A})(u_{A}\ot\mathrm{Id}_{A})f=(\pi\ot m_{A})(u_{A}\ot u_{A}\ot\mathrm{Id}_{A})f=(\pi u_A\ot\mathrm{Id}_{A})f.
\end{split}
\]
Since $i$ is the equalizer of the pair $((\pi\ot\mathrm{Id}_{A})\Delta_{A},\pi u_{A}\ot\mathrm{Id}_{A})$ in $\Mm$, there exists a unique morphism $p:C\to A^{\mathrm{co}Q}$ in $\Mm$ such that $ip=f$. Therefore, $i$ is the equalizer of the pair $(q_{A,A}(u_{A}\ot\mathrm{Id}_{A}),q_{A,A}(\mathrm{Id}_{A}\ot u_{A}))$ in $\Mm$.
\end{proof}

We now prove that Proposition \ref{prop:definitionphi} and Proposition \ref{prop:definitionpsi} define maps between subobjects of $A$ in $\Hopf$ and quotients of $A$ in $\mathsf{Comon}_{\mathrm{coc}}(_{A}\Mm)$. By a subobject of $A$ in $\Hopf$ we mean an isomorphism class of monomorphisms $i:K\to A$ in $\Hopf$ and by a quotient of $A$ in $\mathsf{Comon}_{\mathrm{coc}}(_{A}\Mm)$ we mean an isomorphism class of epimorphisms $\pi:A\to Q$ in $\mathsf{Comon}_{\mathrm{coc}}(_{A}\Mm)$. Sometimes we will write equalities between subobjects (resp.\ quotients) meaning that they are equal as classes, i.e.\ there is an isomorphism between their domains (resp.\ codomains).

\begin{proposition}\label{prop:definitionbijections}
Let $A$ be an object in $\Hopf$. 
\begin{itemize}
    \item[1)] There is a well-defined map $\phi_{A}$ from subobjects of $A$ in $\Hopf$ to quotients of $A$ in $\mathsf{Comon}_{\mathrm{coc}}(_{A}\Mm)$ defined by $\phi_{A}(i):=\pi$ as in Proposition \ref{prop:definitionphi}.
    \item[2)] There is a well-defined map $\psi_{A}$ from quotients of $A$ in $\mathsf{Comon}_{\mathrm{coc}}(_{A}\Mm)$ to subobjects of $A$ in $\Hopf$ defined by $\psi_{A}(\pi):=i$ as in Proposition \ref{prop:definitionpsi}.
\end{itemize}
\end{proposition}

\begin{proof}
1). For any monomorphism $i:K \to A$ in $\Hopf$ and isomorphism $j:K' \to K$ in $\Hopf$, denote by $\pi:A \to Q$ and $\pi':A \to Q'$ the coequalizers of $(m_{A}(\mathrm{Id}_{A}\ot i),\mathrm{Id}_{A}\ot\varepsilon_{K})$  and $(m_{A}(\mathrm{Id}_{A}\ot ij),\mathrm{Id}_{A}\ot\varepsilon_{K'})$ in $\Mm$, respectively, as in Proposition \ref{prop:definitionphi}. It suffices to show that there is an isomorphism $t: Q' \to Q$ in $\mathsf{Comon}_{\mathrm{coc}}(_{A}\Mm)$ such that $\pi = t \pi'$. Since $\pi m_{A}(\mathrm{Id}_{A}\ot ij)
= \pi(\mathrm{Id}_{A}\ot\varepsilon_{K}j)= \pi(\mathrm{Id}_{A}\ot\varepsilon_{K'})$, there is a unique morphism $t:Q' \to Q$ in $\Mm$ such that $\pi = t \pi'$. Since $j: K' \to K$ is an isomorphism in $\Hopf$, there is a morphism $j':K \to K'$ in $\Hopf$ such that $jj' = \mathrm{Id}_{K}$ and $j'j = \mathrm{Id}_{K'}$. Therefore, $\pi' m_{A}(\mathrm{Id}_{A}\ot i)= \pi' m_{A}(\mathrm{Id}_{A}\ot ijj')= \pi'(\mathrm{Id}_{A}\ot\varepsilon_{K'}j')= \pi'(\mathrm{Id}_{A}\ot\varepsilon_{K}).$
It follows that there is a unique morphism $t':Q \to Q'$ in $\Mm$ such that $\pi' = t' \pi$. As a result, $t:Q' \to Q$ is an isomorphism in $\Mm$ with inverse $t'$. One can easily verify that $t$ and $t'$ are indeed in $\mathsf{Comon}_{\mathrm{coc}}(_{A}\Mm)$. 
\begin{invisible}
It remains to show $t: Q' \to Q$ and $t':Q \to Q'$ are in $\mathsf{Comon}_{\mathrm{coc}}(_{A}\Mm)$. By Proposition \ref{prop:definitionphi}, $(t\ot t) \Delta_{Q'} \pi' = (t\ot t) (\pi' \ot \pi') \Delta_A = (\pi \ot \pi) \Delta_A = \Delta_Q \pi = \Delta_Q t \pi'$ and $\varepsilon_Q t \pi' = \varepsilon_Q \pi = \varepsilon_A = \varepsilon_{Q'} \pi'$. Since $\pi'$ is an epimorphism in $\Mm$, we obtain $(t\ot t) \Delta_{Q'} = \Delta_Q t$, and $\varepsilon_Q t = \varepsilon_{Q'}$. Furthermore, $\Delta_{Q'}t'= (t' \ot t')(t\ot t) \Delta_{Q'}t' = (t' \ot t')\Delta_Q t t' = (t' \ot t')\Delta_Q$, and $\varepsilon_{Q'} t' = \varepsilon_Q t t' =\varepsilon_Q$. Moreover, again by Proposition \ref{prop:definitionphi}, $\mu_Q (\mathrm{Id}_{A} \ot t)(\mathrm{Id}_{A} \ot \pi')= \mu_Q (\mathrm{Id}_{A} \ot \pi) = \pi m_A = t \pi' m_A = t \mu_{Q'} (\mathrm{Id}_{A} \ot \pi')$. Since $\mathrm{Id}_{A} \ot \pi'$ is an epimorphism in $\Mm$, we have $\mu_Q (\mathrm{Id}_{A} \ot t) = t \mu_{Q'}$. Moreover, $t' \mu_Q = t' \mu_Q (\mathrm{Id}_{A} \ot t) (\mathrm{Id}_{A} \ot t') = t' t \mu_{Q'} (\mathrm{Id}_{A} \ot t') = \mu_{Q'} (\mathrm{Id}_{A} \ot t')$. Therefore, $t$ is in $\mathsf{Comon}_{\mathrm{coc}}(_{A}\Mm)$. Similarly, $t'$ is in $\mathsf{Comon}_{\mathrm{coc}}(_{A}\Mm)$.
\end{invisible}

2). For any epimorphism $\pi:A\to Q$ in $\mathsf{Comon}_{\mathrm{coc}}(_{A}\Mm)$ and isomorphism $j:Q\to Q'$ in $\mathsf{Comon}_{\mathrm{coc}}(_{A}\Mm)$, denote by $i:A^{\mathrm{co}Q}\to A$ and $i':A^{\mathrm{co}Q'}\to A$ the equalizers of $(\pi,\pi u_{A}\varepsilon_{A})$ and $(j\pi,j\pi u_{A}\varepsilon_{A})$ in $\mathsf{Comon}_{\mathrm{coc}}(_{A}\Mm)$, respectively, as in Proposition \ref{prop:definitionpsi}. We show that there is an isomorphism $t: A^{\mathrm{co}Q} \to A^{\mathrm{co}Q'}$ in $\Hopf$ such that $i = i't$. Since $A^{\mathrm{co}Q}=\mathsf{Eq}((\pi\ot\mathrm{Id}_{A})\Delta_{A},\pi u_{A}\ot\mathrm{Id}_{A})$ and $A^{\mathrm{co}Q'}=\mathsf{Eq}(( j\pi\ot\mathrm{Id}_{A})\Delta_{A}, j\pi u_{A}\ot\mathrm{Id}_{A})$ in $\Mm$ and $(j \pi\ot\mathrm{Id}_{A})\Delta_A i 
= (j \pi u_A\ot\mathrm{Id}_{A} )i$, there is a unique morphism $t: A^{\mathrm{co}Q} \to A^{\mathrm{co}Q'}$ in $\Mm$ such that $i = i't$. Since $j: Q \to Q'$ is an isomorphism in $\mathsf{Comon}_{\mathrm{coc}}(_{A}\Mm)$, there is a morphism $j':Q' \to Q$ in $\mathsf{Comon}_{\mathrm{coc}}(_{A}\Mm)$ such that $jj' = \mathrm{Id}_{Q'}$ and $j'j = \mathrm{Id}_{Q}$. Hence,  $(\pi\ot\mathrm{Id}_{A})\Delta_A i' = (j'j\pi\ot\mathrm{Id}_{A})\Delta_A i' = (j'j \pi u_A\ot\mathrm{Id}_{A}) i' = (\pi u_A\ot\mathrm{Id}_{A}) i'$. Therefore, there is a unique morphism $t':A^{\mathrm{co}Q'} \to A^{\mathrm{co}Q}$ such that $i'=it'$. As a result, $t$ is an isomorphism in $\Mm$ with inverse $t'$. One can easily verify that $t$ and $t'$ are indeed in $\Hopf$.
\begin{invisible}
It remains to show $t$ and $t'$ are in $\Hopf$. By Proposition \ref{prop:definitionpsi}, we know $i m_{A^{\mathrm{co}Q}} = m_A (i \ot i)$ ,$i' m_{A^{\mathrm{co}Q'}} = m_A (i' \ot i')$, 
and $i u_{A^{\mathrm{co}Q}} = u_A = i' u_{A^{\mathrm{co}Q'}}
$. Hence, $i' t m_{A^{\mathrm{co}Q}} = i m_{A^{\mathrm{co}Q}} = m_A (i \ot i) = m_A (i' \ot i')(t \ot t) = i' m_{A^{\mathrm{co}Q'}} (t \ot t)$ and $i' t u_{A^{\mathrm{co}Q}} = i u_{A^{\mathrm{co}Q}} = u_A = i' u_{A^{\mathrm{co}Q'}}$. Since $i'$ is a monomorphism in $\Mm$, we have $t m_{A^{\mathrm{co}Q}} = m_{A^{\mathrm{co}Q'}} (t \ot t)$ and $t u_{A^{\mathrm{co}Q}} = u_{A^{\mathrm{co}Q'}}$. Consequently, $m_{A^{\mathrm{co}Q}}(t' \ot t') =t' t  m_{A^{\mathrm{co}Q}}(t' \ot t') = t' m_{A^{\mathrm{co}Q'}} (t \ot t)(t' \ot t') = t' m_{A^{\mathrm{co}Q'}}$ and $t' u_{A^{\mathrm{co}Q'}} = t' t u_{A^{\mathrm{co}Q}} = u_{A^{\mathrm{co}Q}}$. Thus, $t$ and $t'$ 
is in $\mathsf{Mon}(\Mm)$. By Proposition \ref{prop:definitionpsi} again, $(i'\ot i')(t \ot t)\Delta_{A^{\mathrm{co}Q}} = (i \ot i)\Delta_{A^{\mathrm{co}Q}} = \Delta_{A} i = \Delta_{A} i' t = (i'\ot i') \Delta_{A^{\mathrm{co}Q'}} t$. Since the tensor product is biexact, $i'\ot i'$ is a monomorphism in $\Mm$ and then $(t \ot t)\Delta_{A^{\mathrm{co}Q}}=\Delta_{A^{\mathrm{co}Q'}} t$. It follows that $\Delta_{A^{\mathrm{co}Q}} t' = (t' \ot t')(t \ot t)\Delta_{A^{\mathrm{co}Q}} t' = (t' \ot t')\Delta_{A^{\mathrm{co}Q'}} t t' = (t' \ot t')\Delta_{A^{\mathrm{co}Q'}}$. Besides, $\varepsilon_{A^{\mathrm{co}Q'}} t = \varepsilon_A i' t = \varepsilon_A i = \varepsilon_{A^{\mathrm{co}Q}}$. Thus, $t$ and $t'$ are in $\mathsf{Comon}(\Mm)$.
\end{invisible}
\end{proof}
Finally, we can provide a generalization of the bijective correspondence \eqref{bijectionNewman} known for cocommutative Hopf algebras (which goes back to \cite{Newman}), in the setting of $\Hopf$.

\begin{theorem}\label{thm:NewmanforM}
Let $A$ be an object in $\Hopf$. Then, there is a bijective correspondence between:
\begin{itemize}
    \item[1)] subobjects of $A$ in $\Hopf$ which are equalizers as in \eqref{iequalizer},
    \item[2)] quotients of $A$ in $\mathsf{Comon}_{\mathrm{coc}}(_{A}\Mm)$ which are coequalizers as in \eqref{picoequalizer}.
\end{itemize}
The mutually inverse bijections are given by $\phi_{A}$ and $\psi_{A}$ defined as in Proposition \ref{prop:definitionbijections}.
\end{theorem}

\begin{proof}
First, we show $\psi_{A}\phi_{A}= \id$. Given a subobject $i:K\to A$ of $A$ in $\Hopf$, define $\pi:=\phi_{A}(i):A\to Q$ in $\mathsf{Comon}_{\mathrm{coc}}(_{A}\Mm)$ and $i':=\psi_{A}(\pi):A^{\mathrm{co}Q}\to A$ in $\Hopf$. It suffices to show that $K$ is isomorphic to $A^{\mathrm{co}Q}$ in $\Mm$, so that they are isomorphic in $\mathsf{Hopf}_{\mathrm{coc}}(\Mm)$ since they are subobjects of $A$ in $\Hopf$. By definition of $\pi$, we know $\pi m_{A}(\mathrm{Id}_{A}\ot i)= \pi(\mathrm{Id}_{A}\ot\varepsilon_{K})$, thus we have 
\begin{align*}
(\mathrm{Id}_A\otimes \pi) \Delta_A i 
&= (\mathrm{Id}_A\otimes \pi) (\mathrm{Id}_A\otimes m_A(u_A\ot \mathrm{Id}_A))\Delta_A i\\
&= (\mathrm{Id}_A\otimes \pi) (\mathrm{Id}_A\otimes m_A)(\mathrm{Id}_A\otimes u_A\ot \mathrm{Id}_A) (i\ot i) \Delta_K\\
&= (\mathrm{Id}_A\otimes \pi) (\mathrm{Id}_A\otimes m_A)
(\mathrm{Id}_A\otimes \mathrm{Id}_A\otimes i)
(i\otimes u_A\ot \mathrm{Id}_K) \Delta_K\\
&= (\mathrm{Id}_A\otimes \pi) (\mathrm{Id}_A\otimes \mathrm{Id}_A\otimes \varepsilon_K)
(i\otimes u_A\ot \mathrm{Id}_K) \Delta_K\\
&= (\mathrm{Id}_A\otimes \pi) 
(i\otimes u_A) (\mathrm{Id}_K\otimes \varepsilon_K) \Delta_K = (\id_{A}\otimes \pi u_A)i.
\end{align*}
As a consequence, since $i':A^{\mathrm{co}Q}\to A$ is the equalizer of the pair of morphisms $((\mathrm{Id}_{A}\ot\pi)\Delta_{A},\mathrm{Id}_{A}\ot\pi u_{A})$ in $\Mm$, there exists a unique morphism $\varphi:K \to A^{\mathrm{co}Q}$ in $\Mm$ such that 
$i'\varphi=i$. By assumption $i$ is the equalizer of the pair of morphisms $(q_{A,A}(u_{A}\ot\mathrm{Id}_{A}),q_{A,A}(\mathrm{Id}_{A}\ot u_{A}))$ in $\Mm$. 
We define the following morphism in $\Mm$
\[
\zeta:=(\mathrm{Id}_{Q}\ot m_{A})((\pi\ot\mathrm{Id}_{A})\Delta_{A}\ot\mathrm{Id}_{A}):A\ot A\to Q\ot A.
\]
Because
\[
\begin{split}
    \zeta&(m_{A}(\mathrm{Id}_{A}\ot i)\ot\mathrm{Id}_{A})=(\mathrm{Id}_{Q}\ot m_{A})((\pi\ot\mathrm{Id}_{A})\Delta_{A}\ot\mathrm{Id}_{A})(m_{A}(\mathrm{Id}_{A}\ot i)\ot\mathrm{Id}_{A})\\&=(\mathrm{Id}_{Q}\ot m_{A})((\pi\ot\mathrm{Id}_{A})(m_{A}\ot m_{A})(\mathrm{Id}_{A}\ot\sigma_{A,A}\ot\mathrm{Id}_{A})(\Delta_{A}\ot\Delta_{A})(\mathrm{Id}_{A}\ot i)\ot\mathrm{Id}_{A})\\&=(\mathrm{Id}_{Q}\ot m_{A})((\pi\ot\mathrm{Id}_{A})(m_{A}\ot m_{A})(\mathrm{Id}_{A}\ot\sigma_{A,A}\ot\mathrm{Id}_{A})(\mathrm{Id}_{A\ot A}\ot i\ot i)(\Delta_{A}\ot\Delta_{K})\ot\mathrm{Id}_{A})\\&=(\mathrm{Id}_{Q}\ot m_{A})((\pi\ot\mathrm{Id}_{A})(m_{A}\ot m_{A})(\mathrm{Id}_{A}\ot i\ot\mathrm{Id}_{A}\ot i)(\mathrm{Id}_{A}\ot\sigma_{A,K}\ot\mathrm{Id}_{K})(\Delta_{A}\ot\Delta_{K})\ot\mathrm{Id}_{A})\\&=(\mathrm{Id}_{Q}\ot m_{A})((\pi (\mathrm{Id}_{A}\ot \varepsilon_{K})\ot m_{A}(\mathrm{Id}_{A}\ot i))(\mathrm{Id}_{A}\ot\sigma_{A,K}\ot\mathrm{Id}_{K})(\Delta_{A}\ot\Delta_{K})\ot\mathrm{Id}_{A})\\&=(\mathrm{Id}_{Q}\ot m_{A})((\pi \ot m_{A}(\mathrm{Id}_{A}\ot i))(\Delta_{A}\ot\mathrm{Id}_{K})\ot\mathrm{Id}_{A})\\&=(\mathrm{Id}_{Q}\ot m_{A})(\mathrm{Id}_{Q}\ot\mathrm{Id}_{A}\ot m_{A})((\pi\ot\mathrm{Id}_{A})\Delta_{A}\ot i\ot\mathrm{Id}_{A})\\&=(\mathrm{Id}_{Q}\ot m_{A})((\pi\ot\mathrm{Id}_{A})\Delta_{A}\ot\mathrm{Id}_{A})(\mathrm{Id}_{A}\ot m_{A}(i\ot\mathrm{Id}_{A}))\\&=\zeta(\mathrm{Id}_{A}\ot m_{A}(i\ot\mathrm{Id}_{A})),
\end{split}
\]
by the universal property of the coequalizer $q_{A,A}:A\ot A\to A\ot_{K}A$, there is a unique morphism $\overline{\mathsf{can}}:A\ot_{K}A\to Q\ot A$ in $\Mm$ such that $\overline{\mathsf{can}}q_{A,A}=\zeta$. We compute
\begin{eqnarray*}
\overline{\mathsf{can}}q_{A,A}(u_{A}\ot\mathrm{Id}_{A})=\zeta(u_{A}\ot\mathrm{Id}_{A})=(\mathrm{Id}_{Q}\ot m_{A})((\pi\ot\mathrm{Id}_{A})\Delta_{A}\ot\mathrm{Id}_{A})(u_{A}\ot\mathrm{Id}_{A})=\pi u_{A}\ot\mathrm{Id}_{A}
\end{eqnarray*}
and
\begin{eqnarray*}
\overline{\mathsf{can}}q_{A,A}(\mathrm{Id}_{A}\ot u_{A})=\zeta(\mathrm{Id}_{A}\ot u_{A})=(\mathrm{Id}_{Q}\ot m_{A})((\pi\ot\mathrm{Id}_{A})\Delta_{A}\ot\mathrm{Id}_{A})(\mathrm{Id}_{A}\ot u_{A})=(\pi\ot\mathrm{Id}_{A})\Delta_{A}.
\end{eqnarray*}
Hence, the following diagram in $\Mm$ commutes:
\begin{equation}\label{commutativediagram}
\begin{tikzcd}
	K & A && A\ot_{K}A \\
	A^{\mathrm{co}Q} & A && Q\ot A
	\arrow[from=1-1, to=1-2, "i"]
	\arrow[from=1-1, to=2-1, "\varphi"']
	\arrow[shift left, from=1-2, to=1-4, "q_{A,A}(u_{A}\ot\mathrm{Id}_{A})"]
	\arrow[shift right, from=1-2, to=1-4, "q_{A,A}(\mathrm{Id}_{A}\ot u_{A})"']
	\arrow[from=1-2, to=2-2,"\mathrm{Id}_{A}"']
	\arrow[from=1-4, to=2-4, "\overline{\mathsf{can}}"]
	\arrow[from=2-1, to=2-2, "i'"']
	\arrow[shift left, from=2-2, to=2-4, "\pi u_{A}\ot\mathrm{Id}_{A}"]
	\arrow[shift right, from=2-2, to=2-4, "(\pi\ot\mathrm{Id}_{A})\Delta_{A}"']
\end{tikzcd}
\end{equation}
Observe that $\zeta=(\pi\ot\mathrm{Id}_{A})\mathsf{can}$, where $\mathsf{can}:=(\mathrm{Id}_{A}\ot m_{A})(\Delta_{A}\ot\mathrm{Id}_{A})$. Note that $\mathsf{can}$ is an isomorphism in $\Mm$ since $A$ is in $\mathsf{Hopf}(\Mm)$ and its inverse is given by $\mathsf{can}^{-1}=(\mathrm{Id}_{A}\ot m_{A})(\mathrm{Id}_{A}\ot S_{A}\ot\mathrm{Id}_{A})(\Delta_{A}\ot\mathrm{Id}_{A})$, see e.g.\ \cite[Theorem 1.8]{Vercruysse}. Since $\pi$ is the coequalizer of the pair of morphisms $(m_{A}(\mathrm{Id}_{A}\ot i),\mathrm{Id}_{A}\ot\varepsilon_{K})$ in $\Mm$ and $(-)\ot A$ preserves coequalizers, the morphism $\pi\ot\mathrm{Id}_{A}$ is the coequalizer of the pair of morphisms $(m_{A}(\mathrm{Id}_{A}\ot i)\ot\mathrm{Id}_{A},\mathrm{Id}_{A}\ot\varepsilon_{K}\ot\mathrm{Id}_{A})$ in $\Mm$. Then, $\zeta=(\pi\ot\mathrm{Id}_{A})\mathsf{can}$ is the coequalizer of the pair $(\mathsf{can}^{-1}(m_{A}(\mathrm{Id}_{A}\ot i)\ot\mathrm{Id}_{A}),\mathsf{can}^{-1}(\mathrm{Id}_{A}\ot\varepsilon_{K}\ot\mathrm{Id}_{A}))$ in $\Mm$. Since
\begin{align*}
    &\mathsf{can}^{-1}(m_{A}(\mathrm{Id}_{A}\ot i)\ot\mathrm{Id}_{A})=(\mathrm{Id}_{A}\ot m_{A})(\mathrm{Id}_{A}\ot S_{A}\ot\mathrm{Id}_{A})(\Delta_{A}\ot\mathrm{Id}_{A})(m_{A}(\mathrm{Id}_{A}\ot i)\ot\mathrm{Id}_{A})\\&=(\mathrm{Id}_{A}\ot m_{A})(\mathrm{Id}_{A}\ot S_{A}\ot\mathrm{Id}_{A})((m_{A}\ot m_{A})(\mathrm{Id}_{A}\ot\sigma_{A,A}\ot\mathrm{Id}_{A})(\Delta_{A}\ot\Delta_{A})(\mathrm{Id}_{A}\ot i)\ot\mathrm{Id}_{A})\\&=(\mathrm{Id}_{A}\ot m_{A})(\mathrm{Id}_{A}\ot S_{A}\ot\mathrm{Id}_{A})((m_{A}\ot m_{A})(\mathrm{Id}_{A}\ot\sigma_{A,A}\ot\mathrm{Id}_{A})(\mathrm{Id}_{A\ot A}\ot i\ot i)(\Delta_{A}\ot\Delta_{K})\ot\mathrm{Id}_{A})\\&=(\mathrm{Id}_{A}\ot m_{A})(\mathrm{Id}_{A}\ot S_{A}\ot\mathrm{Id}_{A})((m_{A}\ot m_{A})(\mathrm{Id}_{A}\ot i\ot\mathrm{Id}_{A}\ot i)(\mathrm{Id}_{A}\ot\sigma_{A,K}\ot\mathrm{Id}_{K})(\Delta_{A}\ot\Delta_{K})\ot\mathrm{Id}_{A})\\&=(m_{A}(\mathrm{Id}_{A}\ot i)\ot \mathrm{Id}_{A})(\mathrm{Id}_{A\ot K}\ot m_{A})((\mathrm{Id}_{A\ot K}\ot S_{A}m_{A}(\mathrm{Id}_{A}\ot i))(\mathrm{Id}_{A}\ot\sigma_{A,K}\ot\mathrm{Id}_{K})(\Delta_{A}\ot\Delta_{K})\ot\mathrm{Id}_{A})
\end{align*}  
and 
\begin{align*}
    &\mathsf{can}^{-1}(\mathrm{Id}_{A}\ot\varepsilon_{K}\ot\mathrm{Id}_{A})=(\mathrm{Id}_{A}\ot m_{A})(\mathrm{Id}_{A}\ot S_{A}\ot\mathrm{Id}_{A})(\Delta_{A}\ot\mathrm{Id}_{A})(\mathrm{Id}_{A}\ot\varepsilon_{K}\ot\mathrm{Id}_{A})\\
    &=
    (\mathrm{Id}_{A}\ot m_{A})(\mathrm{Id}_{A}\ot u_{A}\ot m_{A})((\mathrm{Id}_{A}\ot S_{A})\Delta_{A}\ot\varepsilon_{K}\ot\mathrm{Id}_{A})\\
    &=(\mathrm{Id}_{A}\ot m_{A})(\mathrm{Id}_{A}\ot u_{A}\ot m_{A})(\id_A \ot \varepsilon_K \ot \id_{A\ot A})((\mathrm{Id}_{A\ot K}\ot S_{A})(\mathrm{Id}_{A}\ot\sigma_{A,K})(\Delta_{A}\ot\mathrm{Id}_{K})\ot\mathrm{Id}_{A})\\
    &=(\mathrm{Id}_{A}\ot m_{A})(\mathrm{Id}_{A}\ot i\ot\mathrm{Id}_{A})(\mathrm{Id}_{A}\ot u_{K}\varepsilon_{K}\ot m_{A})((\mathrm{Id}_{A\ot K}\ot S_{A})(\mathrm{Id}_{A}\ot\sigma_{A,K})(\Delta_{A}\ot\mathrm{Id}_{K})\ot\mathrm{Id}_{A})\\&=(\mathrm{Id}_{A}\ot m_{A})(\mathrm{Id}_{A}\ot i\ot\mathrm{Id}_{A})(\mathrm{Id}_{A}\ot m_{K}(\mathrm{Id}_{K}\ot S_{K})\Delta_{K}\ot m_{A})\\&\hspace{0.5cm}((\mathrm{Id}_{A\ot K}\ot S_{A})(\mathrm{Id}_{A}\ot\sigma_{A,K})(\Delta_{A}\ot\mathrm{Id}_{K})\ot\mathrm{Id}_{A})\\
    &=(\mathrm{Id}_{A}\ot m_{A})(\mathrm{Id}_{A}\ot i\ot\mathrm{Id}_{A})(\mathrm{Id}_{A}\ot m_{K}\ot m_{A})\\&\hspace{0.5cm}((\mathrm{Id}_{A\ot K}\ot S_{K}\ot S_{A})(\mathrm{Id}_{A}\ot\Delta_{K}\ot\mathrm{Id}_{A})(\mathrm{Id}_{A}\ot\sigma_{A,K})(\Delta_{A}\ot\mathrm{Id}_{K})\ot\mathrm{Id}_{A})\\&=(\mathrm{Id}_{A}\ot m_{A})(\mathrm{Id}_{A}\ot i\ot\mathrm{Id}_{A})(\mathrm{Id}_{A}\ot m_{K}\ot m_{A})\\&\hspace{0.5cm}((\mathrm{Id}_{A\ot K}\ot S_{K}\ot S_{A})(\mathrm{Id}_{A}\ot\sigma_{A,K\ot K})(\Delta_{A}\ot\Delta_{K})\ot\mathrm{Id}_{A})\\&=(\mathrm{Id}_{A}\ot m_{A})(\mathrm{Id}_{A}\ot i\ot\mathrm{Id}_{A})(\mathrm{Id}_{A}\ot m_{K}\ot m_{A})\\&\hspace{0.5cm}((\mathrm{Id}_{A\ot K}\ot(S_{K}\ot S_{A}) \sigma_{A,K})(\mathrm{Id}_{A}\ot\sigma_{A,K}\ot\mathrm{Id}_{K})(\Delta_{A}\ot\Delta_{K})\ot\mathrm{Id}_{A})\\&=(\mathrm{Id}_{A}\ot m_{A})(\mathrm{Id}_{A}\ot m_{A}\ot\mathrm{Id}_{A})(\mathrm{Id}_{A}\ot i\ot i\ot \mathrm{Id}_{A})(\mathrm{Id}_{A\ot K\ot K}\ot m_{A})\\&\hspace{0.5cm}((\mathrm{Id}_{A\ot K}\ot \sigma_{A,K}(S_{A}\ot S_{K}))(\mathrm{Id}_{A}\ot\sigma_{A,K}\ot\mathrm{Id}_{K})(\Delta_{A}\ot\Delta_{K})\ot\mathrm{Id}_{A})\\&=(\mathrm{Id}_{A}\ot m_{A})(\mathrm{Id}_{A}\ot i\ot m_{A})(\mathrm{Id}_{A\ot K}\ot i\ot m_{A})\\&\hspace{0.5cm}((\mathrm{Id}_{A\ot K}\ot \sigma_{A,K}(S_{A}\ot S_{K}))(\mathrm{Id}_{A}\ot\sigma_{A,K}\ot\mathrm{Id}_{K})(\Delta_{A}\ot\Delta_{K})\ot\mathrm{Id}_{A})\\&=(\mathrm{Id}_{A}\ot m_{A}(i\ot\mathrm{Id}_{A}))(\mathrm{Id}_{A\ot K}\ot m_{A})(\mathrm{Id}_{A\ot K}\ot i\ot m_{A})\\&\hspace{0.5cm}((\mathrm{Id}_{A\ot K}\ot \sigma_{A,K}(S_{A}\ot S_{K}))(\mathrm{Id}_{A}\ot\sigma_{A,K}\ot\mathrm{Id}_{K})(\Delta_{A}\ot\Delta_{K})\ot\mathrm{Id}_{A})\\&
    =(\mathrm{Id}_{A}\ot m_{A}(i\ot\mathrm{Id}_{A}))(\mathrm{Id}_{A\ot K}\ot m_{A})(\mathrm{Id}_{A\ot K}\ot m_{A}(i\ot\mathrm{Id}_{A})\ot\mathrm{Id}_{A})\\&\hspace{0.5cm}((\mathrm{Id}_{A\ot K}\ot \sigma_{A,K}(S_{A}\ot S_{K}))(\mathrm{Id}_{A}\ot\sigma_{A,K}\ot\mathrm{Id}_{K})(\Delta_{A}\ot\Delta_{K})\ot\mathrm{Id}_{A})\\&=(\mathrm{Id}_{A}\ot m_{A}(i\ot\mathrm{Id}_{A}))(\mathrm{Id}_{A\ot K}\ot m_{A})(\mathrm{Id}_{A\ot K}\ot m_{A}\sigma_{A,A}(S_{A}\ot S_{A}i)\ot\mathrm{Id}_{A})\\&\hspace{0.5cm}((\mathrm{Id}_{A}\ot\sigma_{A,K}\ot\mathrm{Id}_{K})(\Delta_{A}\ot\Delta_{K})\ot\mathrm{Id}_{A})\\&=(\mathrm{Id}_{A}\ot m_{A}(i\ot\mathrm{Id}_{A}))(\mathrm{Id}_{A\ot K}\ot m_{A})(\mathrm{Id}_{A\ot K}\ot S_{A}m_{A}(\id_{A}\ot i)\ot\mathrm{Id}_{A})\\&\hspace{0.5cm}((\mathrm{Id}_{A}\ot\sigma_{A,K}\ot\mathrm{Id}_{K})(\Delta_{A}\ot\Delta_{K})\ot\mathrm{Id}_{A})
\end{align*}
we get
\begin{align*}
    &q_{A,A}\mathsf{can}^{-1}(m_{A}(\mathrm{Id}_{A}\ot i)\ot\mathrm{Id}_{A})
    \\&=q_{A,A}(m_{A}(\mathrm{Id}_{A}\ot i)\ot \mathrm{Id}_{A})(\mathrm{Id}_{A\ot K}\ot m_{A})(\mathrm{Id}_{A\ot K}\ot S_{A}m_{A}(\mathrm{Id}_{A}\ot i)\ot\mathrm{Id}_{A})\\
    &\hspace{0.5cm}((\mathrm{Id}_{A}\ot\sigma_{A,K}\ot\mathrm{Id}_{K})(\Delta_{A}\ot\Delta_{K})\ot\mathrm{Id}_{A})\\&=q_{A,A}(\mathrm{Id}_{A}\ot m_{A}(i\ot\mathrm{Id}_{A}))(\mathrm{Id}_{A\ot K}\ot m_{A})(\mathrm{Id}_{A\ot K}\ot S_{A}m_{A}(\id_{A}\ot i)\ot\mathrm{Id}_{A})\\&\hspace{0.5cm}((\mathrm{Id}_{A}\ot\sigma_{A,K}\ot\mathrm{Id}_{K})(\Delta_{A}\ot\Delta_{K})\ot\mathrm{Id}_{A})\\&=q_{A,A}\mathsf{can}^{-1}(\mathrm{Id}_{A}\ot\varepsilon_{K}\ot\mathrm{Id}_{A}).
\end{align*}
Hence, there exists a unique morphism $\overline{\mathsf{can}}^{-1}:Q\ot A\to A\ot_{K}A$ in $\Mm$ such that $\overline{\mathsf{can}}^{-1}\zeta=q_{A,A}$. We show that $\overline{\mathsf{can}}^{-1}$ is indeed the inverse of $\overline{\mathsf{can}}$. Since $\mathsf{can}$ is an isomorphism in $\Mm$ and $\pi\ot\mathrm{Id}_{A}$ is an epimorphism in $\Mm$, we have $\zeta=(\pi\ot\mathrm{Id}_{A})\mathsf{can}$ is an epimorphism in $\Mm$. Note that $q_{A,A}$ is also an epimorphism in $\Mm$. Thus, the equations
$\overline{\mathsf{can}}\ \overline{\mathsf{can}}^{-1}\zeta=\overline{\mathsf{can}}q_{A,A}=\zeta$ and $\overline{\mathsf{can}}^{-1}\ \overline{\mathsf{can}}q_{A,A}=\overline{\mathsf{can}}^{-1}\zeta=q_{A,A}$ imply that $\overline{\mathsf{can}}^{-1}$ is the inverse of $\overline{\mathsf{can}}$. Since both lines in the diagram \eqref{commutativediagram} are equalizers in $\Mm$ and $\overline{\mathsf{can}}$ is an isomorphism in $\Mm$, it follows that $\varphi$ is an isomorphism in $\Mm$ as well. \medskip


Now, we show that $\phi_{A}\psi_{A}=\id$. Given a quotient $\pi:A\to Q$ in $\mathsf{Comon}_{\mathrm{coc}}(_{A}\Mm)$, let $i:=\psi_{A}(\pi):A^{\mathrm{co}Q}\to A$ in $\Hopf$ and $\pi':=\phi_{A}(i):A\to Q'$ in $\mathsf{Comon}_{\mathrm{coc}}(_{A}\Mm)$. We prove that $Q$ is isomorphic to $Q'$ in $\Mm$, so that $Q$ and $Q'$ are isomorphic in $\mathsf{Comon}_{\mathrm{coc}}(_{A}\Mm)$ since they are quotients of $A$ in $\mathsf{Comon}_{\mathrm{coc}}(_{A}\Mm)$. Recall that $i$ is the equalizer of the pair $(\pi,\pi u_{A}\varepsilon_{A})$ in $\mathsf{Comon}_{\mathrm{coc}}(\Mm)$ and $\pi':A\to Q'$ is the coequalizer of the pair $(m_{A}(\mathrm{Id}_{A}\ot i),\mathrm{Id}_{A}\ot\varepsilon_{A^{\mathrm{co}Q}})$ in $\Mm$. We compute 
\[
\pi m_{A}(\mathrm{Id}_{A}\ot i)=\mu_{Q}(\mathrm{Id}_{A}\ot\pi i)=\mu_{Q}(\mathrm{Id}_{A}\ot\pi u_{A}\varepsilon_{A}i)=\pi m_{A}(\mathrm{Id}_{A}\ot u_{A}\varepsilon_{A^{\mathrm{co}Q}})=\pi(\mathrm{Id}_{A}\ot\varepsilon_{A^{\mathrm{co}Q}}).
\]
Thus, there exists a unique morphism $\varphi:Q'\to Q$ in $\Mm$ such that  $\varphi\pi'=\pi$. As said before, since $A$ is in $\mathsf{Hopf}(\Mm)$, $\mathsf{can}:=(\mathrm{Id}_{A}\ot m_{A})(\Delta_{A}\ot\mathrm{Id}_{A})$ is an isomorphism in $\Mm$, see e.g.\ \cite[Theorem 1.8]{Vercruysse}. Consider the following morphism in $\Mm$:
\[
f:=(\mathrm{Id}_{A}\ot\mu_{Q}\ot m_{A})(\mathrm{Id}_{A\ot A}\ot \sigma_{A,Q}\ot\mathrm{Id}_{A})(\mathrm{Id}_{A}\ot\Delta_{A}\ot\mathrm{Id}_{Q\ot A})(\Delta_{A}\ot\mathrm{Id}_{Q\ot A}):A\ot Q\ot A\to A\ot Q\ot A.
\]
We compute
\begin{align*}
&f(\mathrm{Id}_{A}\ot\pi u_{A}\ot\mathrm{Id}_{A})\\
&=(\mathrm{Id}_{A}\ot\mu_{Q}\ot m_{A})(\mathrm{Id}_{A\ot A}\ot \sigma_{A,Q}\ot\mathrm{Id}_{A})((\Delta_{A}\ot\mathrm{Id}_{A})\Delta_{A}\ot\mathrm{Id}_{Q\ot A})(\mathrm{Id}_{A}\ot\pi u_{A}\ot\mathrm{Id}_{A})\\&=(\mathrm{Id}_{A}\ot\mu_{Q}\ot m_{A})(\mathrm{Id}_{A\ot A}\ot \sigma_{A,Q}\ot\mathrm{Id}_{A})(\mathrm{Id}_{A\ot A\ot A}\ot\pi u_{A}\ot\mathrm{Id}_{A})((\Delta_{A}\ot\mathrm{Id}_{A})\Delta_{A}\ot\mathrm{Id}_{A})\\&=(\mathrm{Id}_{A}\ot\mu_{Q}\ot m_{A})(\mathrm{Id}_{A\ot A}\ot\pi u_{A}\ot\mathrm{Id}_{ A\ot A})((\Delta_{A}\ot\mathrm{Id}_{A})\Delta_{A}\ot\mathrm{Id}_{A})\\&=(\mathrm{Id}_{A}\ot\mu_{Q}(\mathrm{Id}_{A}\ot\pi u_{A})\ot m_{A})((\Delta_{A}\ot\mathrm{Id}_{A})\Delta_{A}\ot\mathrm{Id}_{A})\\&=(\mathrm{Id}_{A}\ot \pi m_{A}(\mathrm{Id}_{A}\ot u_{A})\ot m_{A})((\Delta_{A}\ot\mathrm{Id}_{A})\Delta_{A}\ot\mathrm{Id}_{A})\\&=(\mathrm{Id}_{A}\ot \pi \ot m_{A})((\Delta_{A}\ot\mathrm{Id}_{A})\Delta_{A}\ot\mathrm{Id}_{A})\\&=((\mathrm{Id}_{A}\ot\pi)\Delta_{A}\ot\mathrm{Id}_{A})(\mathrm{Id}_{A}\ot m_{A})(\Delta_{A}\ot\mathrm{Id}_{A})\\&=((\mathrm{Id}_{A}\ot\pi)\Delta_{A}\ot\mathrm{Id}_{A})\mathsf{can}
\end{align*}
and
\begin{align*}
&f(\mathrm{Id}_{A}\ot(\pi\ot\mathrm{Id}_{A})\Delta_{A})\\&=(\mathrm{Id}_{A}\ot\mu_{Q}\ot m_{A})(\mathrm{Id}_{A\ot A}\ot \sigma_{A,Q}\ot\mathrm{Id}_{A})((\mathrm{Id}_{A}\ot\Delta_{A})\Delta_{A}\ot\mathrm{Id}_{Q\ot A})(\mathrm{Id}_{A}\ot(\pi\ot\mathrm{Id}_{A})\Delta_{A})\\&=(\mathrm{Id}_{A}\ot\mu_{Q}\ot m_{A})(\mathrm{Id}_{A\ot A}\ot \sigma_{A,Q}\ot\mathrm{Id}_{A})(\mathrm{Id}_{A\ot A\ot A}\ot\pi\ot\mathrm{Id}_{A})((\mathrm{Id}_{A}\ot\Delta_{A})\Delta_{A}\ot\Delta_{A})\\&=(\mathrm{Id}_{A}\ot\mu_{Q}\ot m_{A})(\mathrm{Id}_{A\ot A}\ot\pi\ot\mathrm{Id}_{A\ot A})(\mathrm{Id}_{A\ot A}\ot \sigma_{A,A}\ot\mathrm{Id}_{A})((\mathrm{Id}_{A}\ot\Delta_{A})\Delta_{A}\ot\Delta_{A})\\&=(\mathrm{Id}_{A}\ot\pi m_{A}\ot m_{A})(\mathrm{Id}_{A\ot A}\ot \sigma_{A,A}\ot\mathrm{Id}_{A})((\mathrm{Id}_{A}\ot\Delta_{A})\Delta_{A}\ot\Delta_{A})\\&=(\mathrm{Id}_{A}\ot\pi\ot\mathrm{Id}_{A})(\mathrm{Id}_{A}\ot m_{A}\ot m_{A})(\mathrm{Id}_{A\ot A}\ot \sigma_{A,A}\ot\mathrm{Id}_{A})(\mathrm{Id}_{A}\ot\Delta_{A}\ot\Delta_{A})(\Delta_{A}\ot\mathrm{Id}_{A})\\&=(\mathrm{Id}_{A}\ot\pi\ot\mathrm{Id}_{A})(\mathrm{Id}_{A}\ot \Delta_{A}m_{A})(\Delta_{A}\ot\mathrm{Id}_{A})\\&=(\mathrm{Id}_{A}\ot(\pi\ot\mathrm{Id}_{A})\Delta_{A})\mathsf{can}.
\end{align*}
Therefore, we obtain the following commutative diagram
\[
\begin{tikzcd}
	A\ot A^{\mathrm{co}Q} & A\ot A &&& A\ot Q\ot A \\
	A\square_{Q}A & A\ot A &&& A\ot Q\ot A
	\arrow[from=1-1, to=1-2,"\mathrm{Id}_{A}\ot i"]
	\arrow[from=1-1, to=2-1,"\xi"']
	\arrow[shift left, from=1-2, to=1-5, "\mathrm{Id}_{A}\ot\pi u_{A}\ot\mathrm{Id}_{A}"]
	\arrow[shift right, from=1-2, to=1-5,"\mathrm{Id}_{A}\ot(\pi\ot\mathrm{Id}_{A})\Delta_{A}"']
	\arrow[from=1-2, to=2-2, "\mathsf{can}"']
	\arrow[from=1-5, to=2-5,"f"]
	\arrow[from=2-1, to=2-2, "e_{A,A}"']
	\arrow[shift left, from=2-2, to=2-5,"(\mathrm{Id}_{A}\ot\pi)\Delta_{A}\ot\mathrm{Id}_{A}"]
	\arrow[shift right, from=2-2, to=2-5, "\mathrm{Id}_{A}\ot(\pi\ot\mathrm{Id}_{A})\Delta_{A}"']
\end{tikzcd}\]
where $\xi$ is obtained by universal property. In fact, since $i:A^{\mathrm{co}Q}\to A$ is an equalizer in $\Mm$ and $A\ot(-)$ preserves equalizers in $\Mm$, the first row in the previous diagram is an equalizer in $\Mm$, while the second row is an equalizer in $\Mm$ by definition. Moreover, we have
\[
\begin{split}
(\mathrm{Id}_{A}\ot(\pi\ot\mathrm{Id}_{A})\Delta_{A})\mathsf{can}(\mathrm{Id}_{A}\ot i)&=f(\mathrm{Id}_{A}\ot(\pi\ot\mathrm{Id}_{A})\Delta_{A})(\mathrm{Id}_{A}\ot i)
=f(\mathrm{Id}_{A}\ot\pi u_{A}\ot\mathrm{Id}_{A})(\mathrm{Id}_{A}\ot i)\\&=((\mathrm{Id}_{A}\ot\pi)\Delta_{A}\ot\mathrm{Id}_{A})\mathsf{can}(\mathrm{Id}_{A}\ot i).
\end{split}
\]
Since $A$ has antipode $S_{A}$ and it is cocommutative, one can easily verify that $f$ is an isomorphism in $\Mm$ with inverse given by
\[
f^{-1}:=(\mathrm{Id}_{A}\ot\mu_{Q}\ot m_{A})(\mathrm{Id}_{A}\ot S_{A}\ot\mathrm{Id}_{Q}\ot S_{A}\ot\mathrm{Id}_{A})(\mathrm{Id}_{A\ot A}\ot \sigma_{A,Q}\ot\mathrm{Id}_{A})(\mathrm{Id}_{A}\ot\Delta_{A}\ot\mathrm{Id}_{Q\ot A})(\Delta_{A}\ot\mathrm{Id}_{Q\ot A}).
\]
Since $\mathsf{can}$ and $f$ are isomorphisms in $\Mm$ and both lines are equalizers in $\Mm$, we obtain that $\xi$ is an isomorphism in $\Mm$. Moreover, we compute
\[
\begin{split}
(\varepsilon_{A}\ot\mathrm{Id}_{A})e_{A,A}\xi&=(\varepsilon_{A}\ot\mathrm{Id}_{A})\mathsf{can}(\mathrm{Id}_{A}\ot i)=(\varepsilon_{A}\ot\mathrm{Id}_{A})(\mathrm{Id}_{A}\ot m_{A})(\Delta_{A}\ot\mathrm{Id}_{A})(\mathrm{Id}_{A}\ot i)\\&=m_{A}(\varepsilon_{A}\ot\mathrm{Id}_{A\ot A})(\Delta_{A}\ot\mathrm{Id}_{A})(\mathrm{Id}_{A}\ot i)=m_{A}(\mathrm{Id}_{A}\ot i)
\end{split}
\]
and
\[
\begin{split}
(\mathrm{Id}_{A}\ot\varepsilon_{A})e_{A,A}\xi&=(\mathrm{Id}_{A}\ot\varepsilon_{A})\mathsf{can}(\mathrm{Id}_{A}\ot i)=(\mathrm{Id}_{A}\ot\varepsilon_{A})(\mathrm{Id}_{A}\ot m_{A})(\Delta_{A}\ot\mathrm{Id}_{A})(\mathrm{Id}_{A}\ot i)\\&=(\mathrm{Id}_{A}\ot\varepsilon_{A}\ot\varepsilon_{A})(\Delta_{A}\ot\mathrm{Id}_{A})(\mathrm{Id}_{A}\ot i)=\mathrm{Id}_{A}\ot\varepsilon_{A}i=\mathrm{Id}_{A}\ot\varepsilon_{A^{\mathrm{co}Q}}.
\end{split}
\]
Therefore, we have the following commmutative diagram in $\Mm$: 
\[\begin{tikzcd}
	A\ot A^{\mathrm{co}Q} && A & Q' \\
	A\square_{Q}A && A & Q
	\arrow[shift left, from=1-1, to=1-3, "m_{A}(\mathrm{Id}_{A}\ot i)"]
	\arrow[shift right, from=1-1, to=1-3,"\mathrm{Id}_{A}\ot\varepsilon_{A^{\mathrm{co}Q}}"']
	\arrow[from=1-1, to=2-1,"\xi"']
	\arrow[from=1-3, to=1-4, "\pi'"]
	\arrow[from=1-3, to=2-3,"\mathrm{Id}_{A}"]
	\arrow[from=1-4, to=2-4,"\varphi"]
	\arrow[shift left, from=2-1, to=2-3,"(\varepsilon_{A}\ot\mathrm{Id}_{A})e_{A,A}"]
	\arrow[shift right, from=2-1, to=2-3,"(\mathrm{Id}_{A}\ot\varepsilon_{A})e_{A,A}"']
	\arrow[from=2-3, to=2-4,"\pi"']
\end{tikzcd}\]
Since both rows are coequalizers in $\Mm$ (the second one by assumption) and $\xi$ is an isomorphism in $\Mm$, it follows that $\varphi$ is an isomorphism in $\Mm$ as well.
\end{proof}

We point out that the conditions assumed on subojects and quotients are needed to prove that, for an arbitrary $A\in\Hopf$, the maps $\phi_{A}$ and $\psi_{A}$ are inverse to each other, not to define them.

\begin{remark}
By applying Theorem \ref{thm:NewmanforM} to $(\Mm,\ot,\mathbf{1},\sigma)=(\mathsf{Vec}_{\Bbbk},\ot_{\Bbbk},\Bbbk,\tau)$, one recovers the well-known bijective correspondence for cocommutative Hopf algebras \eqref{bijectionNewman} which goes back to \cite[Theorem 4.1]{Newman}, while for $(\Mm,\ot,\mathbf{1},\sigma)=(\mathsf{Vec}_{G},\ot_{\Bbbk},\Bbbk,\sigma)$, where $G$ is a finitely generated abelian group and $\mathrm{char}(\Bbbk)\not=2$ (not needed if $G$ is finite of odd dimension), one recovers \cite[Theorem 5.20]{AS}. 
We also recall that the bijective correspondence given in \cite{Newman} for cocommutative Hopf algebras was also generalized to the setting of Hopf algebroids in \cite{EGSV,EGSV1} and we observe that the corresponding coequalizer condition on quotients was previously observed and studied in \cite{Saracco}. Given $(\Mm,\ot,\mathbf{1},\sigma)=(\mathsf{Vec}_{\Bbbk},\ot_{\Bbbk},\Bbbk,\tau)$ our categorical proof of Newman's Theorem recovers the one provided in \cite{EGSV,EGSV1}. In that case, as pointed out in \cite[page 10 (b)]{Saracco}, any subobject of $A$ in $\Hopf$ is an equalizer as in \eqref{iequalizer} and any quotient of $A$ in $\mathsf{Comon}(_{A}\Mm)$ is a coequalizer as in \eqref{picoequalizer}. In general, we do not know when this is true for $\Hopf$ given an arbitrary braided monoidal category $(\Mm,\ot,\mathbf{1},\sigma)$ with equalizers and coequalizers which are preserved by $\ot$. We will come back to this question in the next section.
\end{remark}

In Corollary \ref{cor:bijectivecorrespondence} we will prove that the bijective correspondence given in Theorem \ref{thm:NewmanforM} induces a bijective correspondence between kernels in $\Hopf$ and quotients in $\mathsf{Comon}_{\mathrm{coc}}(_{A}\Mm)$ which are regular epimorphisms in $\Hopf$, once the category $\Mm$ is pointed. We end this section by proving a result that will be used for that purpose. First, we need the following easy observation.

\begin{remark}\label{rmk:kerHopfideal}
    Let $f:A\to B$ be a morphism in $\Hopf$. Consider the kernel $\mathsf{ker}(f):\mathsf{Ker}(f)\to A$ of $f$ in $\Mm$, we have that $\mathsf{Ker}(f)$ is a Hopf ideal (see Section \ref{sec:prelimiaries}). In fact, it is routine to check that the morphism $\mathsf{ker}(f)$ is in $_{A}\Mm_{A}$, where the structure of $\mathsf{Ker}(f)$ in $_{A}\Mm_{A}$ is induced by $A$. In addition, $\mathsf{Ker}(f)$ is a two-sided coideal. In fact, for $\pi:=\mathsf{coker}(\mathsf{ker}(f))$ and $h:=\mathsf{ker}(\mathsf{coker}(f))$ in $\Mm$, we get
\[
0=\Delta_{B}f\mathsf{ker}(f)=(f\ot f)\Delta_{A}\mathsf{ker}(f)=(h\ot h)(\pi\ot\pi)\Delta_{A}\mathsf{ker}(f),
\]
and then, since $h\ot h$ is a monomorphism in $\Mm$, we obtain that $(\pi\ot\pi)\Delta_{A}\mathsf{ker}(f)=0$. Moreover, $\varepsilon_{A}\mathsf{ker}(f)=\varepsilon_{B}f\mathsf{ker}(f)=0$. Finally, since $0=S_{B}f\mathsf{ker}(f)=fS_{A}\mathsf{ker}(f)=h\pi S_{A}\mathsf{ker}(f)$ and $h$ is a monomorphism in $\Mm$, we obtain $\pi S_{A}\mathsf{ker}(f)=0$. As a result, $\mathsf{Ker}(f)$ is a Hopf ideal. Hence, $\pi$ is a morphism in $\Hopf$, which in fact is also a morphism in $\mathsf{Comon}_{\mathrm{coc}}(_{A}\Mm)$. 
Moreover, $h$ is also in $\Hopf$, since $f=h\pi$ with $\pi$ and $f$ in $\Hopf$ and $\pi$ is an epimorphism in $\Mm$.
\end{remark}

The following result will be useful in the next section.

\begin{lemma}\label{lem:inducedbijectivecorr}
The following statements hold:
\begin{itemize}
    \item[1)] Let $\pi:A\to Q$ be a quotient of $A$ in $\mathsf{Comon}_{\mathrm{coc}}(_{A}\Mm)$ which is also a morphism in $\Hopf$. Then $\psi_{A}(\pi)=\mathsf{hker}(\pi)$.
    \item[2)] Let $f:A\to B$ be a morphism in $\Hopf$ such that $\mathsf{coker}(\mathsf{ker}(f))$ is a coequalizer as in \eqref{picoequalizer}. Then $\phi_{A}(\mathsf{hker}(f))=\mathsf{coker}(\mathsf{ker}(f))$.
\end{itemize}
\end{lemma}

\begin{proof}
    1). Since $\pi:A\to Q$ is a morphism in $\mathsf{Hopf}_{\mathrm{coc}}(\Mm)$, we have that $\pi u_{A}\varepsilon_{A}=u_{Q}\varepsilon_{A}$. Thus, $\psi_{A}(\pi)$, which is defined as the equalizer of the pair $(\pi,\pi u_{A}\varepsilon_{A})$ in $\mathsf{Comon}_{\mathrm{coc}}(\Mm)$, is indeed the equalizer of the pair $(\pi,u_{Q}\varepsilon_{A})$ in $\Hopf$, i.e.\ $\mathsf{hker}(\pi)$.
    
2). We set $\pi:=\mathsf{coker}(\mathsf{ker}(f))$ and $h:=\mathsf{ker}(\mathsf{coker}(f))$ so that $f=h\pi$. By Remark \ref{rmk:kerHopfideal}, we already know that $\pi$ and $h$ are morphisms in $\Hopf$. Since $h\ot\mathrm{Id}_{A}$ is a monomorphism in $\Mm$, we have
\begin{align}
\mathsf{hker}(f)&=\mathsf{ker}((f\ot\id_{A})\Delta_{A}-u_{B}\ot\id_{A})\label{auxeq}\\
&=\mathsf{ker}((h\pi\ot\id_{A})\Delta_{A}-hu_{\mathsf{Ker}(\mathsf{coker}(f))}\ot\id_{A})\notag\\
&=\mathsf{ker}((h\ot\id_{A})((\pi\ot\id_{A})\Delta_{A}-u_{\mathsf{Ker}(\mathsf{coker}(f))}\ot\id_{A}))\notag\\
&=\mathsf{ker}((\pi\ot\id_{A})\Delta_{A}-u_{\mathsf{Ker}(\mathsf{coker}(f))}\ot\id_{A})\notag\\
&=\mathsf{hker}(\pi)=\mathsf{hker}(\mathsf{coker}(\mathsf{ker}(f)))\notag.
\end{align}
By 1) we already know that $\psi_{A}(\mathsf{coker}(\mathsf{ker}(f)))=\mathsf{hker}(\mathsf{coker}(\mathsf{ker}(f)))$, hence we get
\[
\mathsf{coker}(\mathsf{ker}(f))=\phi_{A}(\psi_{A}(\mathsf{coker}(\mathsf{ker}(f))))=\phi_{A}(\mathsf{hker}(\mathsf{coker}(\mathsf{ker}(f))))\overset{\eqref{auxeq}}{=}\phi_{A}(\mathsf{hker}(f)),
\]
where we use Theorem \ref{thm:NewmanforM} for the first equality, since $\mathsf{coker}(\mathsf{ker}(f))$ is coequalizer as in \eqref{picoequalizer}.
\end{proof}



In the next section, we provide an equivalent description for kernels in $\Hopf$.

\section{An equivalent characterization for kernels in $\Hopf$}

In this section, $(\Mm,\ot,\mathbf{1},\sigma)$ will be a braided monoidal category which has equalizers and coequalizers that are preserved by $\ot$. Additional assumptions will be assumed if needed.

As an application of Theorem \ref{thm:NewmanforM}, we can obtain another characterization of kernels in $\Hopf$. In order to do this, recall the definition of the adjoint morphism in $\Mm$: 
\[
\mathrm{ad}_{A}:=m_{A}(m_{A}\ot\mathrm{Id}_{A})(\mathrm{Id}_{A}\ot\sigma_{A,A})((\mathrm{Id}_{A}\ot S_{A})\Delta_{A}\ot\mathrm{Id}_{A}):A\ot A\to A.
\]
Equivalently, $\mathrm{ad}_{A}=m_{A}(m_{A}\ot S_{A})(\id_{A}\ot\sigma_{A,A})(\Delta_{A}\ot\id_{A})$.
Note that the morphism $\mathrm{ad}_{A}$ defines a left action of $A$ over itself, see e.g.\ \cite[Proposition 3.7.1]{HS}.

We introduce the following definition:

\begin{definition}\label{def:normal}
    Let $A$ be an object in $\Hopf$ and $i:K\to A$ be a monomorphism in $\Hopf$. We say that $i$ is \textit{left normal} if there exists a morphism $\psi:A\ot K\to K$ in $\Mm$ such that the following diagram
\begin{equation}\label{diagram:normal}
\begin{tikzcd}
	A\ot K && K \\
	A\ot A && A
	\arrow[from=1-1, to=1-3, "\psi"]
	\arrow[from=1-1, to=2-1,"\mathrm{Id}_{A}\ot i"']
	\arrow[from=1-3, to=2-3, "i"]
	\arrow[from=2-1, to=2-3, "\mathrm{ad}_{A}"']
\end{tikzcd}
\end{equation}
commutes.
\end{definition}

One can give the definition of normal monomorphism for $\mathsf{Hopf}(\Mm)$ in the same way, without assuming cocommutativity. In this paper, we will work only with $\Hopf$. 

\begin{remark}
    Given $A$ in $\Hopf$, one could also define 
\[
\mathrm{ad}_{A}':=m_{A}(m_{A}\ot\mathrm{Id}_{A})(\sigma_{A,A}\ot\mathrm{Id}_{A})(\mathrm{Id}_{A}\ot(S_{A}\ot\mathrm{Id}_{A})\Delta_{A}):A\ot A\to A.
\]
A monomorphism $i:K\to A$ in $\Hopf$ is said to be right normal if there exists a morphism $\psi':K\ot A\to K$ in $\Mm$ such that $i\psi'=\mathrm{ad}'_{A}(i\ot\mathrm{Id}_{A})$. 
\end{remark}

The following result shows that $i$ is left normal if and only if $i$ is right normal. In the following, we will simply say that $i$ is a normal monomorphism.

\begin{lemma}\label{lem:leftrightnormal}
Suppose $(\Mm,\ot,\mathbf{1},\sigma)$ is a symmetric monoidal category. Let $i:K\to A$ be a monomorphism in $\Hopf$. Then $i$ is left normal if and only if it is right normal.
\end{lemma}

\begin{proof}
Since $A$ is cocommutative, we have
\[
\begin{split}
\mathrm{ad}_{A}(S_{A}\ot\mathrm{Id}_{A})\sigma_{A,A}&=m_{A}(m_{A}\ot\mathrm{Id}_{A})(\mathrm{Id}_{A}\ot\sigma_{A,A})((\mathrm{Id}_{A}\ot S_{A})\Delta_{A}\ot\mathrm{Id}_{A})(S_{A}\ot\mathrm{Id}_{A})\sigma_{A,A}\\&=m_{A}(m_{A}\ot\mathrm{Id}_{A})(\mathrm{Id}_{A}\ot\sigma_{A,A})((S_{A}\ot S^{2}_{A})\Delta_{A}\ot\mathrm{Id}_{A})\sigma_{A,A}\\&=m_{A}(m_{A}\ot\mathrm{Id}_{A})(\mathrm{Id}_{A}\ot\sigma_{A,A})((S_{A}\ot\mathrm{Id}_{A})\Delta_{A}\ot\mathrm{Id}_{A})\sigma_{A,A}\\&=m_{A}(m_{A}\ot\mathrm{Id}_{A})(\mathrm{Id}_{A}\ot\sigma_{A,A})\sigma_{A,A\ot A}(\mathrm{Id}_{A}\ot(S_{A}\ot\mathrm{Id}_{A})\Delta_{A})\\&=m_{A}(m_{A}\ot\mathrm{Id}_{A})(\mathrm{Id}_{A}\ot\sigma_{A,A}^{2})(\sigma_{A,A}\ot\mathrm{Id}_{A})(\mathrm{Id}_{A}\ot(S_{A}\ot\mathrm{Id}_{A})\Delta_{A})\\&=\mathrm{ad}'_{A}.
\end{split}
\]
If $i$ is left normal, then there exists a morphism $\psi:A\ot K\to K$ in $\Mm$ such that $i\psi=\mathrm{ad}_{A}(\mathrm{Id}_{A}\ot i)$. Define $\psi':=\psi(S_{A}\ot\mathrm{Id}_{K})\sigma_{K,A}:K\ot A\to K$ in $\Mm$. We get
\[
i\psi'=i\psi(S_{A}\ot\mathrm{Id}_{K})\sigma_{K,A}=\mathrm{ad}_{A}(\mathrm{Id}_{A}\ot i)(S_{A}\ot\mathrm{Id}_{K})\sigma_{K,A}=\mathrm{ad}_{A}(S_{A}\ot\mathrm{Id}_{A})\sigma_{A,A}(i\ot\mathrm{Id}_{A})=\mathrm{ad}_{A}'(i\ot\mathrm{Id}_{A}).
\]
Hence, $i$ is right normal. Conversely, if $i$ is right normal, i.e.\ there exists $\psi':K\ot A\to K$ in $\Mm$ such that $i\psi'=\mathrm{ad}'_{A}(i\ot\mathrm{Id}_{A})$, define $\psi:=\psi'(\mathrm{Id}_{K}\ot S_{A})\sigma_{A,K}$. Then,
\begin{align*}
i\psi &= i\psi'(\mathrm{Id}_{K}\ot S_{A})\sigma_{A,K} = \mathrm{ad}'_{A}(i\ot\mathrm{Id}_{A})(\mathrm{Id}_{K}\ot S_{A})\sigma_{A,K}\\ &= \mathrm{ad}_{A}(S_{A}\ot\mathrm{Id}_{A})\sigma_{A,A} (i\ot\mathrm{Id}_{A})(\mathrm{Id}_{K}\ot S_{A})\sigma_{A,K}\\
&= \mathrm{ad}_{A}(S_{A}^{2}\ot\id_{A}) \sigma_{A,A}^2(\mathrm{Id}_{A}\ot i)
=\mathrm{ad}_{A}(\mathrm{Id}_{A}\ot i).
\end{align*}
This means that $i$ is left normal.
\end{proof}

We also obtain the following result, which will be very useful in the sequel.
\begin{lemma}\label{properties:ad}
    Let $(\Mm,\ot,\mathbf{1},\sigma)$ be a symmetric monoidal category. Let $A$ be an object in $\Hopf$. The following properties hold:
\begin{itemize}
    \item[1)] $\mathrm{ad}_{A}$ is a morphism in $\mathsf{Comon}_{\mathrm{coc}}(\Mm)$. 
    \item[2)] Given $g:A\to B$ in $\Hopf$, then $\mathrm{ad}_{B}(g\ot g)=g\mathrm{ad}_A$. 
    \item[3)] Suppose $\Mm$ is an abelian symmetric monoidal category. Given $g:A\to B$ in $\Hopf$ which is an epimorphism in $\Mm$ and a normal monomorphism $i:D\to A$ in $\Hopf$, then $\mathsf{ker}(\mathsf{coker}(gi))$ is normal. 
\end{itemize}
\end{lemma}

\begin{proof}
1). Since $(\Mm,\ot,\mathbf{1},\sigma)$ is a symmetric monoidal category, $\sigma_{X,Y}$ is in $\mathsf{Comon}(\Mm)$ for any objects $X$, $Y$ in $\mathsf{Comon}(\Mm)$. Therefore, the morphism $\mathrm{ad}_{A}$ is in $\mathsf{Comon}(\Mm)$ (hence in $\mathsf{Comon}_{\mathrm{coc}}(\Mm)$) as it is composition of morphisms in $\mathsf{Comon}(\Mm)$. We point out that the cocommutativity of $A$ is used to have that $S_{A}$ and $\Delta_{A}$ are morphisms in $\mathsf{Comon}(\Mm)$.

2). It is straightforward and it doesn't use the fact that $\sigma$ is a symmetry.
 \begin{invisible}
    \[
    \begin{split}
        \mathrm{ad}_{B}(g\ot g)&=m_{B}(m_{B}\ot\mathrm{Id}_{B})(\mathrm{Id}_{B}\ot\sigma_{B,B})((\mathrm{Id}_{B}\ot S_{B})\Delta_{B}\ot\mathrm{Id}_{B})(g\ot g)\\&=m_{B}(m_{B}\ot\mathrm{Id}_{B})(\mathrm{Id}_{B}\ot\sigma_{B,B})((g\ot gS_{A})\Delta_{A}\ot g))\\&=m_{B}(m_{B}\ot\mathrm{Id}_{B})(g\ot g\ot g)(\mathrm{Id}_{A}\ot\sigma_{A,A})((\mathrm{Id}_{A}\ot S_{A})\Delta_{A}\ot\mathrm{Id}_{A})\\&=m_{B}(g\ot g)(m_{A}\ot\mathrm{Id}_{A})(\mathrm{Id}_{A}\ot\sigma_{A,A})((\mathrm{Id}_{A}\ot S_{A})\Delta_{A}\ot\mathrm{Id}_{A})\\&=gm_{A}(m_{A}\ot\mathrm{Id}_{A})(\mathrm{Id}_{A}\ot\sigma_{A,A})((\mathrm{Id}_{A}\ot S_{A})\Delta_{A}\ot\mathrm{Id}_{A})\\&=g\mathrm{ad}_{A}
    \end{split}
    \]
\end{invisible}

3). By Remark \ref{rmk:kerHopfideal}, we already know that $\mathsf{ker}(\mathsf{coker}(gi))$ is a monomorphism in $\Hopf$.
For brevity, we set $t:= \mathsf{ker}(\mathsf{coker}(gi))$ and $t':= \mathsf{coker}(\mathsf{ker}(gi))$. Since $i$ is normal, there is a morphism $\psi_D:A \ot D \to D$ in $\Mm$ such that $\mathrm{ad}_{A} (\id_A \ot i) = i \psi_D$. By 2), we have
$$
\mathsf{coker}(gi) \mathrm{ad}_{B} (g \ot t t') = \mathsf{coker}(gi) \mathrm{ad}_{B} (g \ot gi)  = \mathsf{coker}(gi) g \mathrm{ad}_{A} (\id_A \ot i)= \mathsf{coker}(gi) gi \psi_D = 0. 
$$
Since $g$ and $t'$ are epimorphisms in $\Mm$ and $\ot$ preserves them, we obtain that $g\ot t'$ is also an epimorphism in $\Mm$. Hence, $\mathsf{coker}(gi) \mathrm{ad}_{B} (\id_B \ot t) = 0$. Thus, by the universal property of the kernel, there exists a unique morphism $\psi: B \ot \mathsf{Ker}(\mathsf{coker}(gi)) \to \mathsf{Ker}(\mathsf{coker}(gi))$ in $\Mm$ such that $t \psi = \mathrm{ad}_{B} (\id_B \ot t)$.
\end{proof}

\begin{remark}\label{rmk:psipsi'counitary}
Given a normal monomorphism $i:K\to A$ in $\Hopf$ and the corresponding morphism $\psi:A\ot K\to K$ in $\Mm$, one immediately obtains that $\varepsilon_{K}\psi=\varepsilon_{A}i\psi=\varepsilon_{A}\mathrm{ad}_{A}(\mathrm{Id}_{A}\ot i)=\varepsilon_{A}\ot\varepsilon_{A}i=\varepsilon_{A}\ot\varepsilon_{K}$, since $\mathrm{ad}_{A}$ is counitary. As a consequence, given the corresponding $\psi':=\psi(S_{A}\ot\mathrm{Id}_{K})\sigma_{K,A}:K\ot A\to K$, we have $\varepsilon_{K}\psi'=\varepsilon_{K}\ot\varepsilon_{A}$.
\end{remark}

By means of normal monomorphism, we obtain the aforementioned characterization of kernels in $\Hopf$. This is proven thanks to the following two results.

\begin{proposition}\label{prop:normalimpliesHopfmap}
    Let $i:K\to A$ be a monomorphism in $\Hopf$, where $(\Mm,\ot,\mathbf{1},\sigma)$ is a symmetric monoidal category. If $i$ is normal then $\pi:=\phi_{A}(i)=\mathsf{coeq}(m_{A}(\id_{A}\ot i),\id_{A}\ot\varepsilon_{K})$ is a morphism in $\Hopf$.
\end{proposition}

\begin{proof}
   By 1) of Proposition \ref{prop:definitionbijections} we know that $\pi:=\phi_{A}(i)=\mathsf{coeq}(m_{A}(\mathrm{Id}_{A}\ot i),\mathrm{Id}_{A}\ot\varepsilon_{K}):A\to Q$ is an epimorphism in $\mathsf{Comon}_{\mathrm{coc}}(_{A}\Mm)$. We prove that $\pi$ is a morphism in $\Hopf$. First, we compute
\[
\begin{split}
  &m_{A}(m_{A}\ot\mathrm{Id}_{A})=m_{A}(m_{A}\ot\mathrm{Id}_{A})(\mathrm{Id}_{A\ot A}\ot\varepsilon_{A}\ot\mathrm{Id}_{A})(\mathrm{Id}_{A\ot A}\ot\Delta_{A})\\&=m_{A}(\mathrm{Id}_{A}\ot m_{A})(\mathrm{Id}_{A}\ot\varepsilon_{A}\ot\mathrm{Id}_{A\ot A})(\mathrm{Id}_{A}\ot\sigma_{A,A}\ot\mathrm{Id}_{A})(\mathrm{Id}_{A\ot A}\ot\Delta_{A})\\&=m_{A}(m_{A}\ot m_{A})(\mathrm{Id}_{A}\ot u_{A}\varepsilon_{A}\ot\mathrm{Id}_{A\ot A})(\mathrm{Id}_{A}\ot\sigma_{A,A}\ot\mathrm{Id}_{A})(\mathrm{Id}_{A\ot A}\ot\Delta_{A})\\&=m_{A}(m_{A}\ot m_{A})(\mathrm{Id}_{A}\ot\sigma_{A,A}\ot\mathrm{Id}_{A})(\mathrm{Id}_{A\ot A}\ot u_{A}\varepsilon_{A}\ot\mathrm{Id}_{A})(\mathrm{Id}_{A\ot A}\ot\Delta_{A})\\&=m_{A}(m_{A}\ot m_{A})(\mathrm{Id}_{A}\ot\sigma_{A,A}\ot\mathrm{Id}_{A})(\mathrm{Id}_{A\ot A}\ot m_{A}(\mathrm{Id}_{A}\ot S_{A})\Delta_{A}\ot\mathrm{Id}_{A})(\mathrm{Id}_{A\ot A}\ot\Delta_{A})\\&=m_{A}(m_{A}\ot m_{A})(\mathrm{Id}_{A}\ot m_{A}\ot\mathrm{Id}_{A\ot A})(\mathrm{Id}_{A}\ot\sigma_{A,A\ot A}\ot\mathrm{Id}_{A})(\mathrm{Id}_{A\ot A}\ot(\mathrm{Id}_{A}\ot S_{A})\Delta_{A}\ot\mathrm{Id}_{A})(\mathrm{Id}_{A\ot A}\ot\Delta_{A})\\&=m_{A}(m_{A}\ot m_{A}(m_{A}\ot\mathrm{Id}_{A}))(\mathrm{Id}_{A}\ot\sigma_{A,A\ot A}\ot\mathrm{Id}_{A})(\mathrm{Id}_{A\ot A}\ot(\mathrm{Id}_{A}\ot S_{A})\Delta_{A}\ot\mathrm{Id}_{A})(\mathrm{Id}_{A\ot A}\ot\Delta_{A})\\&=m_{A}(m_{A}\ot m_{A}(m_{A}\ot\mathrm{Id}_{A}))(\mathrm{Id}_{A\ot A}\ot\sigma_{A,A}\ot\mathrm{Id}_{A})(\mathrm{Id}_{A}\ot\sigma_{A,A}\ot\mathrm{Id}_{A\ot A})\\&\hspace{0.5cm}(\mathrm{Id}_{A\ot A\ot A}\ot( S_{A}\ot\mathrm{Id}_{A})\Delta_{A})(\mathrm{Id}_{A\ot A}\ot\Delta_{A})\\&=m_{A}(m_{A}\ot m_{A}(m_{A}\ot\mathrm{Id}_{A})(\sigma_{A,A}\ot\mathrm{Id}_{A}))(\mathrm{Id}_{A}\ot\sigma_{A,A}\ot(S_{A}\ot\mathrm{Id}_{A})\Delta_{A})(\mathrm{Id}_{A\ot A}\ot\Delta_{A})\\&=m_{A}(m_{A}\ot m_{A}(m_{A}\ot\mathrm{Id}_{A})(\sigma_{A,A}\ot\mathrm{Id}_{A})(\mathrm{Id}_{A}\ot(S_{A}\ot\mathrm{Id}_{A})\Delta_{A}))(\mathrm{Id}_{A}\ot\sigma_{A,A}\ot\mathrm{Id}_{A})(\mathrm{Id}_{A\ot A}\ot\Delta_{A})\\&=m_{A}(m_{A}\ot\mathrm{ad}'_{A})(\mathrm{Id}_{A}\ot\sigma_{A,A}\ot\mathrm{Id}_{A})(\mathrm{Id}_{A\ot A}\ot\Delta_{A}),
\end{split}
\]
so we get
\begin{equation}\label{auxeqq2}
    m_{A}(m_{A}\ot\mathrm{Id}_{A})=m_{A}(m_{A}\ot\mathrm{ad}'_{A})(\mathrm{Id}_{A}\ot\sigma_{A,A}\ot\mathrm{Id}_{A})(\mathrm{Id}_{A\ot A}\ot\Delta_{A}).
\end{equation}
Since $i$ is normal, using Lemma \ref{lem:leftrightnormal}, there exists a morphism $\psi':K\ot A\to K$ in $\Mm$ such that $\mathrm{ad}'_{A}(i\ot\mathrm{Id}_{A})=i\psi'$. Therefore, we get
\[
\begin{split}
m_{A}(m_{A}\ot\mathrm{Id}_{A})(\mathrm{Id}_{A}\ot i\ot\mathrm{Id}_{A})&\overset{\eqref{auxeqq2}}{=}m_{A}(m_{A}\ot\mathrm{ad}'_{A})(\mathrm{Id}_{A}\ot\sigma_{A,A}\ot\mathrm{Id}_{A})(\mathrm{Id}_{A\ot A}\ot\Delta_{A})(\mathrm{Id}_{A}\ot i\ot\mathrm{Id}_{A})\\&=
m_{A}(m_{A}\ot\mathrm{ad}'_{A})(\mathrm{Id}_{A\ot A}\ot i\ot\mathrm{Id}_{A})(\mathrm{Id}_{A}\ot\sigma_{K,A}\ot\mathrm{Id}_{A})(\mathrm{Id}_{A\ot K}\ot\Delta_{A})\\&=
m_{A}(m_{A}\ot i\psi')(\mathrm{Id}_{A}\ot\sigma_{K,A}\ot\mathrm{Id}_{A})(\mathrm{Id}_{A\ot K}\ot\Delta_{A}).
\end{split}
\]
Then, by Remark \ref{rmk:psipsi'counitary}, we have
\[
\begin{split}
\pi m_{A}(m_{A}\ot\mathrm{Id}_{A})(\mathrm{Id}_{A}\ot i\ot\mathrm{Id}_{A})&=\pi m_{A}(m_{A}\ot i\psi')(\mathrm{Id}_{A}\ot\sigma_{K,A}\ot\mathrm{Id}_{A})(\mathrm{Id}_{A\ot K}\ot\Delta_{A})\\&=\pi m_{A}(\mathrm{Id}_{A}\ot i)(m_{A}\ot\psi')(\mathrm{Id}_{A}\ot\sigma_{K,A}\ot\mathrm{Id}_{A})(\mathrm{Id}_{A\ot K}\ot\Delta_{A})\\&=\pi (\mathrm{Id}_{A}\ot \varepsilon_{K})(m_{A}\ot\psi')(\mathrm{Id}_{A}\ot\sigma_{K,A}\ot\mathrm{Id}_{A})(\mathrm{Id}_{A\ot K}\ot\Delta_{A})\\&=\pi (m_{A}\ot\varepsilon_{K}\ot\varepsilon_{A})(\mathrm{Id}_{A}\ot\sigma_{K,A}\ot\mathrm{Id}_{A})(\mathrm{Id}_{A\ot K}\ot\Delta_{A})\\&=\pi m_{A}(\mathrm{Id}_{A}\ot\varepsilon_{K}\ot(\mathrm{Id}_{A}\ot\varepsilon_{A})\Delta_{A})\\&=\pi m_{A}(\mathrm{Id}_{A}\ot\varepsilon_{K}\ot\mathrm{Id}_{A}).
\end{split}
\]
Since $\pi\ot\mathrm{Id}_{A}$ is the coequalizer of the pair $(m_{A}(\mathrm{Id}_{A}\ot i)\ot\mathrm{Id}_{A},\mathrm{Id}_{A}\ot\varepsilon_{K}\ot\mathrm{Id}_{A})$ in $\Mm$, there exists a unique morphism $\xi:Q\ot A\to Q$ in $\Mm$ such that 
\begin{equation}\label{equationpimA}
    \xi(\pi\ot\mathrm{Id}_{A})=\pi m_{A}.
\end{equation}
Consequently,
\[
\begin{split}
    \xi(\mathrm{Id}_{Q}\ot m_{A}(\mathrm{Id}_{A}\ot i))(\pi\ot\mathrm{Id}_{A\ot K})&=\xi(\pi\ot\mathrm{Id}_{A})(\mathrm{Id}_{A}\ot m_{A}(\mathrm{Id}_{A}\ot i))\overset{\eqref{equationpimA}}{=}\pi m_{A}(\mathrm{Id}_{A}\ot m_{A}(\mathrm{Id}_{A}\ot i))\\&=\pi m_{A}(m_{A}\ot i)=\pi(m_{A}\ot\varepsilon_{K})=\xi(\mathrm{Id}_{Q\ot A}\ot\varepsilon_{K})(\pi\ot\mathrm{Id}_{A\ot K}).
\end{split}
\]
Since $\pi\ot\mathrm{Id}_{A\ot K}$ is an epimorphism in $\Mm$, we obtain $\xi(\mathrm{Id}_{Q}\ot m_{A}(\mathrm{Id}_{A}\ot i))=\xi(\mathrm{Id}_{Q\ot A}\ot\varepsilon_{K})$. By the universal property of the coequalizer $\mathrm{Id}_{Q}\ot\pi$, there exists a unique morphism $m_{Q}:Q\ot Q\to Q$ in $\Mm$ such that $m_{Q}(\mathrm{Id}_{Q}\ot\pi)=\xi$. Moreover, we define $u_{Q}:=\pi u_{A}$. One can check that $(Q,m_{Q},u_{Q})$ is in $\mathsf{Mon}(\Mm)$, since $A$ is in $\mathsf{Mon}(\Mm)$ and $\pi$ is an epimorphism in $\Mm$ which is preserved by $\ot$. Moreover, since
$\pi m_{A}=\xi(\pi\ot\mathrm{Id}_{A})=m_{Q}(\pi\ot\pi)$ and $u_{Q}=\pi u_{A}$, we get that $\pi$ is in $\mathsf{Mon}(\Mm)$. Indeed, $Q$ is in $\mathsf{Bimon}_{\mathrm{coc}}(\Mm)$ since $A$ is in $\mathsf{Bimon}_{\mathrm{coc}}(\Mm)$ and $\pi$ is an epimorphism in $\Mm$, so $\pi:A\to Q$ is in $\mathsf{Bimon}_{\mathrm{coc}}(\Mm)$. In order to conclude that $\pi$ is in $\Hopf$, it remains to show that $Q$ has an antipode. Because
\[
\begin{split}
    \pi S_{A}m_{A}(\mathrm{Id}_{A}\ot i)&=\pi m_{A}\sigma_{A,A}(S_{A}\ot S_{A})(\mathrm{Id}_{A}\ot i)=\pi m_{A}\sigma_{A,A}(\mathrm{Id}_{A}\ot i)(S_{A}\ot S_{K})\\&=\pi m_{A}(i\ot\mathrm{Id}_{A})\sigma_{A,K}(S_{A}\ot S_{K})=m_{Q}(\pi\ot\pi)(i\ot\mathrm{Id}_{A})\sigma_{A,K}(S_{A}\ot S_{K})\\&=m_{Q}(\pi u_{A}\varepsilon_{K}\ot\pi)\sigma_{A,K}(S_{A}\ot S_{K})=m_{Q}(\pi\ot\pi)\sigma_{A,A}(\mathrm{Id}_{A}\ot u_{A}\varepsilon_{K})(S_{A}\ot S_{K})\\&=\pi m_{A}\sigma_{A,A}(S_{A}\ot u_{A}\varepsilon_{K})=\pi m_{A}(u_{A}\ot\mathrm{Id}_{A})(S_{A}\ot \varepsilon_{K})=\pi S_{A}(\mathrm{Id}_{A}\ot\varepsilon_{K}),
\end{split}
\]
there exists a unique morphism $S_{Q}:Q\to Q$ in $\Mm$ such that $\pi S_{A}=S_{Q}\pi$. One can check that $S_{Q}$ is an antipode for $Q$ using the fact that $S_{A}$ is the antipode of $A$. As a consequence, $\pi$ is a morphism in $\Hopf$.
\end{proof}

\begin{proposition}\label{prop:kernelimpliesmono}
    Let $i:K\to A$ be a monomorphism in $\Hopf$, where $(\Mm,\ot,\mathbf{1},\sigma)$ is a symmetric monoidal category. If $i$ is a kernel in $\Hopf$ then it is normal. 
\end{proposition}

\begin{proof}
    Suppose $i= \mathsf{hker}(f)$ for some $f:A\to B$ in $\Hopf$, i.e.\ $i$ is the equalizer of the pair $((f\ot\mathrm{Id}_{A})\Delta_{A},u_{B}\ot\mathrm{Id}_{A})$ in $\Mm$ (see Lemma \ref{lem:eqkernel}). Since 
\begin{align*}
\mathrm{ad}_{B} (f \ot u_B) = &m_{B}(m_{B}\ot\mathrm{Id}_{B})(\mathrm{Id}_{B}\ot\sigma_{B,B})((\mathrm{Id}_{B}\ot S_{B})\Delta_{B}\ot\mathrm{Id}_{B})(f \ot u_B)\\
= &m_{B}(m_{B}\ot\mathrm{Id}_{B})(\mathrm{Id}_{B}\ot\sigma_{B,B})(\mathrm{Id}_{B}\ot S_{B} \ot \id_B) (f \ot f \ot u_B) \Delta_{A}\\
= &m_{B}(m_{B}\ot\mathrm{Id}_{B})(\mathrm{Id}_{B}\ot\sigma_{B,B}) (f \ot f \ot u_B) (\mathrm{Id}_{A}\ot S_{A}) \Delta_{A}\\
= &m_{B}(m_{B}\ot\mathrm{Id}_{B}) (f \ot u_B \ot f) (\mathrm{Id}_{A}\ot S_{A}) \Delta_{A}\\
= &m_{B} (f \ot f) (\mathrm{Id}_{A}\ot S_{A}) \Delta_{A} = f m_A (\mathrm{Id}_{A}\ot S_{A}) \Delta_{A} \\=& f u_A \varepsilon_A = u_B \varepsilon_A
\end{align*}
and $(f\ot\mathrm{Id}_{A})\Delta_{A}i = u_{B}\ot i$, by using 1) and 2) of Lemma \ref{properties:ad} we obtain
\begin{align*}
(f\ot\mathrm{Id}_{A}) \Delta_{A}\mathrm{ad}_{A} (\id_A \ot i) = &(f\ot\mathrm{Id}_{A})(\mathrm{ad}_{A} \ot \mathrm{ad}_{A})\Delta_{A \ot A} (\id_A \ot i)
\\=& (f \mathrm{ad}_{A} \ot \mathrm{ad}_{A})(\id_A \ot \sigma_{A,A} \ot \id_A) (\Delta_{A} \ot \Delta_{A}) (\id_A \ot i)\\
= &(\mathrm{ad}_{B} \ot \mathrm{ad}_{A})(f \ot f \ot \id_A \ot \id_A) (\id_A \ot \sigma_{A,A} \ot \id_A) (\Delta_{A} \ot \Delta_{A}) (\id_A \ot i)\\
= &(\mathrm{ad}_{B} \ot \mathrm{ad}_{A}) (\id_A \ot \sigma_{A,B} \ot \id_A) (f \ot \id_A \ot f \ot \id_A)(\Delta_{A} \ot \Delta_{A}) (\id_A \ot i)\\
= &(\mathrm{ad}_{B} \ot \mathrm{ad}_{A}) (\id_A \ot \sigma_{A,B} \ot \id_A) (f \ot \id_A \ot u_B \ot i)(\Delta_{A} \ot \id_K)\\
= &(\mathrm{ad}_{B} \ot \mathrm{ad}_{A}) (f \ot u_B \ot \id_A \ot i)(\Delta_{A} \ot \id_K)\\
= &(u_B \varepsilon_A \ot \mathrm{ad}_{A}(\id_A \ot i))(\Delta_{A} \ot \id_K)
= u_B \ot \mathrm{ad}_{A}(\id_A \ot i).
\end{align*}
Therefore, by the universal property of equalizer, there is a unique morphism $\psi:A \ot K \to A$ in $\Mm$ such that $i\psi = \mathrm{ad}_{A}(\id_A \ot i)$, so $i$ is normal.
\end{proof}

We finally obtain the following result:

\begin{theorem}\label{cor:normalobjects}
    Let $i:K\to A$ be a monomorphism in $\Hopf$ which is an equalizer as in \eqref{iequalizer}, where $(\Mm,\ot,\mathbf{1},\sigma)$ is a symmetric monoidal category. The following conditions are equivalent:
\begin{itemize}
    \item[1)] $i$ is normal;
\item[2)] $\pi:=\phi_{A}(i)=\mathsf{coeq}(m_{A}(\id_{A}\ot i),\id_{A}\ot\varepsilon_{K})$ is a morphism in $\Hopf$;
\item[3)] $i$ is a kernel in $\Hopf$.
\end{itemize}
The implications $3)\Rightarrow1)\Rightarrow2)$ hold without asking that $i$ is an equalizer as in \eqref{iequalizer}.
\end{theorem}

\begin{proof}
$1)\Rightarrow2)$. This is Proposition \ref{prop:normalimpliesHopfmap}.

$2)\Rightarrow3)$. Since $i$ is a monomorphism in $\Hopf$ which is an equalizer as in \eqref{iequalizer}, we have that $i=\psi_{A}(\phi_{A}(i))$ by Theorem \ref{thm:NewmanforM}. Then, since $\pi:=\phi_{A}(i)$ is a morphism in $\Hopf$, we get $i=\psi_{A}(\pi)=\mathsf{hker}(\pi)$ by 1) of Lemma \ref{lem:inducedbijectivecorr}.

$3)\Rightarrow1)$. This is Proposition \ref{prop:kernelimpliesmono}.
\end{proof}

Usually a monomorphism is said to be normal if it is the kernel of some morphism. The previous theorem justifies the name ``normal'' adopted in Definition \ref{def:normal}. 

\begin{lemma}\label{lem:regepicoeq}
    Let $f:A\to B$ be a morphism in $\Hopf$. Then, $\phi_{A}(\mathsf{hker}(f))=\mathsf{hcoker}(\mathsf{hker}(f))$.
\end{lemma}

\begin{proof}
By $3)\Rightarrow 1)\Rightarrow2)$ of Theorem \ref{cor:normalobjects} we know that $\pi:=\phi_{A}(\mathsf{hker}(f))$ is a morphism in $\Hopf$. Set $\zeta:=m_{A}(\id_{A}\ot\mathsf{hker}(f))-\id_{A}\ot\varepsilon_{\mathsf{Hker}(f)}$ so that $\pi=\mathsf{coker}(\zeta)$. Since $\mathsf{coker}(\zeta)$ is a morphism in $\Hopf$, we have that $\mathsf{ker}(\mathsf{coker}(\zeta)):\mathrm{Im}(\zeta)\to A$ is also a morphism in $\Mm_{A}$ and, since and $\mu_{\mathrm{Im}(\zeta)}$ and $\mathsf{coker}(\mathsf{ker}(\zeta))\ot\id_{A}$ are epimorphisms in $\Mm$, we get
\[
\begin{split}
    \mathsf{hcoker}(\mathsf{hker}(f))&=\mathsf{coker}(m_{A}(m_{A}\ot\mathrm{Id}_{A})(\mathrm{Id}_{A}\ot \mathsf{hker}(f)\ot\mathrm{Id}_{A})-m_{A}(\mathrm{Id}_{A}\ot\varepsilon_{\mathsf{Hker}(f)}\ot\mathrm{Id}_{A}))
    \\&=\mathsf{coker}(m_{A}(\zeta\ot\id_{A}))=\mathsf{coker}(m_{A}(\mathsf{ker}(\mathsf{coker}(\zeta))\ot\id_{A})(\mathsf{coker}(\mathsf{ker}(\zeta))\ot\id_{A}))\\&=\mathsf{coker}(m_{A}(\mathsf{ker}(\mathsf{coker}(\zeta))\ot\id_{A}))=\mathsf{coker}(\mathsf{ker}(\mathsf{coker}(\zeta))\mu_{\mathrm{Im}(\zeta)})\\&=\mathsf{coker}(\mathsf{ker}(\mathsf{coker}(\zeta)))=\mathsf{coker}(\zeta)=\phi_{A}(\mathsf{hker}(f)),
\end{split}
\]
i.e.\ $\phi_{A}(\mathsf{hker}(f))=\mathsf{hcoker}(\mathsf{hker}(f))$. 
\end{proof}

As a consequence, recalling that regular epimorphisms in $\Hopf$ are exactly those morphisms $f$ such that $f=\mathsf{hcoker}(\mathsf{hker}(f))$ (see Corollary \ref{coeqequalcoker}), we obtain the following two results:

\begin{corollary}\label{cor:bijectivecorrespondence}
    Let $A$ be an object in $\Hopf$. The bijective correspondence given in Theorem \ref{thm:NewmanforM} restricts to a bijective correspondence between subobjects of $A$ which are kernels in $\Hopf$ and quotients of $A$ in $\mathsf{Comon}_{\mathrm{coc}}(_{A}\Mm)$ which are regular epimorphisms in $\mathsf{Hopf}_{\mathrm{coc}}(\Mm)$. 
\end{corollary}

\begin{corollary}
    Regular epimorphisms (equivalently, cokernels) in $\Hopf$ are coequalizers as in \eqref{picoequalizer}. 
\end{corollary}

\begin{proof}
     Since $f=\mathsf{hcoker}(\mathsf{hker}(f))$, by Lemma \ref{lem:regepicoeq} we get $f=\phi_{A}(\mathsf{hker}(f))$ and so $f$ is a coequalizer as in \eqref{picoequalizer} by 2) of Proposition \ref{prop:definitionphi}.
\end{proof}

As we will see in the next section, in order to have a regular epi-mono factorization for any morphism in $\Hopf$ we will need that all the morphisms in $\Hopf$ that are epimorphisms (equivalently, cokernels) in $\Mm$ are coequalizers as in \eqref{picoequalizer}, not just the cokernels in $\Hopf$.

\section{Regularity of $\Hopf$}\label{sec:regularity}

Now, we denote by $(\Mm,\ot,\mathbf{1},\sigma)$ an abelian symmetric monoidal category. The aim of this section is to prove that $\Hopf$ is regular by using Theorem \ref{thm:NewmanforM}. By Proposition \ref{prop:Hopfcocfintcomplete}, we know that $\Hopf$ is finitely complete. It suffices to show that any morphism in $\Hopf$ factorizes as a regular epimorphism followed by a monomorphism and regular epimorphisms in $\Hopf$ are stable under pullbacks. For this purpose, we need that the category $\Hopf$ satisfies some special properties.

\subsection{Factorization of morphisms in $\Hopf$}

First, we consider the regular epi-mono factorization for morphisms in $\Hopf$.  The following result shows that any morphism $f$ in $\Hopf$ such that $\mathsf{coker}(\mathsf{ker}(f))$ is a coequalizer as in \eqref{picoequalizer} can be factorized as expected. Note that regular epimorphisms in $\Hopf$ coincide with cokernels, see Corollary \ref{coeqequalcoker}.

\begin{proposition}\label{prop:factorizationmorphisms}
Let $f$ be a morphism in $\Hopf$ such that $\mathsf{coker}(\mathsf{ker}(f))$ is a coequalizer in $\Mm$ as in \eqref{picoequalizer}. Then, the morphism $f$ factorizes as a regular epimorphism in $\Hopf$ followed by a monomorphism in $\Hopf$. More specifically, $f = i \pi$ where $\pi:=\mathsf{hcoker}(\mathsf{hker}(f))=\mathsf{coker}(\mathsf{ker}(f))$ and $i=\mathsf{ker}(\mathsf{coker}(f))$.

\end{proposition}

\begin{proof}
We already know that $\Hopf$ is pointed (Lemma \ref{lem:Hopfpointed}), finitely complete (Proposition \ref{prop:Hopfcocfintcomplete}) and has coequalizers (Proposition \ref{prop:coequalizerHopf}). In particular, any morphism in $\Hopf$ has kernel and cokernel. For a morphism $f:A\to B$ in $\Hopf$, we consider $\mathsf{hker}(f):\mathsf{HKer}(f)\to A$ in $\Hopf$ and $\pi:=\mathsf{hcoker}(\mathsf{hker}(f)):A\to \mathsf{HCoker}(\mathsf{hker}(f))$ in $\Hopf$. By the universal property of cokernel, there is a unique morphism $i:\mathsf{HCoker}(\mathsf{hker}(f))\to B$ in $\Hopf$ such that the following diagram commutes:
\begin{equation}\label{factorizationdiagram}
\begin{tikzcd}
	\mathsf{HKer}(f) & A && B\\
	&& \mathsf{HCoker}(\mathsf{hker}(f))
	\arrow[from=1-1, to=1-2, "\mathsf{hker}(f)"]
	\arrow[from=1-2, to=1-4, "f"]
	\arrow[from=1-2, to=2-3, "\pi"']
	\arrow[from=2-3, to=1-4, "i"']
\end{tikzcd}
\end{equation}
The morphism $\pi$ is a regular epimorphism in $\Hopf$ by definition. It remains to show that $i$ is a monomorphism in $\Hopf$. 
   
Consider $\mathsf{ker}(f):\mathsf{Ker}(f)\to A$, i.e. the kernel of $f$ in $\Mm$, and its cokernel $\pi':=\mathsf{coker}(\mathsf{ker}(f)):A\to\mathsf{Coker}(\mathsf{ker}(f))$ in $\Mm$. 
By Lemma \ref{lem:regepicoeq}, we know that $\mathsf{hcoker}(\mathsf{hker}(f))=\phi_{A}(\mathsf{hker}(f))$. Since $\pi'$ is a morphism in $\Hopf$ (Remark \ref{rmk:kerHopfideal}), we have
\[
\mathsf{hcoker}(\mathsf{hker}(f))=\phi_{A}(\mathsf{hker}(f))\overset{\eqref{auxeq}}{=}\phi_{A}(\mathsf{hker}(\mathsf{coker}(\mathsf{ker}(f))))=\phi_{A}(\mathsf{hker}(\pi'))=\phi_{A}\psi_{A}(\pi')
\]
using 1) of Lemma \ref{lem:inducedbijectivecorr} for the last equality. Since $\pi'$ is a coequalizer as in \eqref{picoequalizer}, we obtain $\mathsf{hcoker}(\mathsf{hker}(f))=\mathsf{coker}(\mathsf{ker}(f))$ by Theorem \ref{thm:NewmanforM}. Hence, using the image factorization of $f$ in $\Mm$, we get that $i=\mathsf{ker}(\mathsf{coker}(f))$. This means the morphism $i$ is a monomorphism in $\Mm$. Since $i$ is a morphism in $\Hopf$, we get that $i$ is a monomorphism in $\Hopf$.
\end{proof}

\begin{remark}
    As proven in Proposition \ref{prop:factorizationmorphisms}, for a morphism $f$ in $\Hopf$ such that $\mathsf{coker}(\mathsf{ker}(f))$ is coequalizer as in \eqref{picoequalizer}, the regular epimorphism-monomorphism factorization of $f$ in $\Hopf$ is provided by the image factorization of $f$ in the abelian category $\Mm$. This means the factorization is unique up to isomorphism. 
\end{remark}

\begin{remark}\label{rmk:remregepi}
By Corollary \ref{coeqequalcoker}, we know that regular epimorphisms in $\Hopf$ coincide with cokernels. We also point out that, if $\Cc$ is an homological category, the regular-epi mono factorization is obtained as $f=i\mathsf{coker}(\mathsf{ker}(f))$, see e.g.\ \cite[Proposition 4.1.2]{BB}. This supports our result for $\Hopf$. 

Besides, observe that, for a morphism $f$ in $\Hopf$, the morphism $\mathsf{coker}(\mathsf{ker}(f))$ is in $\Hopf$ (Remark \ref{rmk:kerHopfideal}), and an epimorphism in $\Mm$. Hence, to have that for any morphism in $\Hopf$ the factorization \eqref{factorizationdiagram} coincides with the image factorization in $\Mm$, we need that any morphism in $\Hopf$ which is an epimorphism in $\Mm$ is a coequalizer as in \eqref{picoequalizer}.

\end{remark}

Inspired by the previous remark, we prove the following result. 

\begin{invisible}
\rd{[Since we already have Proposition \ref{prop:cofaitfullyflat} proving the same, we do not need this. I keep it here in case it might be useful in other parts.]}

\begin{proposition}\label{prop:differentproof}
Let $p:A\to B$ be a morphism in $\mathsf{Hopf}(\Mm)$ which is an epimorphism in $\Mm$. Suppose that the functor $A\square_{B}(-)$ preserves and reflects coequalizers. Then, the following diagram
\begin{equation}\label{doublecotensor}
\begin{tikzcd}
	A\square_{B}A\square_{B}A &&& A\square_{B}A && A
	\arrow[shift left, from=1-1, to=1-4, "\id_{A}\square_{B}(\varepsilon_{A}\ot\id_{A})e_{A,A}"]
	\arrow[shift right, from=1-1, to=1-4,"\id_{A}\square_{B}(\id_{A}\ot\varepsilon_{A})e_{A,A}"']
	\arrow[from=1-4, to=1-6,"(\id_{A}\ot\varepsilon_{A})e_{A,A}"]
\end{tikzcd}
\end{equation}
is a split coequalizer in $\Mm$. As a consequence, $p$ is a coequalizer as in \eqref{picoequalizer}.
\end{proposition}

\begin{proof}

As recalled in the preliminaries, the category $(^{B}\Mm^{B},\square_{B},B)$ is monoidal. Since $p$ is a morphism in $\mathsf{Comon}(\Mm)$, $A$ is an object in $^{B}\Mm^{B}$ with structures $(p\ot\id_{A})\Delta_{A}$ and $(\id_{A}\ot p)\Delta_{A}$, so that $p$ becomes a morphism in $^{B}\Mm^{B}$. The object $A\square_{B}A$ is in $^{B}\Mm^{B}$ with comodule strucures defined such that $e_{A,A}$ becomes a morphism in $^{B}\Mm^{B}$. Hence, $(\varepsilon_{A}\ot\id_{A})e_{A,A}$ and $(\id_{A}\ot\varepsilon_{A})e_{A,A}$ are morphisms in $^{B}\Mm^{B}$, given $\mathbf{1}$ in $^{B}\Mm^{B}$ with trivial coactions determined $u_{B}$.

We prove that \eqref{doublecotensor} is a split coequalizer. First, we have
\[
\begin{split}
(\id_{A}\ot\varepsilon_{A})e_{A,A}(\id_{A}\square_{B}(\varepsilon_{A}\ot\id_{A})e_{A,A})&=(\id_{A}\ot\varepsilon_{A})(\id_{A}\ot(\varepsilon_{A}\ot\id_{A})e_{A,A})e_{A,A\square_{B}A}\\&=(\id_{A}\ot\varepsilon_{A})(\id_{A}\ot(\id_{A}\ot\varepsilon_{A})e_{A,A})e_{A,A\square_{B}A}\\&=(\id_{A}\ot\varepsilon_{A})e_{A,A}(\id_{A}\square_{B}(\id_{A}\ot\varepsilon_{A})e_{A,A})
\end{split}
\]
Moreover, there is a morphism $\Delta'_{A}:A\to A\square_{B}A$ in $\Mm$ such that $e_{A,A}\Delta'_{A}=\Delta_{A}$, hence $(\id_{A}\ot\varepsilon_{A})e_{A,A}\Delta'_{A}=\id_{A}$. Given $\Delta'_{A}\square_{B}\id_{A}:A\square_{B}A\to A\square_{B}A\square_{B}A$ one have that
\[
(\id_{A}\square_{B}(\varepsilon_{A}\ot\id_{A})e_{A,A})(\Delta'_{A}\square_{B}\id_{A})=((\id_{A}\ot\varepsilon_{A})e_{A,A}\square_{B}\id_{A})(\Delta'_{A}\square_{B}\id_{A})=\id
\]
and $(\id_{A}\square_{B}(\id_{A}\ot\varepsilon_{A})e_{A,A})(\Delta'_{A}\square_{B}\id_{A})=(\Delta'_{A}\ot\varepsilon_{A})e_{A,A}$, so that \eqref{doublecotensor} is a split coequalizer \rd{[To add details]}. In particular, \eqref{doublecotensor} is a coequalizer. 

Now, given the canonical isomorphism $\Lambda_{A}:A\to A\square_{B}B$ in $\Mm$ determined by $e_{A,B}\Lambda_{A}=\rho_{A}=(\id_{A}\ot p)\Delta_{A}$, we compute
\[
\begin{split}
e_{A,B}\Lambda_{A}(\id_{A}\ot\varepsilon_{A})e_{A,A}&=(\id_{A}\ot p)\Delta_{A}(\id_{A}\ot\varepsilon_{A})e_{A,A}\\&= (\id_{A\ot B}\ot\varepsilon_{A})(\id_{A}\ot p\ot\id_{A})(\Delta_{A}\ot\id_{A})e_{A,A}\\&=(\id_{A\ot B}\ot\varepsilon_{A})(\id_{A}\ot p\ot\id_{A})(\id_{A}\ot\Delta_{A})e_{A,A}\\&= (\id_{A}\ot p)e_{A,A}=e_{A,B}(\id_{A}\square_{B} p).
\end{split}
\]
Hence, since $e_{A,B}$ is a monomorphism in $\Mm$, we get $\Lambda_{A}(\id_{A}\ot\varepsilon_{A})e_{A,A}=\id_{A}\square_{B}p$. Therefore, since $\Lambda_{A}$ is an isomorphism in $\Mm$ and \eqref{doublecotensor} is a coequalizer in $\Mm$, we get that $\id_{A}\square_{B}p$ is the coequalizer of the pair of morphisms $(\id_{A}\square_{B}(\varepsilon_{A}\ot\id_{A})e_{A,A}),\id_{A}\square_{B}(\id_{A}\ot\varepsilon_{A})e_{A,A})$ in $\Mm$. Since $A\square_{B}(-)$ reflects coequalizers, $p$ is the coequalizer of the pair $((\varepsilon_{A}\ot\id_{A})e_{A,A},(\id_{A}\ot\varepsilon_{A})e_{A,A})$ in $\Mm$.
\end{proof}
\end{invisible}

\begin{proposition}\label{prop:cofaitfullyflat}
Let $\pi:A \to Q$ be a morphism in $\mathsf{Comon}(\Mm)$ which is an epimorphism in $\Mm$. Suppose that the functor $(-)\square_{Q}A$ preserves and reflects epimorphisms. Then, $\id_{\mathsf{Ker}(\pi)}\square_{Q}\pi$ is an epimorphism in $\Mm$. As a consequence, $\pi$ is a coequalizer as in \eqref{picoequalizer}.

\begin{invisible}
In particular, this happens when $\pi$ is a morphism in $\Hopf$ which is an epimorphism in $\Mm$ and such that $(-)\square_{Q}A$ preserves and reflects epimorphisms.
\end{invisible}
\end{proposition}

\begin{proof}
Since $\mathsf{ker}(\pi)\ot\mathrm{Id}_{Q}$ is the kernel of the morphism $\pi\ot\mathrm{Id}_{Q}$ in $\Mm$ and $(\pi\ot\id_{Q})(\id_{A}\ot\pi)\Delta_{A}\mathsf{ker}(\pi)=\Delta_{Q}\pi\mathsf{ker}(\pi)=0$, there is a unique morphism $\rho_{\mathsf{Ker}(\pi)}:\mathsf{Ker}(\pi)\to\mathsf{Ker}(\pi)\ot Q$ in $\Mm$ such that $(\mathsf{ker}(\pi)\ot\id_{Q})\rho_{\mathsf{Ker}(\pi)}=(\id_{A}\ot\pi)\Delta_{A}\mathsf{ker}(\pi)$. One can check that $\rho_{\mathsf{Ker}(\pi)}$ is a right $Q$-coaction on $\mathsf{Ker}(\pi)$, so that $\mathsf{ker}(\pi)$ is a morphism in $\Mm^{Q}$ by considering $A$ in $\Mm^{Q}$ with $(\id_{A}\ot\pi)\Delta_{A}$. Moreover, $A$ is an object in $^{Q}\Mm^{Q}$ with left comodule structure $(\pi\ot\id_{A})\Delta_{A}$, so that $\pi$ becomes a morphism in $^{Q}\Mm^{Q}$. Since $(^{Q}\Mm^{Q},\square_{Q},Q)$ is a monoidal category, the morphisms $\id_{A}\square_{Q}\pi$ and $\pi\square_{Q}\id_{A}$ are in $^{Q}\Mm^{Q}$.

Let $\Lambda'_{A}:A\to Q\square_{Q} A$ be the canonical isomorphism in $\Mm$ determined by $e_{Q,A}\Lambda'_{A}= \lambda_{A} = (\pi\ot\id_{A})\Delta_A$. There is a morphism $\Delta'_{A}:A\to A\square_{Q} A$ in $\Mm$ such that $e_{A,A}\Delta'_{A}=\Delta_A$. Consider $(\Lambda'_{A})^{-1}(\pi\square_{Q} \id_{A}):A\square_{Q} A\to A$, where $(\Lambda'_{A})^{-1}=(\varepsilon_{Q}\ot\id_{A})e_{Q,A}$. It is easy to check that $(\Lambda'_{A})^{-1} (\pi\square_{Q}\id_{A})\Delta'_A= \id_A$. Hence, $\pi\square_{Q}\id_{A}$ is a split epimorphism in $\Mm$. Therefore, $\id_{\mathsf{Ker}(\pi)}\square_{Q}\pi\square_{Q}\id_{A}$ is a split epimorphism in $\Mm$ and then, since $(-)\square_{Q}A$ reflects epimorphisms, we get that $\id_{\mathsf{Ker}(\pi)}\square_{Q}\pi$ is an epimorphism in $\Mm$.

We now show that $\pi$ is a coequalizer as in \eqref{picoequalizer}. Since
\[
\begin{split}
\pi(\varepsilon_{A}\ot\id_{A})e_{A,A}&=(\varepsilon_{A}\ot\id_{Q}\ot\varepsilon_{A})(\id_{A}\ot(\pi\ot\id_{A})\Delta_{A})e_{A,A}\\&=(\varepsilon_{A}\ot\id_{Q}\ot\varepsilon_{A})((\id_{A}\ot \pi)\Delta_{A}\ot\id_{A})e_{A,A}=\pi(\id_{A}\ot\varepsilon_{A})e_{A,A},
\end{split}
\]
i.e.\ $\pi$ coequalizes the pair $((\varepsilon_{A}\ot\id_{A})e_{A,A},(\id_{A}\ot\varepsilon_{A})e_{A,A})$, it remains to prove the universal property. Suppose there is a morphism $f:A\to Z$ in $\Mm$ such that $f(\varepsilon_{A}\ot\id_{A})e_{A,A}=f(\id_{A}\ot\varepsilon_{A})e_{A,A}$.  
We have
\[
\begin{split}
f\mathsf{ker}(\pi)(\mathrm{Id}_{\mathsf{Ker}(\pi)}\ot \varepsilon_{A})e_{\mathsf{Ker}(\pi),A}&=
f (\mathrm{Id}_{A}\ot \varepsilon_{A})(\mathsf{ker}(\pi)\ot \id_A)e_{\mathsf{Ker}(\pi),A}
\\&=f(\id_{A}\ot\varepsilon_{A})e_{A,A}(\mathsf{ker}(\pi)\square_{Q}\id_{A})\\&=f(\varepsilon_{A}\ot\id_{A})e_{A,A}(\mathsf{ker}(\pi)\square_{Q}\id_{A})\\&=f( \varepsilon_{A}\mathsf{ker}(\pi)\ot\mathrm{Id}_{A})e_{\mathsf{Ker}(\pi),A}\\&=f(\varepsilon_Q\pi\mathsf{ker}(\pi) \ot\mathrm{Id}_{A})e_{\mathsf{Ker}(\pi),A}=0.
\end{split}
\]
Because
\[
(\id_{\mathsf{Ker}(\pi)} \ot \varepsilon_{A})e_{\mathsf{Ker}(\pi),A} = (\id_{\mathsf{Ker}(\pi)} \ot \varepsilon_{Q}\pi)e_{\mathsf{Ker}(\pi),A} =  (\id_{\mathsf{Ker}(\pi)} \ot \varepsilon_{Q})e_{\mathsf{Ker}(\pi),Q}(\mathrm{Id}_{\mathsf{Ker}(\pi)} \square_{Q}\pi),
\]
and, as recalled in the preliminaries, the morphism $(\id_{\mathsf{Ker}(\pi)} \otimes \varepsilon_Q)e_{\mathsf{Ker}(\pi),Q}$ is an isomorphism in $\Mm$, we obtain $(\id_{\mathsf{Ker}(\pi)} \ot \varepsilon_{A})e_{\mathsf{Ker}(\pi),A}$ is an epimorphism in $\Mm$. Consequently, the previous calculation implies $f\mathsf{ker}(\pi)=0$. By the universal property of $\pi=\mathsf{coker}(\mathsf{ker}(\pi))$, there exists a unique morphism $\varphi:Q\to Z$ in $\Mm$ such that $\varphi\pi=f$.
\end{proof}

We now introduce the following definition:

\begin{definition}\label{deffaithcoflat}
    Let $(\Mm,\ot,\mathbf{1},\sigma)$ be an abelian symmetric monoidal category. We say that $(\Mm,\ot,\mathbf{1},\sigma)$ satisfies the ``faithful coflatness condition'' if, for any object $A$ in $\Hopf$ and any morphism $\pi:A\to Q$ in $\mathsf{Comon}_{\mathrm{coc}}(_{A}\Mm)$ which is an epimorphism in $\Mm$, $A$ is \textit{faithfully coflat} over $Q$, i.e.\ 
$(-)\square_{Q} A$ preserves and reflects epimorphisms. 
\end{definition}

\begin{remark}\label{rmk:prototypefaithcoflat}
    The prototype of this condition is given by $(\Mm,\ot,\mathbf{1},\sigma)=(\mathsf{Vec}_{\Bbbk},\ot_{\Bbbk},\Bbbk,\tau)$. It is known that, for any $A$ in $\mathsf{Hopf}_{\Bbbk,\mathrm{coc}}$ and any quotient morphism $\pi:A\to Q$ in $\mathsf{Comon}_{\mathrm{coc}}(_{A}\mm)$, $A$ is faithfully coflat over $Q$; this was proven in \cite[Theorem 1.3 (2)]{Masuoka}. 
\end{remark}

Under the cocommutativity assumption, $(-)\square_{Q}A$ preserves and reflects epimorphisms if and only if $A\square_{Q}(-)$ does.

\begin{lemma}\label{lemm:piididpi}
    Given a morphism $\pi:A\to Q$ in $\mathsf{Comon}_{\mathrm{coc}}(\Mm)$, we have that $\id_{A}\square_{Q}\pi$ is an epimorphism in $\Mm$ if and only if $\pi\square_{Q}\id_{A}$ is an epimorphism in $\Mm$.

    As a consequence, $A\square_{Q}(-)$ preserves and reflects epimorphisms if and only if $(-)\square_{Q}A$ preserves and reflects epimorphisms.
\end{lemma}

\begin{proof}
Since $A$ is cocommutative, we have
\[
\begin{split}
((\id_{A}\ot\pi)\Delta_{A}\ot\id_{A})e_{A,A}&=(\id_{A}\ot(\pi\ot\id_{A})\Delta_{A})e_{A,A}=(\id_{A}\ot(\pi\ot\id_{A})\sigma_{A,A}\Delta_{A})e_{A,A}\\&=(\id_{A}\ot\sigma_{A,Q})(\id_{A}\ot(\id_{A}\ot\pi)\Delta_{A})e_{A,A},
\end{split}
\]
and hence we obtain
\[
\begin{split}
    &(\sigma_{A,Q}\ot\id_{A})((\id_{A}\ot\pi)\Delta_{A}\ot\id_{A})\sigma_{A,A}e_{A,A}\\&=(\sigma_{A,Q}\ot\id_{A})\sigma_{A,A\ot Q}(\id_{A}\ot(\id_{A}\ot\pi)\Delta_{A})e_{A,A}\\&=(\sigma_{A,Q}\ot\id_{A})(\id_{A}\ot\sigma_{A,Q})(\sigma_{A,A}\ot\id_{Q})(\id_{A}\ot(\id_{A}\ot\pi)\Delta_{A})e_{A,A}\\&=(\id_{Q}\ot\sigma_{A,A})(\sigma_{A,Q}\ot\id_{A})(\id_{A}\ot\sigma_{A,Q})(\id_{A}\ot(\id_{A}\ot\pi)\Delta_{A})e_{A,A}\\&=(\id_{Q}\ot\sigma_{A,A})(\sigma_{A,Q}\ot\id_{A})((\id_{A}\ot\pi)\Delta_{A}\ot\id_{A})e_{A,A}\\&=(\id_{Q}\ot\sigma_{A,A})((\pi\ot\id_{A})\sigma_{A,A}\Delta_{A}\ot\id_{A})e_{A,A}\\&=(\id_{Q}\ot\sigma_{A,A})((\pi\ot\id_{A})\Delta_{A}\ot\id_{A})e_{A,A}.
\end{split}
\]
Therefore, 
\[
\begin{split}
((\id_{A}\ot\pi)\Delta_{A}\ot\id_{A})\sigma_{A,A}e_{A,A}&=(\sigma_{Q,A}\ot\id_{A})(\id_{Q}\ot\sigma_{A,A})((\pi\ot\id_{A})\Delta_{A}\ot\id_{A})e_{A,A}\\&=\sigma_{Q\ot A,A}((\pi\ot\id_{A})\Delta_{A}\ot\id_{A})e_{A,A}\\&=(\id_{A}\ot(\pi\ot\id_{A})\Delta_{A})\sigma_{A,A}e_{A,A}.
\end{split}
\]
By the universal property of the equalizer, there exists a unique morphism $\xi_{A,A}:A\square_{Q}A\to A\square_{Q}A$ in $\Mm$ such that $e_{A,A}\xi_{A,A}=\sigma_{A,A}e_{A,A}$. Similarly, there exists a unique morphism $\xi_{Q,A}:Q\square_{Q}A\to A\square_{Q}Q$ in $\Mm$ such that $e_{A,Q}\xi_{Q,A}=\sigma_{Q,A}e_{Q,A}$, which in fact can be obtained immediately by setting $\xi_{Q,A} = \Lambda_A (\Lambda_A')^{-1}$. \begin{invisible}
    In fact, $\xi$ is a natural transformation. 
\end{invisible}
More precisely, 
\begin{align*}
e_{A,Q}\xi_{Q,A} = &e_{A,Q} \Lambda_A (\Lambda_A')^{-1} = (\id_{A} \ot \pi) \Delta_{A} (\varepsilon_Q \ot \id_A) e_{Q,A} = (\varepsilon_Q \ot (\id_{A} \ot \pi)\sigma_{A,A}\Delta_{A}) e_{Q,A}\\
= &\sigma_{Q,A}(\varepsilon_Q \ot (\pi \ot \id_{A}) \Delta_{A}) e_{Q,A} = \sigma_{Q,A}((\varepsilon_Q \ot \id_{Q})\Delta_Q \ot \id_{A}) e_{Q,A} =\sigma_{Q,A} e_{Q,A}.
\end{align*}
Therefore, we have
\[
\begin{split}
e_{A,Q}\xi_{Q,A}(\pi\square_{Q}\id_{A})&=\sigma_{Q,A}e_{Q,A}(\pi\square_{Q}\id_{A})=\sigma_{Q,A}(\pi\ot\id_{A})e_{A,A}=(\id_{A}\ot\pi)\sigma_{A,A}e_{A,A}\\&=(\id_{A}\ot\pi)e_{A,A}\xi_{A,A}=e_{A,Q}(\id_{A}\square_{Q}\pi)\xi_{A,A}.
\end{split}
\]
Since $e_{A,Q}$ is a monomorphism in $\Mm$, we get $\xi_{Q,A}(\pi\square_{Q}\id_{A})=(\id_{A}\square_{Q}\pi)\xi_{A,A}$. Since $\sigma_{Q,A}$ and $\sigma_{A,A}$ are isomorphisms in $\Mm$, one can check that $\xi_{Q,A}$ and $\xi_{A,A}$ are isomorphisms in $\Mm$, so $\id_{A}\square_{Q}\pi$ is an epimorphism in $\Mm$ if and only if $\pi\square_{Q}\id_{A}$ is an epimorphism in $\Mm$.
\end{proof}

From now on, let $(\Mm,\ot,\mathbf{1},\sigma)$ be an abelian symmetric monoidal category that satisfies the ``faithful coflatness condition'', see Definition \ref{deffaithcoflat}. With the faithful coflatness condition, we are able to show that any morphism in $\Hopf$ factorizes as in Proposition \ref{prop:factorizationmorphisms}.

\begin{proposition}\label{prop:factmorp2}
   Every morphism $f:A\to B$ in $\Hopf$ satisfies that $\mathsf{coker}(\mathsf{ker}(f))$ is a coequalizer as in \eqref{picoequalizer}. 
   
   As a consequence, $\phi_{A}(\mathsf{hker}(f))=\mathsf{hcoker}(\mathsf{hker}(f))=\mathsf{coker}(\mathsf{ker}(f))$ and $f$ factorizes as a regular epimorphism (in fact, a cokernel) in $\Hopf$ followed by a monomorphism in $\Hopf$.
\end{proposition}

\begin{proof}
    We know that $\pi:=\mathsf{coker}(\mathsf{ker}(f)):A\to\mathrm{Im}(f)$ is a morphism in $\mathsf{Comon}_{\mathrm{coc}}(_{A}\Mm)$ (Remark \ref{rmk:kerHopfideal}), and an epimorphism in $\Mm$. Hence, by assumption on $\Mm$, the functor $(-)\square_{\mathrm{Im}(f)}A$ preserves and reflects epimorphisms. Therefore, by Proposition \ref{prop:cofaitfullyflat}, we get that $\pi$ is a coequalizer as in \eqref{picoequalizer}. 
    \begin{invisible}
    By 2) of Lemma \ref{lem:inducedbijectivecorr}, we have $\mathsf{coker}(\mathsf{ker}(f))=\phi_{A}(\mathsf{hker}(f))$ and then, by Lemma \ref{lem:regepicoeq}, we obtain $\mathsf{coker}(\mathsf{ker}(f))=\mathsf{hcoker}(\mathsf{hker}(f))$. Therefore, we have $f=i\pi$, where $\pi$ is a regular epimorphism in $\Hopf$ and $i$ is a monomorphism in $\Hopf$.
    \end{invisible}
    By Proposition \ref{prop:factorizationmorphisms}, we can conclude.
\end{proof}


By Proposition \ref{prop:factmorp2}, we can describe regular epimorphisms and monomorphisms in $\Hopf$. To this end, we observe that a morphism in $\Hopf$ is a monomorphism if and only if its kernel in $\Hopf$ is the zero morphism in $\Hopf$.

\begin{lemma}\label{lem:monohker}
Let $f:A\to B$ be a morphism in $\Hopf$. Then, $f$ is a monomorphism in $\Hopf$ if and only if $\mathsf{hker}(f)$ is the zero morphism in $\Hopf$ i.e.\ $\mathsf{hker}(f)=u_{A}\varepsilon_{\mathsf{Hker}(f)}$. 
\end{lemma}

\begin{proof}
Suppose $f$ is a monomorphism in $\Hopf$. Since $f\mathsf{hker}(f)=u_{B}\varepsilon_{\mathsf{Hker}(f)}=fu_{A}\varepsilon_{A}\mathsf{hker}(f)$, we obtain $\mathsf{hker}(f)=u_{A}\varepsilon_{A}\mathsf{hker}(f)=u_{A}\varepsilon_{\mathsf{Hker}(f)}$. Conversely, if $f$ has zero kernel in $\Hopf$, then $\pi:=\mathsf{hcoker}(\mathsf{hker}(f))=\mathsf{hcoker}(u_{A}\varepsilon_{\mathsf{Hker}(f)})=\mathrm{Id}_{A}$, see Corollary \ref{cor:cokernel}. Hence, since $f=i\pi$ by Proposition \ref{prop:factmorp2}, we get that $f=i$ is a monomorphism in $\Hopf$. 
\end{proof}

As a consequence, we obtain:

\begin{corollary}\label{cor:regepimono}
    The following facts hold in $\Hopf$:
\begin{itemize}
    \item[1)] regular epimorphisms in $\Hopf$ (equivalently, cokernels in $\Hopf$) coincide with the morphisms in $\Hopf$ that are epimorphisms in $\Mm$;
    \item[2)] monomorphisms in $\Hopf$ are exactly the morphisms in $\Hopf$ that are monomorphisms in $\Mm$.
\end{itemize}
\end{corollary}

\begin{proof}
1). On one hand, the coequalizer in $\Hopf$ is a morphism in $\Hopf$ which is an epimorphism in $\Mm$, see Section \ref{sec:colimits}. On the other hand, for a morphism $f$ in $\Hopf$ which is an epimorphism in $\Mm$,  the factorization $f=i\pi$ given in \eqref{factorizationdiagram} implies that $i$ is an isomorphism in $\Mm$. Thus, $f$ is a coequalizer in $\Hopf$.

2). Clearly, a morphism in $\Hopf$ which is a monomorphism in $\Mm$ is a monomorphism in $\Hopf$. Conversely, for a monomorphism $f:A\to B$ in $\Hopf$, by Lemma \ref{lem:monohker}, we have $\mathsf{hker}(f)=u_{A}\varepsilon_{\mathsf{Hker}(f)}$. It follows that $\mathsf{coker}(\mathsf{ker}(f))=\mathsf{hcoker}(\mathsf{hker}(f))=\mathsf{hcoker}(u_{A}\varepsilon_{\mathsf{Hker}(f)})=\mathrm{Id}_{A}$ and then $f=\mathsf{ker}(\mathsf{coker}(f))$. Thus, $f$ is a monomorphism in $\Mm$.
\end{proof}

\subsection{Stability of regular epimorphisms along pullbacks}

To obtain that the finitely complete category $\Hopf$ is regular, it remains to prove that regular epimorphisms in $\Hopf$ are stable under pullbacks along any morphism in $\Hopf$. In this subsection, we let $(\Mm,\ot,\mathbf{1},\sigma)$ be an abelian symmetric monoidal category that satisfies the ``faithful coflatness condition'', see Definition \ref{deffaithcoflat}.

By \cite[Lemma 2.1]{GSV}, the fact that regular epimorphisms in $\Hopf$ are stable under pullbacks along any morphism in $\Hopf$ is equivalent to the following facts:
\begin{itemize}
    \item[1)] given any regular epimorphism $f:A\to B$ and any object $E$ in $\Hopf$, the induced arrow $\mathrm{Id}_{E}\times f:E\times A\to E\times B$ is a regular epimorphism in $\Hopf$;
    \item[2)] regular epimorphisms in $\Hopf$ are stable under pullbacks along split monomorphisms in $\Hopf$.
\end{itemize}

The first condition is clearly satisfied, as the following result shows:

\begin{lemma}\label{lem:regepipreservedbybinary}
     Let $f:A\to B$ be a regular epimorphism in $\Hopf$ and $E$ be an object in $\Hopf$. Then, the induced arrow $\mathrm{Id}_{E}\times f:E\times A\to E\times B$ is a regular epimorphism in $\Hopf$. 

\end{lemma}

\begin{proof}

Recall that the binary products $E\times A$ and $E\times B$ are given by $E\otimes A$ and $E\otimes B$, respectively, and the induced arrow $\mathrm{Id}_{E}\times f$ is $\mathrm{Id}_{E}\otimes f$, see Subsection \ref{subsect:binprod}. By 1) of Corollary \ref{cor:regepimono}, we know that $f$ is is an epimorphism in $\Mm$. Consequently, $\mathrm{Id}_{E}\ot f$ is an epimorphism in $\Mm$. Hence, $\mathrm{Id}_{E}\ot f$ is a regular epimorphism in $\Hopf$ by 1) of Corollary \ref{cor:regepimono}.
\end{proof}

Therefore, it remains to prove that regular epimorphisms in $\Hopf$ are stable under pullbacks along split monomorphisms in $\Hopf$. In order to do this, we first prove the following results.

\begin{lemma}\label{lem:p1p2Hopf}
    Let $p:A\to B$ be a morphism in $\Hopf$ and $i:C\to B$ a monomorphism in $\Hopf$. We denote by $(p^{-1}(C),p_{1},p_{2})$ the following pullback in $\Mm$:
\begin{equation}\label{pullbackpC}
\begin{tikzcd}
	p^{-1}(C) & C\ot A \\
	A & B\ot A
	\arrow[from=1-1, to=1-2,"p_{2}"]
	\arrow[from=1-1, to=2-1, "p_{1}"']
	\arrow["\lrcorner"{anchor=center, pos=0.125}, draw=none, from=1-1, to=2-2]
	\arrow[from=1-2, to=2-2, "i\ot\mathrm{Id}_{A}"]
	\arrow[from=2-1, to=2-2, "(p\ot\mathrm{Id}_{A})\Delta_{A}"']
\end{tikzcd}
\end{equation}
Then, $p^{-1}(C)$ is in $\Hopf$ and $p_{1}$ and $p_{2}$ are morphisms in $\Hopf$.   
\end{lemma}

\begin{proof}
    We compute
\[
\begin{split}
(p\ot\mathrm{Id}_{A})\Delta_{A}m_{A}(p_{1}\ot p_{1})&=(p\ot\mathrm{Id}_{A})(m_{A}\ot m_{A})(\mathrm{Id}_{A}\ot\sigma_{A,A}\ot\mathrm{Id}_{A})(\Delta_{A}\ot\Delta_{A})(p_{1}\ot p_{1})\\&=(m_{B}\ot m_{A})(p\ot p\ot\mathrm{Id}_{A\ot A})(\mathrm{Id}_{A}\ot\sigma_{A,A}\ot\mathrm{Id}_{A})(\Delta_{A}\ot\Delta_{A})(p_{1}\ot p_{1})\\&=(m_{B}\ot m_{A})(\mathrm{Id}_{B}\ot\sigma_{A,B}\ot\mathrm{Id}_{A})(p\ot\mathrm{Id}_{A}\ot p\ot\mathrm{Id}_{A})(\Delta_{A}\ot\Delta_{A})(p_{1}\ot p_{1})\\&=(m_{B}\ot m_{A})(\mathrm{Id}_{B}\ot\sigma_{A,B}\ot\mathrm{Id}_{A})(i\ot\mathrm{Id}_{A}\ot i\ot\mathrm{Id}_{A})(p_{2}\ot p_{2})\\&=(m_{B}\ot m_{A})(i\ot i\ot\mathrm{Id}_{A\ot A})(\mathrm{Id}_{C}\ot\sigma_{A,C}\ot\mathrm{Id}_{A})(p_{2}\ot p_{2})\\&=(i\ot\mathrm{Id}_{A})(m_{C}\ot m_{A})(\mathrm{Id}_{C}\ot\sigma_{A,C}\ot\mathrm{Id}_{A})(p_{2}\ot p_{2})\\&=(i\ot\mathrm{Id}_{A})m_{C\ot A}(p_{2}\ot p_{2}),
\end{split}
\]
so, by the universal property of the pullback, there exists a unique morphism $m_{p^{-1}(C)}:p^{-1}(C)\ot p^{-1}(C)\to p^{-1}(C)$ in $\Mm$ such that $p_{1}m_{p^{-1}(C)}=m_{A}(p_{1}\ot p_{1})$ and $p_{2}m_{p^{-1}(C)}=m_{C\ot A}(p_{2}\ot p_{2})$. Moreover, we have
\[
(p\ot\mathrm{Id}_{A})\Delta_{A}u_{A}=(p\ot\mathrm{Id}_{A})(u_{A}\ot u_{A})=u_{B}\ot u_{A}=(i\ot\mathrm{Id}_{A})(u_{C}\ot u_{A}),
\]
so there exists a unique morphism $u_{p^{-1}(C)}:\mathbf{1}\to p^{-1}(C)$ in $\Mm$ such that $u_{A}=p_{1}u_{p^{-1}(C)}$ and $p_{2}u_{p^{-1}(C)}=u_{C\ot A}$. Since $i$ is a monomorphism in $\Mm$, also $i\ot\mathrm{Id}_{A}$ is a monomorphism in $\Mm$ and then also $p_{1}$ is a monomorphism in $\Mm$ as monomorphisms are stable under pullbacks. Then, using that $p_{1}$ is a monomorphism in $\Mm$ and that $(A,m_{A},u_{A})$ is an object in $\mathsf{Mon}(\Mm)$, one can show that $(p^{-1}(C),m_{p^{-1}(C)},u_{p^{-1}(C)})$ is an object in $\mathsf{Mon}(\Mm)$, so that $p_{1}$ and $p_{2}$ are morphisms in $\mathsf{Mon}(\Mm)$. Observe that, since we are considering an abelian monoidal category, we have that $\ot$ preserves binary products. 
\begin{invisible}
Indeed, we have
$A\ot(B\times C)=A\ot (B \oplus C) \cong (A \ot B) \oplus (A \ot C)=(A\ot B)\times(A\ot C)$, for any $A,B,C\in\Mm$, i.e.\ $A\ot(-)$ preserves binary products for any $A\in\Mm$. 
\end{invisible}
As a consequence, since $\ot$ preserves equalizers, we have that $\ot$ also preserves pullbacks. Hence we obtain the following pullback in $\Mm$:
\[\begin{tikzcd}
	p^{-1}(C)\ot A && C\ot A\ot A \\
	A\ot A && B\ot A\ot A
	\arrow[from=1-1, to=1-3,"p_{2}\ot\mathrm{Id}_{A}"]
	\arrow[from=1-1, to=2-1, "p_{1}\ot\mathrm{Id}_{A}"']
	\arrow["\lrcorner"{anchor=center, pos=0.125}, draw=none, from=1-1, to=2-3]
	\arrow[from=1-3, to=2-3, "i\ot\mathrm{Id}_{A\ot A}"]
	\arrow[from=2-1, to=2-3, "(p\ot\mathrm{Id}_{A})\Delta_{A}\ot\mathrm{Id}_{A}"']
\end{tikzcd}\]
Since $((p\ot\mathrm{Id}_{A})\Delta_{A}\ot\mathrm{Id}_{A})\Delta_{A}p_{1}=(p\ot\Delta_{A})\Delta_{A}p_{1}=
(i\ot\mathrm{Id}_{A\ot A})(\mathrm{Id}_{C}\ot\Delta_{A})p_{2}$, there exists a unique morphism $\psi:p^{-1}(C)\to p^{-1}(C)\ot A$ in $\Mm$ such that $(p_{1}\ot\mathrm{Id}_{A})\psi=\Delta_{A}p_{1}$ and $(p_{2}\ot\mathrm{Id}_{A})\psi=(\mathrm{Id}_{C}\ot\Delta_{A})p_{2}$. Moreover, also the following diagram is a pullback in $\Mm$:
\[\begin{tikzcd}
	p^{-1}(C)\ot C \ot A && C\ot A\ot C \ot A \\
	A\ot C \ot A && B\ot A\ot C \ot A
	\arrow[from=1-1, to=1-3,"p_{2}\ot\mathrm{Id}_{C\ot A}"]
	\arrow[from=1-1, to=2-1, "p_{1}\ot\mathrm{Id}_{C\ot A}"']
	\arrow["\lrcorner"{anchor=center, pos=0.125}, draw=none, from=1-1, to=2-3]
	\arrow[from=1-3, to=2-3, "i\ot\mathrm{Id}_{A}\ot\mathrm{Id}_{C\ot A}"]
	\arrow[from=2-1, to=2-3, "(p\ot\mathrm{Id}_{A})\Delta_{A}\ot\mathrm{Id}_{C\ot A}"']
\end{tikzcd}\]
Since $\mathrm{Id}_{B\ot A} \ot i \ot \mathrm{Id}_A$ is a monomorphism in $\Mm$ and
\begin{align*}
&(\mathrm{Id}_{B\ot A} \ot i \ot \mathrm{Id}_A) (i  \ot \mathrm{Id}_A \ot \mathrm{Id}_{C \ot A})\Delta_{C\ot A}p_{2}\\=&(\mathrm{Id}_{B\ot A} \ot i \ot \mathrm{Id}_A) (i  \ot \mathrm{Id}_A \ot \mathrm{Id}_{C \ot A}) (\mathrm{Id}_C \ot \sigma_{C,A} \ot \mathrm{Id}_A) (\Delta_C \ot \Delta_A) p_2\\
= &(\mathrm{Id}_B \ot \sigma_{B,A} \ot \mathrm{Id}_A) (i \ot i \ot \mathrm{Id}_A \ot \mathrm{Id}_A)(\Delta_C \ot \Delta_A) p_2
= (\mathrm{Id}_B \ot \sigma_{B,A} \ot \mathrm{Id}_A) (\Delta_B \ot \Delta_A) (i\ot \mathrm{Id}_A) p_2\\ =& (\mathrm{Id}_B \ot \sigma_{B,A} \ot \mathrm{Id}_A) (\Delta_B \ot \Delta_A) (p\ot\mathrm{Id}_{A})\Delta_{A}p_{1}
= (\mathrm{Id}_B \ot \sigma_{B,A} \ot \mathrm{Id}_A) ( (p\ot p)\Delta_A \ot \Delta_A) \Delta_{A}p_{1}\\
= &(p \ot \mathrm{Id}_A \ot p \ot \mathrm{Id}_A)(\mathrm{Id}_A \ot \sigma_{A,A} \ot \mathrm{Id}_A) (\Delta_A \ot \Delta_A) \Delta_{A}p_{1}
\\
= &(p \ot \mathrm{Id}_A \ot p \ot \mathrm{Id}_A)(\mathrm{Id}_A \ot \sigma_{A,A} \ot \mathrm{Id}_A) (\mathrm{Id}_A \ot \Delta_A \ot \mathrm{Id}_A)(\Delta_A \ot \mathrm{Id}_A) \Delta_{A}p_{1}\\
= &(p \ot \mathrm{Id}_A \ot p \ot \mathrm{Id}_A) (\mathrm{Id}_A \ot \Delta_A \ot \mathrm{Id}_A)(\Delta_A \ot \mathrm{Id}_A) \Delta_{A}p_{1}\\
= &(p \ot \mathrm{Id}_A \ot p \ot \mathrm{Id}_A) (\Delta_A \ot \Delta_A) \Delta_{A}p_{1}
= (p \ot \mathrm{Id}_A \ot p \ot \mathrm{Id}_A) (\Delta_A \ot \Delta_A) \sigma_{A,A} \Delta_{A}p_{1}\\
= &(p \ot \mathrm{Id}_A \ot p \ot \mathrm{Id}_A) (\Delta_A \ot \Delta_A) \sigma_{A,A} (p_{1}\ot \mathrm{Id}_A)\psi
= ((p \ot \mathrm{Id}_A)\Delta_A \ot (p \ot \mathrm{Id}_A)\Delta_A p_1)  \sigma_{p^{-1}(C),A} \psi\\
= &((p \ot \mathrm{Id}_A)\Delta_A \ot (i \ot \mathrm{Id}_A)p_2)  \sigma_{p^{-1}(C),A} \psi\\
= &(\mathrm{Id}_{B\ot A} \ot i \ot \mathrm{Id}_A) ((p \ot \mathrm{Id}_A)\Delta_A \ot \mathrm{Id}_{C \ot A}) (\mathrm{Id}_A \ot p_2)\sigma_{p^{-1}(C),A} \psi,
\end{align*}
we obtain 
\[
(i  \ot \mathrm{Id}_A \ot \mathrm{Id}_{C \ot A})\Delta_{C\ot A} p_2 = ((p \ot \mathrm{Id}_A)\Delta_A \ot \mathrm{Id}_{C \ot A}) (\mathrm{Id}_A \ot p_2)\sigma_{p^{-1}(C),A} \psi.
\]
Therefore, there exists a unique morphism $\psi': p^{-1}(C) \to p^{-1}(C) \ot C \ot A$ in $\Mm$ such that $(p_2 \ot \mathrm{Id}_{C \ot A}) \psi' =\Delta_{C\ot A} p_2$ and $(p_1 \ot \mathrm{Id}_{C \ot A}) \psi' = (\mathrm{Id}_A \ot p_2)\sigma_{p^{-1}(C),A} \psi$. Furthermore, also the following diagram is a pullback in $\Mm$:
\[\begin{tikzcd}
	p^{-1}(C)\ot p^{-1}(C) && p^{-1}(C)\ot C \ot A \\
	p^{-1}(C) \ot A && p^{-1}(C) \ot B \ot A
	\arrow[from=1-1, to=1-3,"\mathrm{Id}\ot p_{2}"]
	\arrow[from=1-1, to=2-1, "\mathrm{Id}\ot p_{1}"']
	\arrow["\lrcorner"{anchor=center, pos=0.125}, draw=none, from=1-1, to=2-3]
	\arrow[from=1-3, to=2-3, "\mathrm{Id}\ot i\ot\mathrm{Id}_{A}"]
	\arrow[from=2-1, to=2-3, "\mathrm{Id}\ot(p\ot\mathrm{Id}_{A})\Delta_{A}"']
\end{tikzcd}\]
Since
\[
(p_{1}\ot\mathrm{Id}_A)\psi=\Delta_{A}p_{1} = \sigma_{A,A} \Delta_{A}p_{1} = \sigma_{A,A} (p_{1}\ot\mathrm{Id}_A)\psi = (\mathrm{Id}_A\ot p_{1}) \sigma_{p^{-1}(C),A} \psi,
\]
we get
\begin{align*}
&(p_1 \ot \mathrm{Id}_{B \ot A}) (\mathrm{Id}_{p^{-1}(C)} \ot i\ot\mathrm{Id}_{A}) \psi' = (\mathrm{Id}_{A} \ot i\ot\mathrm{Id}_{A}) (p_1 \ot \mathrm{Id}_{C \ot A})\psi'\\
= &(\mathrm{Id}_{A} \ot i\ot\mathrm{Id}_{A}) (\mathrm{Id}_A \ot p_2)\sigma_{p^{-1}(C),A} \psi 
= (\mathrm{Id}_{A} \ot (p\ot\mathrm{Id}_{A})\Delta_{A}) (\mathrm{Id}_{A} \ot p_1)\sigma_{p^{-1}(C),A} \psi \\=& (\mathrm{Id}_{A} \ot (p\ot\mathrm{Id}_{A})\Delta_{A})(p_{1}\ot\mathrm{Id}_A)\psi=(p_{1}\ot\mathrm{Id}_{B \ot A})(\mathrm{Id}_{p^{-1}(C)} \ot (p\ot\mathrm{Id}_{A})\Delta_{A})\psi.
\end{align*}
Since $p_{1}\ot\mathrm{Id}_{B \ot A}$ is a monomorphism in $\Mm$, we obtain $(\mathrm{Id}_{p^{-1}(C)} \ot i\ot\mathrm{Id}_{A}) \psi' = (\mathrm{Id}_{p^{-1}(C)} \ot(p\ot\mathrm{Id}_{A})\Delta_{A}) \psi$. Consequently, there is a unique morphism $\Delta_{p^{-1}(C)}: p^{-1}(C) \to  p^{-1}(C) \ot p^{-1}(C)$ in $\Mm$ such that $(\mathrm{Id}_{p^{-1}(C)} \ot p_2) \Delta_{p^{-1}(C)} = \psi'$ and $(\mathrm{Id}_{p^{-1}(C)} \ot p_1)\Delta_{p^{-1}(C)} = \psi$. Then, we have
\[
(p_{1}\ot p_{1})\Delta_{p^{-1}(C)}=(p_{1}\ot\mathrm{Id}_{A})(\mathrm{Id}_{p^{-1}(C)} \ot p_1)\Delta_{p^{-1}(C)}=(p_{1}\ot\mathrm{Id}_{A})\psi=\Delta_{A}p_{1}
\]
and also
\[
(p_{2}\ot p_{2})\Delta_{p^{-1}(C)}=(p_{2}\ot\mathrm{Id}_{C\ot A})(\mathrm{Id}_{p^{-1}(C)}\ot p_{2})\Delta_{p^{-1}(C)}=(p_{2}\ot\mathrm{Id}_{C\ot A})\psi'=\Delta_{C\ot A}p_{2}.
\]
Moreover, we define $\varepsilon_{p^{-1}(C)}:=\varepsilon_{A}p_{1}$ which is also equal to
\[
\varepsilon_{A}p_{1}=(\varepsilon_{B}\ot\varepsilon_{A})(p\ot\mathrm{Id}_{A})\Delta_{A}p_{1}=(\varepsilon_{B}\ot\varepsilon_{A})(i\ot\mathrm{Id}_{A})p_{2}=(\varepsilon_{C}\ot\varepsilon_{A})p_{2}.
\]
One can check that $(p^{-1}(C),\Delta_{p^{-1}}(C),\varepsilon_{p^{-1}(C)})$ is an object in $\mathsf{Comon}(\Mm)$, so $p_{1}$ and $p_{2}$ becomes morphisms in $\mathsf{Comon}(\Mm)$. Since $A$ is an object in $\mathsf{Bimon}(\Mm)$, we automatically obtain that $p^{-1}(C)$ is an object in $\mathsf{Bimon}(\Mm)$. Finally, we compute
\[
\begin{split}
(p\ot\mathrm{Id}_{A})\Delta_{A}S_{A}p_{1}&=(p\ot\mathrm{Id}_{A})(S_{A}\ot S_{A})\Delta_{A}p_{1}=(S_{B}\ot S_{A})(p\ot\mathrm{Id}_{A})\Delta_{A}p_{1}=(S_{B}\ot S_{A})(i\ot\mathrm{Id}_{A})p_{2}\\&=(i\ot\mathrm{Id}_{A})(S_{C}\ot S_{A})p_{2}=(i\ot\mathrm{Id}_{A})S_{C\ot A}p_{2}.
\end{split}
\]
Thus, there exists a unique morphism $S_{p^{-1}(C)}:p^{-1}(C)\to p^{-1}(C)$ in $\Mm$ such that $p_{1}S_{p^{-1}(C)}=S_{A}p_{1}$ and $p_{2}S_{p^{-1}(C)}=S_{C\ot A}p_{2}$. Since $S_{A}$ is the antipode of $A$, one can check that $S_{p^{-1}(C)}$ is an antipode of $p^{-1}(C)$, which is so an object in $\Hopf$. As a consequence, $p_{1}$ and $p_{2}$ are morphisms in $\Hopf$.
\end{proof}

We now construct the pullback of a morphism $p:A\to B$ in $\Hopf$ along a monomorphism $i:C\to B$ in $\Hopf$, using the previous result.

\begin{lemma}\label{lem:pullbackp(c)}
    Let $p:A\to B$ be a morphism in $\Hopf$ and $i:C\to B$ be a monomorphism in $\Hopf$. Then, the diagram
\begin{equation}\label{pullbackmono}
\begin{tikzcd}
	p^{-1}(C) && C \\
	A && B
	\arrow[from=1-1, to=1-3, "(\mathrm{Id}_{C}\ot\varepsilon_{A})p_{2}"]
	\arrow[from=1-1, to=2-1, "p_{1}"']
	\arrow[from=1-3, to=2-3, "i"]
	\arrow[from=2-1, to=2-3, "p"']
\end{tikzcd}
\end{equation}
is a pullback in $\Hopf$, where $(p^{-1}(C),p_{1},p_{2})$ is the pullback in $\Mm$ defined as in \eqref{pullbackpC}.
\end{lemma}

\begin{proof}
By the previous lemma, we know that $p^{-1}(C)$ is an object in $\Hopf$ and $p_{1},p_{2}$ are morphisms in $\Hopf$. Then, the diagram (\ref{pullbackmono}) is in $\Hopf$ and it commutes:
\[
pp_{1}=(\mathrm{Id}_{B}\ot\varepsilon_{A})(p\ot\mathrm{Id}_{A})\Delta_{A}p_{1}=(\mathrm{Id}_{B}\ot\varepsilon_{A})(i\ot\mathrm{Id}_{A})p_{2}=i(\mathrm{Id}_{C}\ot\varepsilon_{A})p_{2}.
\]
Now, we prove the universal property. Suppose there are morphisms $\alpha:D\to A$ and $\beta:D\to C$ in $\Hopf$ such that $p\alpha=i\beta$. Since $(p\ot\mathrm{Id}_{A})\Delta_{A}\alpha=(p\ot\mathrm{Id}_{A})(\alpha\ot\alpha)\Delta_{D}=(i\ot\mathrm{Id}_{A})(\beta \ot\alpha)\Delta_{D}$, by the universal property of the pullback $(p^{-1}(C),p_{1},p_{2})$, there exists a unique morphism $\psi:D\to p^{-1}(C)$ in $\Mm$ such that $p_{1}\psi=\alpha$ and $p_{2}\psi=(\beta\ot\alpha)\Delta_{D}$. Then, $(\mathrm{Id}_{C}\ot\varepsilon_{A})p_{2}\psi=(\mathrm{Id}_{C}\ot\varepsilon_{A})(\beta\ot\alpha)\Delta_{D}=(\beta\ot\varepsilon_{D})\Delta_{D}=\beta$. Since $p_{1}$ and $\alpha$ are morphisms in $\Hopf$ and $p_{1}$ is a monomorphism, one can check that $\psi:D\to p^{-1}(C)$ is a morphism in $\Hopf$. Thus, we can conclude. 
\end{proof}

\begin{remark}
    Using the description of pullbacks in $\Hopf$ given in Subsection \ref{subsec:pullbacksHopf} one can deduce that $p^{-1}(C)$ is isomorphic to $A\square_{B}C$ in $\Hopf$, where the $B$-coactions involved are $\rho_{A}:=(\mathrm{Id}_{A}\ot p)\Delta_{A}$ and $\lambda_{C}:=(i\ot\mathrm{Id}_{C})\Delta_{C}$.
\end{remark}

Using Lemma \ref{lem:pullbackp(c)}, we will prove that regular epimorphisms in $\Hopf$ are stable under pullbacks along split monomorphisms in $\Hopf$ once a technical condition, which is covered by faithful flatness condition, is satisfied. To do this, we first prove the following result.

\begin{lemma}\label{lem:tildepepi}
    Let $p:A\to B$ be a morphism in $\Hopf$ and $i:C\to B$ be a monomorphism in $\Hopf$. Consider the pullback \eqref{pullbackmono} in $\Hopf$. Then, $(\mathrm{Id}_{C}\ot\varepsilon_{A})p_{2}$ is an epimorphism in $\Mm$ if and only if there exists a monomorphism $\iota:D\to A$ in $\Hopf$ such that $i=\mathsf{ker}(\mathsf{coker}(p\iota))$.
\end{lemma}

\begin{proof}
   Suppose $(\mathrm{Id}_{C}\ot\varepsilon_{A})p_{2}$ is an epimorphism in $\Mm$. Since $pp_{1}=i(\mathrm{Id}_{C}\ot\varepsilon_{A})p_{2}$, we have $\mathsf{ker}(\mathsf{coker}(pp_{1}))=\mathsf{ker}(\mathsf{coker}(i))=i$, since $i$ is a monomorphism in $\Mm$ by 2) of Corollary \ref{cor:regepimono}. Define $D:=p^{-1}(C)$ and $\iota:=p_{1}$. Since $\iota:=p_{1}$ is a morphism in $\Hopf$ by Lemma \ref{lem:p1p2Hopf} and a monomorphism in $\Mm$, it is a monomorphism in $\Hopf$ by 2) of Corollary \ref{cor:regepimono}.
   
   Conversely, suppose that there exists a monomorphism $\iota:D\to A$ in $\Hopf$ such that $i=\mathsf{ker}(\mathsf{coker}(p\iota))$. 
   We consider the pullback $(p^{-1}(C),p_{1},p_{2})$ of the pair of morphisms $((p\ot\mathrm{Id}_{A})\Delta_{A},i\ot\mathrm{Id}_{A})$ in $\Mm$ as in Lemma \ref{lem:p1p2Hopf}. Since
\[
\begin{split}
(i\ot\mathrm{Id}_{A})(\mathsf{coker}(\mathsf{ker}(p\iota))\ot \iota)\Delta_{D}&=(\mathsf{ker}(\mathsf{coker}(p\iota))\ot\mathrm{Id}_{A})(\mathsf{coker}(\mathsf{ker}(p\iota))\ot \iota)\Delta_{D}=(p\iota\ot \iota)\Delta_{D}\\&=(p\ot\mathrm{Id}_{A})\Delta_{A}\iota,
\end{split}
\]
by the universal property of the pullback there exists a unique morphism $j:D\to p^{-1}(C)$ in $\Mm$ such that $p_{1}j=\iota$ and $p_{2}j=(\mathsf{coker}(\mathsf{ker}(p\iota))\ot\iota)\Delta_{D}$. Since 
\[
\mathsf{ker}(\mathsf{coker}(p\iota))(\mathrm{Id}_{C}\ot\varepsilon_{A})p_{2}j=i(\mathrm{Id}_{C}\ot\varepsilon_{A})p_{2}j=pp_{1}j=p\iota=\mathsf{ker}(\mathsf{coker}(p\iota))\mathsf{coker}(\mathsf{ker}(p\iota)), 
\]
we get $(\mathrm{Id}_{C}\ot\varepsilon_{A})p_{2}j=\mathsf{coker}(\mathsf{ker}(p\iota))$. Therefore, $(\mathrm{Id}_{C}\ot\varepsilon_{A})p_{2}$ is an epimorphism in $\Mm$. 
\end{proof}

\begin{invisible}
\rd{[The next result is correct but it seems not useful for the sequel.]}
\begin{lemma}
    Let $p:A\to B$ be a regular epimorphism in $\Hopf$. Then, for any monomorphism $i:D\to A$ in $\Hopf$, there exist unique morphisms $\bar{p}:\mathsf{Ker}(\mathsf{coker}(g))\to\mathsf{Ker}(\mathsf{coker}(f))$ and $\overline{p\ot\mathrm{Id}_{D}}:\mathsf{Coker}(\mathsf{ker}(g))\to\mathsf{Coker}(\mathsf{ker}(f))$ in $\Mm$ such that
\begin{align}
    p\mathsf{ker}(\mathsf{coker}(g))&=\mathsf{ker}(\mathsf{coker}(f))\bar{p},\label{eq:commphip}\\
\overline{p\ot\mathrm{Id}_{D}}\mathsf{coker}(\mathsf{ker}(g))&=\mathsf{coker}(\mathsf{ker}(f))(p\ot\mathrm{Id}_{D}),
\end{align}
where $g:=m_{A}(\mathrm{Id}_{A}\ot i)-\mathrm{Id}_{A}\ot\varepsilon_{D}$ and $f:=m_{B}(\mathrm{Id}_{B}\ot pi)-\mathrm{Id}_{B}\ot\varepsilon_{D}$.
\end{lemma}

\begin{proof}
   Since
\[
\begin{split}
pg&=p(m_{A}(\mathrm{Id}_{A}\ot i)-\mathrm{Id}_{A}\ot\varepsilon_{D})=m_{B}(p\ot p)(\mathrm{Id}_{A}\ot i)-p\ot\varepsilon_{D}\\&=(m_{B}(\mathrm{Id}_{B}\ot pi)-\mathrm{Id}_{B}\ot\varepsilon_{D})(p\ot\mathrm{Id}_{D})=f(p\ot\mathrm{Id}_{D}),
\end{split}
\]
we obtain $\mathsf{coker}(f)pg=\mathsf{coker}(f)f(p\ot\mathrm{Id}_{D})=0$. Therefore, by the universal property of the cokernel, there exists a unique morphism $\tilde{p}:\mathsf{Coker}(g)\to\mathsf{Coker}(f)$ in $\Mm$ such that the following diagram commutes:
\[\begin{tikzcd}
	A\ot D &&& A & \mathsf{Coker}(g) \\
	B\ot D &&& B & \mathsf{Coker}(f)
	\arrow[from=1-1, to=1-4, "g"]
	\arrow[from=1-4, to=1-5,"\mathsf{coker}(g)"]
	\arrow[from=1-4, to=2-4,"p"]
	\arrow[dashed, from=1-5, to=2-5, "\tilde{p}"]
	\arrow[from=2-1, to=2-4, "f"']
	\arrow[from=2-4, to=2-5,"\mathsf{coker}(f)"']
    \arrow[from=1-1, to=2-1, "p\ot\mathrm{Id}_{D}"']
\end{tikzcd}\]
Thus, we have $\mathsf{coker}(f)p\mathsf{ker}(\mathsf{coker}(g))=\tilde{p}\mathsf{coker}(g)\mathsf{ker}(\mathsf{coker}(g))=0$ and so, by the universal property of the kernel, there exists a unique morphism $\bar{p}:\mathsf{Ker}(\mathsf{coker}(g))\to\mathsf{Ker}(\mathsf{coker}(f))$ in $\Mm$ such that $p\mathsf{ker}(\mathsf{coker}(g))=\mathsf{ker}(\mathsf{coker}(f))\bar{p}$.

Moreover, since $f(p\ot\mathrm{Id}_{D})\mathsf{ker}(g)=pg\mathsf{ker}(g)=0$, by the universal property of the kernel there exists a unique morphism $\widetilde{p\ot\mathrm{Id}_{D}}:\mathsf{Ker}(g)\to\mathsf{Ker}(f)$ in $\Mm$ such that the following diagram commutes:
\[\begin{tikzcd}
	\mathsf{Ker}(g) & A\ot D & A \\
	\mathsf{Ker}(f) & B\ot D & B
	\arrow[from=1-1, to=1-2,"\mathsf{ker}(g)"]
	\arrow[dashed, from=1-1, to=2-1,"\widetilde{p\ot\mathrm{Id}_{D}}"']
	\arrow[from=1-2, to=1-3,"g"]
	\arrow[from=1-2, to=2-2,"p\ot\mathrm{Id}_{D}"]
	\arrow[from=2-1, to=2-2,"\mathsf{ker}(f)"']
	\arrow[from=2-2, to=2-3,"f"']
    \arrow[from=1-3, to=2-3,"p"]
\end{tikzcd}\]
Therefore, we have $\mathsf{coker}(\mathsf{ker}(f))(p\ot\mathrm{Id}_{D})\mathsf{ker}(g)=\mathsf{coker}(\mathsf{ker}(f))\mathsf{ker}(f)\widetilde{p\ot\mathrm{Id}_{D}}=0$ and so, by the universal property of the cokernel, there exists a unique morphism $\overline{p\ot\mathrm{Id}_{D}}:\mathsf{Coker}(\mathsf{ker}(g))\to\mathsf{Coker}(\mathsf{ker}(f))$ in $\Mm$ such that $\overline{p\ot\mathrm{Id}_{D}}\mathsf{coker}(\mathsf{ker}(g))=\mathsf{coker}(\mathsf{ker}(f))(p\ot\mathrm{Id}_{D})$.
\end{proof}

We observe that, once $i$ is an equalizer as in \eqref{iequalizer} so that we can apply $\phi_{A}$ to it, the object $\mathsf{coker}(g)$ in the previous result is exactly $\phi_{A}(i)$. 
\end{invisible}

The following result will be used in Proposition \ref{prop:stabilityregepi}.


\begin{proposition}\label{lemma:phip}
    Consider a regular epimorphism $p:A\to B$ in $\Hopf$ and a monomorphism $i:C\to B$ in $\Hopf$. 
    Let $\pi:=\phi_{B}(i):B\to Q$ in $\mathsf{Comon}_{\mathrm{coc}}(_{B}\Mm)$. 
    Then, $\pi p$ is a coequalizer as in \eqref{picoequalizer}.
\end{proposition}

\begin{proof}
Given the $B$-action $\mu_{Q}$ of $Q$, $Q$ automatically equips a left $A$-module structure $\mu_{Q}(p\ot\mathrm{Id}_{Q})$. Since $\pi pm_{A}=\pi m_{B}(p\ot p)=\mu_{Q}(\mathrm{Id}_{B}\ot\pi)(p\ot p)=\mu_{Q}(p\ot\id_{Q})(\id_{A}\ot\pi p)$, we obtain that $\pi p:A\to Q$ is in $\mathsf{Comon}_{\mathrm{coc}}(_{A}\Mm)$. 
Since $p$ is a morphism in $\mathsf{Comon}(\Mm)$, one has
\[
(\pi\ot\id_{B})\Delta_{B}p=(\id_{Q}\ot p)(\pi p\ot\id_{A})\Delta_{A},\quad(\id_{B}\ot\pi)\Delta_{B}p=(p\ot\id_{Q})(\id_{A}\ot\pi p)\Delta_{A},
\]
i.e.\ $p$ is a morphism in $^{Q}\Mm$ and $\Mm^{Q}$. 
\begin{invisible}
so the right squares in the following diagram 
\[\begin{tikzcd}
	A\square_{Q}A & A\ot A && A\ot Q\ot A\\
    B\square_{Q}B & B\ot B && B\ot Q\ot B 
	\arrow[from=2-1, to=2-2, "e_{B,B}"']
	\arrow[shift left, from=2-2, to=2-4, "\mathrm{Id}_{B}\ot(\pi\ot\mathrm{Id}_{B})\Delta_{B}"]
	\arrow[shift right, from=2-2, to=2-4, "(\mathrm{Id}_{B}\ot\pi)\Delta_{B}\ot\mathrm{Id}_{B}"']
	\arrow[dashed, from=1-1, to=2-1]
	\arrow[from=1-1, to=1-2,"e_{A,A}"]
	\arrow[from=1-2, to=2-2,"p\ot p"']
	\arrow[shift left, from=1-2, to=1-4,"\mathrm{Id}_{A}\ot(\pi p\ot\mathrm{Id}_{A})\Delta_{A}"]
	\arrow[shift right, from=1-2, to=1-4,"(\mathrm{Id}_{A}\ot\pi p)\Delta_{A}\ot\mathrm{Id}_{A}"']
	\arrow[from=1-4, to=2-4,"p\ot\mathrm{Id}_{Q}\ot p"]
\end{tikzcd}\]
are commutative. Hence, by the universal property of the equalizer there exists a unique 
\end{invisible}
Thus, we have a morphism 
$p\square_{Q}p:A\square_{Q}A\to B\square_{Q}B$ in $\Mm$ such that $e_{B,B}(p\square_{Q}p)=(p\ot p)e_{A,A}$. Clearly
\[
\begin{split}
\pi p(\varepsilon_{A}\ot\mathrm{Id}_{A})e_{A,A}&=\pi(\varepsilon_{B}\ot\mathrm{Id}_{B})(p\ot p)e_{A,A}=\pi(\varepsilon_{B}\ot\mathrm{Id}_{B})e_{B,B}(p\square_{B}p)\\&=\pi(\mathrm{Id}_{B}\ot\varepsilon_{B})e_{B,B}(p\square_{B}p)=\pi(\mathrm{Id}_{B}\ot\varepsilon_{B})(p\ot p)e_{A,A}\\&=\pi p(\mathrm{Id}_{A}\ot\varepsilon_{A})e_{A,A},
\end{split}
\]
i.e. $\pi p$ coequalizes the pair $((\varepsilon_{A}\ot\mathrm{Id}_{A})e_{A,A},(\mathrm{Id}_{A}\ot\varepsilon_{A})e_{A,A})$. Now, we verify the universal property. Suppose to have a morphism $f:A\to Z$ in $\Mm$ such that $f(\varepsilon_{A}\ot\mathrm{Id}_{A})e_{A,A}=f(\mathrm{Id}_{A}\ot\varepsilon_{A})e_{A,A}$. Using the morphism $\mathsf{can}:=(\id_A \otimes m_{A})(\Delta_{A} \otimes \id_{A}):A\ot A\to A\ot A$, we have
\[
\begin{split}
&(\mathrm{Id}_{A}\ot(\pi p\ot\mathrm{Id}_{A})\Delta_{A})\mathsf{can}(\id_A\ot\mathsf{hker}(p) )\\&=(\mathrm{Id}_{A}\ot(\pi p\ot\mathrm{Id}_{A})\Delta_{A})(\id_A \otimes m_{A})(\Delta_{A} \otimes \id_{A})(\id_A\ot\mathsf{hker}(p))\\&=\big(\mathrm{Id}_{A}\ot(\pi p\ot\mathrm{Id}_{A})(m_{A}\ot m_{A})(\id_{A}\ot\sigma_{A,A}\ot\id_{A})(\Delta_{A}\ot\Delta_{A})\big)(\Delta_{A} \otimes \id_{A})(\id_A\ot\mathsf{hker}(p))\\&=\big(\mathrm{Id}_{A}\ot(\pi \ot\mathrm{Id}_{A})(m_{B}\ot m_{A})(p\ot p\ot\mathrm{Id}_{A\ot A})(\id_{A}\ot\sigma_{A,A}\ot\id_{A})(\Delta_{A}\ot\Delta_{A})\big)(\Delta_{A} \otimes \id_{A})(\id_A\ot\mathsf{hker}(p))\\&=\big(\mathrm{Id}_{A}\ot(\pi \ot\mathrm{Id}_{A})(m_{B}\ot m_{A})(\id_{B}\ot\sigma_{A,B}\ot\id_{A})((p\ot\id_{A})\Delta_{A}\ot(p\ot\mathrm{Id}_{A})\Delta_{A})\big)(\id_{A\ot A}\ot\mathsf{hker}(p))(\Delta_{A} \otimes \id_{\mathsf{Hker}(p)})\\&=\big(\mathrm{Id}_{A}\ot(\pi \ot\mathrm{Id}_{A})(m_{B}\ot m_{A})(\id_{B}\ot\sigma_{A,B}\ot\id_{A})\big)(\id_{A}\ot(p\ot\mathrm{Id}_{A})\Delta_{A}\ot(p\ot\mathrm{Id}_{A})\Delta_{A}\mathsf{hker}(p))(\Delta_{A} \otimes \id_{\mathsf{Hker}(p)})\\&=\big(\mathrm{Id}_{A}\ot(\pi \ot\mathrm{Id}_{A})(m_{B}\ot m_{A})(\id_{B}\ot\sigma_{A,B}\ot\id_{A})\big)(\id_{A}\ot(p\ot\mathrm{Id}_{A})\Delta_{A}\ot(u_{B}\ot\mathrm{Id}_{A})\mathsf{hker}(p))(\Delta_{A} \otimes \id_{\mathsf{Hker}(p)})\\&=\big(\mathrm{Id}_{A}\ot(\pi \ot\mathrm{Id}_{A})(m_{B}\ot m_{A})(\id_{B}\ot\sigma_{A,B}\ot\id_{A})\big)(\id_{A\ot B\ot A}\ot u_{B}\ot\mathrm{Id}_{A})(\id_{A}\ot(p\ot\mathrm{Id}_{A})\Delta_{A}\ot\mathrm{Id}_{A})(\Delta_{A} \otimes \mathsf{hker}(p))\\&=\big(\mathrm{Id}_{A}\ot(\pi \ot\mathrm{Id}_{A})(m_{B}\ot m_{A})(\id_{B}\ot\sigma_{A,B}\ot\id_{A})(\id_{B\ot A}\ot u_{B}\ot\mathrm{Id}_{A})\big)(\id_{A}\ot(p\ot\mathrm{Id}_{A})\Delta_{A}\ot\mathrm{Id}_{A})(\Delta_{A} \otimes \mathsf{hker}(p))\\&=\big(\mathrm{Id}_{A}\ot(\pi \ot\mathrm{Id}_{A})(m_{B}\ot m_{A})(\id_{B}\ot u_{B}\ot\mathrm{Id}_{A\ot A})\big)(\id_{A}\ot(p\ot\mathrm{Id}_{A})\Delta_{A}\ot\mathrm{Id}_{A})(\Delta_{A} \otimes \mathsf{hker}(p))\\&=\big(\mathrm{Id}_{A}\ot(\pi \ot\mathrm{Id}_{A})(\id_{B}\ot m_{A})\big)(\id_{A}\ot(p\ot\mathrm{Id}_{A})\Delta_{A}\ot\mathrm{Id}_{A})(\Delta_{A} \otimes \mathsf{hker}(p))\\&=\big(\mathrm{Id}_{A}\ot\pi p \ot\mathrm{Id}_{A}\big)(\id_{A\ot A}\ot m_{A})(\id_{A}\ot\Delta_{A}\ot\mathrm{Id}_{A})(\Delta_{A} \otimes \mathsf{hker}(p))\\&=((\mathrm{Id}_{A}\ot\pi p)\Delta_{A} \ot\mathrm{Id}_{A})(\id_{A}\ot m_{A})(\Delta_{A} \otimes \mathsf{hker}(p))\\&=((\mathrm{Id}_{A}\ot\pi p)\Delta_{A}\ot\mathrm{Id}_{A})\mathsf{can}(\id_A\ot\mathsf{hker}(p)).
\end{split}
\]
Hence, since $e_{A,A}$ is the equalizer of the pair of morphisms $(\id_{A}\ot(\pi p\ot\id_{A})\Delta_{A},(\id_{A}\ot\pi p)\Delta_{A}\ot\id_{A})$ in $\Mm$, there exists a unique morphism $t:A\ot\mathsf{HKer}(p)\to A\square_{Q}A$ in $\Mm$ such that $\mathsf{can}(\id_A\ot\mathsf{hker}(p))= e_{A,A} t$. Since $f(\varepsilon_A \ot \id_{A}) e_{A,A} = f(\id_{A} \ot \varepsilon_A) e_{A,A}$, we get $f(\varepsilon_A \ot \id_{A}) e_{A,A} t = f(\id_{A} \ot \varepsilon_A) e_{A,A} t$, i.e.
\[
f(\varepsilon_A \ot \id)\mathsf{can}(\id_A\ot\mathsf{hker}(p))= f(\id \ot \varepsilon_A) \mathsf{can}(\id_A\ot\mathsf{hker}(p)).
\]
The left hand side is 
\[
f(\varepsilon_A \ot \id_{A})\mathsf{can}(\id_A\ot\mathsf{hker}(p))=f(\varepsilon_A \ot \id)(\id_A \otimes m_{A})(\Delta_{A} \otimes \id_{A})(\id_A\ot\mathsf{hker}(p))=f m_{A}(\id_{A}\ot\mathsf{hker}(p)) 
\]
while the right hand side is
\[
\begin{split}
f(\id_{A} \ot \varepsilon_A) \mathsf{can}(\id_A\ot\mathsf{hker}(p))&=f(\id_{A} \ot \varepsilon_A) (\id_A \otimes m_{A})(\Delta_{A} \otimes \id_{A})(\id_A\ot\mathsf{hker}(p))\\&=f(\mathrm{Id}_{A}\ot\varepsilon_{A}\mathsf{hker}(p))=f(\mathrm{Id}_{A}\ot\varepsilon_{\mathsf{Hker}(p)}).
\end{split}
\]
Hence, we have $f(m_{A}(\id_{A}\ot\mathsf{hker}(p))-\mathrm{Id}_{A}\ot\varepsilon_{\mathsf{Hker}(p)})=0$. Thus, there exists a unique morphism $\xi:\mathsf{Coker}(m_{A}(\id_{A}\ot\mathsf{hker}(p))-\mathrm{Id}_{A}\ot\varepsilon_{\mathsf{Hker}(p)})\to Z$ in $\Mm$ such that $f=\xi\mathsf{coker}(m_{A}(\id_{A}\ot\mathsf{hker}(p))-\mathrm{Id}_{A}\ot\varepsilon_{\mathsf{Hker}(p)})$. By Proposition \ref{prop:factmorp2} and the fact that $p$ is an epimorphism in $\Mm$ by 1) of Corollary \ref{cor:regepimono}, we have
\[
\mathsf{coker}(m_{A}(\id_{A}\ot\mathsf{hker}(p))-\mathrm{Id}_{A}\ot\varepsilon_{\mathsf{Hker}(p)})=\phi_{A}(\mathsf{hker}(p))=\mathsf{coker}(\mathsf{ker}(p))=p,
\]
so $f=\xi' p$, where $\xi'=\xi\zeta:B\to Z$ and $\zeta:B\to\mathsf{Coker}(m_{A}(\id_{A}\ot\mathsf{hker}(p))-\mathrm{Id}_{A}\ot\varepsilon_{\mathsf{Hker}(p)})$ is an isomorphism in $\Mm$.
Now, we have
\[
\begin{split}
\xi'
(\varepsilon_B \ot\id_{B}) e_{B,B}(p\square_{B}p)&=\xi'
(\varepsilon_B\ot\id_{B})(p \ot p)e_{A,A}=\xi'
p(\varepsilon_A \ot \id_{A}) e_{A,A}=f(\varepsilon_A \ot \id_{A}) e_{A,A}\\&=f(\id_{A}\ot\varepsilon_{A})e_{A,A}=\xi'
p(\id_{A}\ot\varepsilon_A ) e_{A,A}=\xi'
(\id_{B}\ot\varepsilon_B )(p\ot p) e_{A,A}\\&=\xi'
(\id_{B}\ot\varepsilon_B) e_{B,B}(p\square_{B}p).
\end{split}
\]
Since $\pi=\phi_{B}(i):B\to Q$ is a morphism in $\mathsf{Comon}_{\mathrm{coc}}(_{B}\Mm)$ and an epimorphism in $\Mm$, by the faithful coflatness condition on $\Mm$, we have that $(-)\square_{Q}B$ preserves epimorphisms. Moreover, $\pi p:A\to Q$ is a morphism in $\mathsf{Comon}_{\mathrm{coc}}(_{A}\Mm)$ and, since $p$ is an epimorphism in $\Mm$, $\pi p$ is an epimorphism in $\Mm$. Hence $A\square_{Q}(-)$ preserves epimorphisms, by the faithful coflatness condition on $\Mm$. As a consequence, $p\square_{Q}p=(p\square_{Q}\id_{B})(\id_{A}\square_{Q}p)$ is an epimorphism in $\Mm$. Thus, we get that $\xi'
(\varepsilon_B \ot\id_{B}) e_{B,B}=\xi'
(\id_{B}\ot\varepsilon_B) e_{B,B}$. Since $\pi=\phi_{B}(i)$ is the coequalizer of the pair $((\varepsilon_B \ot\id_{B}) e_{B,B},(\id_{B}\ot\varepsilon_B) e_{B,B})$ in $\Mm$ by 2) of Proposition \ref{prop:definitionphi}, there exists a unique morphism $\xi'':Q\to Z$ in $\Mm$ such that $\xi'=\xi''\pi$. It follows that $f=\xi' 
p=\xi''\pi p$, and $\xi''$ is the unique morphism in $\Mm$ such that this happens, so $\pi p$ is the coequalizer of the pair $((\varepsilon_{A}\ot\mathrm{Id}_{A})e_{A,A},(\mathrm{Id}_{A}\ot\varepsilon_{A})e_{A,A})$ in $\Mm$. 
\end{proof}

We are now able to prove that regular epimorphisms in $\Hopf$ are stable under pullbacks along any monomorphism in $\Hopf$ which is an equalizer as in \eqref{iequalizer}, once a technical condition is satisfied.

\begin{proposition}\label{prop:stabilityregepi}
Consider a regular epimorphism $p:A\to B$ in $\Hopf$ and a monomorphism $i:C\to B$ in $\Hopf$ which is an equalizer as in \eqref{iequalizer}. If $\mathsf{ker}(\mathsf{coker}(p\psi_{A}(\phi_{B}(i) p)))$ is an equalizer as in \eqref{iequalizer}, then the morphism $(\mathrm{Id}_{C}\ot\varepsilon_{A})p_{2}$ in the pullback \eqref{pullbackmono} is a regular epimorphism in $\Hopf$.
\end{proposition}

\begin{proof}
    By 1) of Corollary \ref{cor:regepimono} we know that regular epimorphisms in $\Hopf$ are exactly the morphisms in $\Hopf$ that are epimorphisms in $\Mm$ and the morphism $(\mathrm{Id}_{C}\ot\varepsilon_{A})p_{2}$ in the pullback \eqref{pullbackmono} is in $\Hopf$. Hence, it remains to prove that $(\mathrm{Id}_{C}\ot\varepsilon_{A})p_{2}$ is an epimorphism in $\Mm$, which is equivalent, by Lemma \ref{lem:tildepepi}, to prove that there exists a monomorphism $\iota:D\to A$ in $\Hopf$ such that $i=\mathsf{ker}(\mathsf{coker}(p\iota))$.
    
Define the morphism $\pi:=\phi_{B}(i):B\to Q$ in $\mathsf{Comon}_{\mathrm{coc}}(_{B}\Mm)$ and consider the morphism $\pi p:A\to Q$ in $\mathsf{Comon}_{\mathrm{coc}}(_{A}\Mm)$. 
We also define $\iota:=\psi_{A}(\pi p):A^{\mathrm{co}Q}\to A$ in $\Hopf$. Since
\[
\begin{split}
p(m_{A}(\mathrm{Id}_{A}\ot\iota)-\mathrm{Id}_{A}\ot\varepsilon_{A^{\mathrm{co}Q}})&=m_{B}(\mathrm{Id}_{B}\ot p\iota)(p\ot\mathrm{Id}_{A^{\mathrm{co}Q}})-(\mathrm{Id}_{B}\ot\varepsilon_{A^{\mathrm{co}Q}})(p\ot\mathrm{Id}_{A^{\mathrm{co}Q}})\\&=(m_{B}(\mathrm{Id}_{B}\ot p\iota)-\mathrm{Id}_{B}\ot\varepsilon_{A^{\mathrm{co}Q}})(p\ot\mathrm{Id}_{A^{\mathrm{co}Q}})
\end{split}
\]
and $p\ot\mathrm{Id}_{A^{\mathrm{co}Q}}$ is an epimorphism in $\Mm$, we have that 
\begin{equation}\label{auxxeeqq}
\mathsf{coker}(m_{B}(\mathrm{Id}_{B}\ot p\iota)-\mathrm{Id}_{B}\ot\varepsilon_{A^{\mathrm{co}Q}})=\mathsf{coker}(p(m_{A}(\mathrm{Id}_{A}\ot\iota)-\mathrm{Id}\ot\varepsilon_{A^{\mathrm{co}Q}})). 
\end{equation}
Recall that $\mathsf{coker}(\mathsf{ker}(p\iota))$ and $\mathsf{ker}(\mathsf{coker}(p\iota))$ are in $\Hopf$, see Remark \ref{rmk:kerHopfideal}. 
Since $\mathrm{Id}_{B}\ot \mathsf{coker}(\mathsf{ker}(p\iota))$ is an epimorphism in $\Mm$, we have
\[
\begin{split}
    &\phi_{B}(\mathsf{ker}(\mathsf{coker}(p\iota)))=\mathsf{coker}(m_{B}(\mathrm{Id}_{B}\ot \mathsf{ker}(\mathsf{coker}(p\iota)))-\mathrm{Id}_{B}\ot\varepsilon_{\mathrm{Im}(f)})\\&=\mathsf{coker}(m_{B}(\mathrm{Id}_{B}\ot \mathsf{ker}(\mathsf{coker}(p\iota)))-\mathrm{Id}_{B}\ot\varepsilon_{\mathrm{Im}(f)})(\mathrm{Id}_{B}\ot \mathsf{coker}(\mathsf{ker}(p\iota))))\\&=\mathsf{coker}(m_{B}(\mathrm{Id}_{B}\ot p\iota)-\mathrm{Id}_{B}\ot\varepsilon_{A^{\mathrm{co}Q}}) \overset{\eqref{auxxeeqq}}{=}\mathsf{coker}(p(m_{A}(\mathrm{Id}_{A}\ot\iota)-\mathrm{Id}\ot\varepsilon_{A^{\mathrm{co}Q}}))\\&=\mathsf{coker}(p\mathsf{ker} (\mathsf{coker}(m_{A}(\mathrm{Id}_{A}\ot\iota)-\mathrm{Id}_{A}\ot\varepsilon_{A^{\mathrm{co}Q}}))\mathsf{coker}(\mathsf{ker}(m_{A}(\mathrm{Id}_{A}\ot\iota)-\mathrm{Id}_{A}\ot\varepsilon_{A^{\mathrm{co}Q}})))\\&=\mathsf{coker}(p\mathsf{ker} (\mathsf{coker}(m_{A}(\mathrm{Id}_{A}\ot\iota)-\mathrm{Id}_{A}\ot\varepsilon_{A^{\mathrm{co}Q}}))),
\end{split}
\]
i.e.\ the following equality holds
\begin{equation}\label{eq:aux}
    \phi_{B}(\mathsf{ker}(\mathsf{coker}(p\iota)))=\mathsf{coker}(p\mathsf{ker} (\mathsf{coker}(m_{A}(\mathrm{Id}_{A}\ot\iota)-\mathrm{Id}_{A}\ot\varepsilon_{A^{\mathrm{co}Q}}))).
\end{equation}
Since the morphism $\pi p:A\to Q$ in $\mathsf{Comon}_{\mathrm{coc}}(_{A}\Mm)$ is a coequalizer as in \eqref{picoequalizer} by Proposition \ref{lemma:phip}, using Theorem \ref{thm:NewmanforM}, we get 
\[
\pi p=\phi_{A}(\psi_{A}(\pi p))=\phi_{A}(\iota)=\mathsf{coker}(m_{A}(\mathrm{Id}_{A}\ot\iota)-\mathrm{Id}_{A}\ot\varepsilon_{A^{\mathrm{co}Q}}) 
\]
and then $\mathsf{ker}(\pi p)=\mathsf{ker}(\mathsf{coker}(m_{A}(\mathrm{Id}_{A}\ot\iota)-\mathrm{Id}_{A}\ot\varepsilon_{A^{\mathrm{co}Q}}))$. It follows that 
$$
p\mathsf{ker}(\pi p)=p\mathsf{ker}(\mathsf{coker}(m_{A}(\mathrm{Id}_{A}\ot\iota)-\mathrm{Id}_{A}\ot\varepsilon_{A^{\mathrm{co}Q}})).
$$ 
Hence,
\begin{equation}\label{eq:auxfin}
\phi_{B}(\mathsf{ker}(\mathsf{coker}(p\iota)))\overset{\eqref{eq:aux}}{=}\mathsf{coker}(p\mathsf{ker} (\mathsf{coker}(m_{A}(\mathrm{Id}_{A}\ot\iota)-\mathrm{Id}_{A}\ot\varepsilon_{A^{\mathrm{co}Q}})))=\mathsf{coker}(p\mathsf{ker}(\pi p)).
\end{equation}
By applying the Snake Lemma to the following commutative diagram in $\Mm$
\[\begin{tikzcd}
	 \mathsf{Ker}(p) & A & B  \\
	  0 & Q & Q 
	\arrow[from=1-1, to=1-2,"\mathsf{ker}(p)"]
	\arrow[from=1-1, to=2-1]
	\arrow[from=1-2, to=1-3, "p"]
	\arrow[from=1-2, to=2-2,"\pi p"]
	\arrow[from=1-3, to=2-3,"\pi"]
	\arrow[from=2-1, to=2-2]
	\arrow[from=2-2, to=2-3,"\mathrm{Id}_{Q}"']
\end{tikzcd}\]
we obtain the exact sequence in $\Mm$
\[\begin{tikzcd}
	0 & \mathsf{Ker}(p) & \mathsf{Ker}(\pi p) & \mathsf{Ker}(\pi) & 0
	\arrow[from=1-1, to=1-2]
	\arrow[from=1-2, to=1-3]
	\arrow[from=1-3, to=1-4,"\tilde{p}"]
	\arrow[from=1-4, to=1-5]
\end{tikzcd}\]
where $\tilde{p}$ is an epimorphism in $\Mm$ such that $\mathsf{ker}(\pi)\tilde{p}=p\mathsf{ker}(\pi p)$.
This implies that 
\begin{equation}\label{eq:prop1}
\phi_{B}(\mathsf{ker}(\mathsf{coker}(p\iota)))\overset{\eqref{eq:auxfin}}{=}\mathsf{coker}(p\mathsf{ker}(\pi p))=\mathsf{coker}(\mathsf{ker}(\pi)\tilde{p})=\mathsf{coker}(\mathsf{ker}(\pi))=\pi=\phi_{B}(i)
\end{equation}
since $\pi$ is an epimorphism in $\Mm$. Then, since $\mathsf{ker}(\mathsf{coker}(p\iota))$ and $i$ are equalizers as in \eqref{iequalizer}, by applying Theorem \ref{thm:NewmanforM} we obtain $i=\mathsf{ker}(\mathsf{coker}(p\iota))$.  
\end{proof}


In order to have that the technical condition in the previous result always holds, we show that any morphism in $\Hopf$ which is a monomorphism in $\Mm$ (equivalently, any monomorphism in $\Hopf$) is an equalizer as in \eqref{iequalizer}. This is done in the following proposition, which is the dual of Proposition \ref{prop:cofaitfullyflat}. 

\begin{proposition}\label{prop:stability2}
Let $i:K \to A$ be a morphism in $\mathsf{Mon}(\Mm)$ which is a monomorphism in $\Mm$ such that $(-)\ot_{K}A$ preserves and reflects monomorphisms. Then, $\id_{\mathsf{Coker}(i)}\ot_{K}i$ is a monomorphism in $\Mm$. 
As a consequence, $i$ is an equalizer as in \eqref{iequalizer}.

\begin{invisible}
In particular, this happens when $i$ is morphism in $\Hopf$ which is a monomorphism in $\Mm$ and $(-)\ot_{K}A$ preserves and reflects monomorphisms.
\end{invisible}
\end{proposition}

\begin{proof}
The proof is the dual of the proof of Proposition \ref{prop:cofaitfullyflat}.
\begin{invisible}
COMPLETE PROOF HERE. Since $\mathsf{coker}(i)\ot\mathrm{Id}_{K}$ is the cokernel of the morphism $i\ot\mathrm{Id}_{K}$ in $\Mm$ and $\mathsf{coker}(i)m_{A}(\mathrm{Id}_{A}\ot i)(i\ot\mathrm{Id}_{K})=\mathsf{coker}(i)im_{K}=0$, there exists a unique morphism $\mu_{\mathsf{coker}(i)}:\mathsf{Coker}(i)\ot K\to\mathsf{Coker}(i)$ in $\Mm$ such that $\mu_{\mathsf{coker}(i)}(\mathsf{coker}(i)\ot\mathrm{Id}_{K})=\mathsf{coker}(i)m_{A}(\mathrm{Id}_{A}\ot i)$. One can check that $\mu_{\mathsf{Coker}(i)}$ defines a right $K$-action on $\mathsf{Coker}(i)$, so that $\mathsf{coker}(i)$ becomes a morphism in $\Mm_{K}$ by considering $A$ in $\Mm_{K}$ with $m_{A}(\id_{A}\ot i)$. Moreover, $A$ is an object in $_{K}\Mm_{K}$ with left module structure $m_{A}(i\ot\id_{A})$, so that $i$ becomes a morphism in $_{K}\Mm_{K}$. Since $(_{K}\Mm_{K},\ot_{K},K)$ is a monoidal category, the morphisms $\id_{A}\ot_{K}i$ and $i\ot_{K}\id_{A}$ are in $_{K}\Mm_{K}$.

Let $\Upsilon'_{A}:K\ot_{K} A\to A$ be the canonical isomorphism in $\Mm$ determined by $\Upsilon'_{A}q_{K,A}= \mu_{A} =m_{A} (i\ot\id_{A})$. There is a morphism $m'_{A}:A\ot_{K}A\to A$ in $\Mm$ such that $m'_{A}q_{A,A}=m_A$. Consider $(i\ot_{K} \id_{A})(\Upsilon'_{A})^{-1}:A\to A\ot_{K}A$, where $(\Upsilon'_{A})^{-1}=q_{K,A}(u_{K}\ot\id_{A})$. It is easy to check that $m'_A(i\ot_{K}\id_{A})(\Upsilon'_{A})^{-1}= \id_A$. Hence, $i\ot_{K}\id_{A}$ is a split monomorphism in $\Mm$. Therefore, $\id_{\mathsf{Coker}(i)}\ot_{K}i\ot_{K}\id_{A}$ is a split monomorphism in $\Mm$ and then, since $(-)\ot_{K}A$ reflects monomorphisms, we get that $\id_{\mathsf{Coker}(i)}\ot_{K}i$ is a monomorphism in $\Mm$.

Since
\[
\begin{split}
q_{A,A}(u_{A}\ot \mathrm{Id}_{A})i&=q_{A,A}(\mathrm{Id}_{A}\ot m_{A}(i\ot\mathrm{Id}_{A}))(u_{A}\ot\mathrm{Id}_{K}\ot u_{A})\\&=q_{A,A}(m_{A}(\mathrm{Id}_{A}\ot i)\ot\mathrm{Id}_{A})(u_{A}\ot\mathrm{Id}_{K}\ot u_{A})=q_{A,A}(\mathrm{Id}_{A}\ot u_{A})i,
\end{split}
\]
it remains to show the universal property. Suppose there is a morphism $f:A' \to A$ in $\Mm$ such that $q_{A,A}(\mathrm{Id}_{A}\ot u_{A})f = q_{A,A}(u_{A}\ot\mathrm{Id}_{A})f$.  
As recalled in the preliminaries, $\mathsf{coker}(i)\ot_K\mathrm{Id}_{A}: A \ot_K A \to \mathsf{Coker}(i) \ot_K A$ is the unique morphism in $\Mm$ such that $(\mathsf{coker}(i)\ot_K\mathrm{Id}_{A})q_{A,A} = q_{\mathsf{Coker}(i),A}(\mathsf{coker}(i)\ot \id_A)$. 
Hence, we have
\[
\begin{split}
q_{\mathsf{Coker}(i),A}(\mathrm{Id}_{\mathsf{Coker}(i)}\ot u_{A})\mathsf{coker}(i)f&=
q_{\mathsf{Coker}(i),A}(\mathsf{coker}(i)\ot \id_A)(\mathrm{Id}_{A}\ot u_{A})f \\&=(\mathsf{coker}(i)\ot_K\mathrm{Id}_{A})q_{A,A} (\mathrm{Id}_{A}\ot u_{A})f\\&=(\mathsf{coker}(i)\ot_K\mathrm{Id}_{A})q_{A,A} (u_{A}\ot\mathrm{Id}_{A})f 
\\&= q_{\mathsf{Coker}(i),A}(\mathsf{coker}(i) u_{A}\ot\mathrm{Id}_{A})f\\&=q_{\mathsf{Coker}(i),A}(\mathsf{coker}(i)i u_K\ot\mathrm{Id}_{A})f=0.
\end{split}
\]
If we prove that $q_{\mathsf{Coker}(i),A}(\id_{\mathsf{Coker}(i)} \ot u_{A})$ is a monomorphism in $\Mm$, we get that $\mathsf{coker}(i)f=0$ and then, since $i=\mathsf{ker}(\mathsf{coker}(i))$ as it is a monomorphism in $\Mm$, by the universal property of the kernel there exists a unique morphism $\varphi:A'\to K$ in $\Mm$ such that $i\varphi=f$. We know that there exists a unique morphism $\mathrm{Id}_{\mathsf{Coker}(i)} \ot_K i$ in $\Mm$ such that $(\mathrm{Id}_{\mathsf{Coker}(i)} \ot_K i) q_{\mathsf{Coker}(i),K} = q_{\mathsf{Coker}(i),A}(\id_{\mathsf{Coker}(i)} \ot i)$, so that
\[
q_{\mathsf{Coker}(i),A}(\id_{\mathsf{Coker}(i)} \ot u_{A}) = q_{\mathsf{Coker}(i),A}(\id_{\mathsf{Coker}(i)} \ot iu_{K}) = (\mathrm{Id}_{\mathsf{Coker}(i)} \ot_K i) q_{\mathsf{Coker}(i),K}(\id_{\mathsf{Coker}(i)} \ot u_{K}).
\]
As recalled in the preliminaries, we know that $q_{\mathsf{Coker}(i),K}(\id_{\mathsf{Coker}(i)} \otimes u_K)$ is an isomorphism in $\Mm$. Since $\id_{\mathsf{Coker}(i)}\ot_{K}i$ is a monomorphism in $\Mm$, we obtain that $q_{\mathsf{Coker}(i),A}(\id_{\mathsf{Coker}(i)} \ot u_{A})$ is a monomorphism in $\Mm$.
\end{invisible}
\end{proof}

\begin{corollary}\label{cor:splitmonoff}
    Split monomorphisms in $\mathsf{Mon}(\Mm)$ are equalizers as in \eqref{iequalizer}. 
    \begin{invisible}
    In particular, this happens for split monomorphisms in $\Hopf$.
     \end{invisible}
\end{corollary}
\begin{proof}
If $i:K\to A$ is a split monomorphism in $\mathsf{Mon}(\Mm)$, there exists a morphism $\pi:A\to K$ in $\mathsf{Mon}(\Mm)$ such that $\pi i=\id_{K}$. Given the left $K$-action $\mu_{A}:=m_{A}(i\ot\id_{A})$ on $A$, we have $\pi \mu_{A}=\pi m_{A}(i\ot\id_{A})=m_{K}(\pi\ot\pi)(i\ot\id_{A})=m_{K}(\id_{K}\ot\pi)$ so $\pi$ is in $_{K}\Mm$. Then, we can consider the morphism $\id_{\mathsf{Coker}(i)}\ot_{K} \pi$ in $\Mm$ and obtain that $(\id_{\mathsf{Coker}(i)}\ot_{K} \pi)(\id_{\mathsf{Coker}(i)}\ot_{K} i)=\id$. This means that $\id_{\mathsf{Coker}(i)}\ot_K i$ is a split monomorphism in $\Mm$. 
Using the previous result, we obtain that split monomorphisms in $\mathsf{Mon}(\Mm)$ are always equalizers as in \eqref{iequalizer}.
\end{proof}

We introduce the following definition:

\begin{definition}\label{def:faithflatcat}
    Let $(\Mm,\ot,\mathbf{1},\sigma)$ be an abelian symmetric monoidal category. We say that $(\Mm,\ot,\mathbf{1},\sigma)$ satisfies the ``faithful flatness condition'' if, for any object $A$ in $\Hopf$ and any morphism $i:K\to A$ in $\Hopf$ which is a monomorphism in $\Mm$, $A$ is \textit{faithfully flat} over $K$, i.e.\ $(-)\ot_{K} A$ preserves and reflects monomorphisms. 
\end{definition}

\begin{remark}\label{rmk:prototypefaithflat}
    The prototype of this condition is given by $(\Mm,\ot,\mathbf{1},\sigma)=(\mathsf{Vec}_{\Bbbk},\ot_{\Bbbk},\Bbbk,\tau)$. It is known that, for any $A$ in $\mathsf{Hopf}_{\Bbbk,\mathrm{coc}}$ and any Hopf subalgebra $i:K\rightarrow A$, $A$ is faithfully flat over $K$; this was proven in \cite[Theorem 3.1]{Takeuchi}. 
\end{remark}

From now on let $(\Mm,\ot,\mathbf{1},\sigma)$ be an abelian symmetric monoidal category that satisfies the ``faithful coflatness condition'' (Definition \ref{deffaithcoflat}) and the ``faithful flatness condition'' (Definition \ref{def:faithflatcat}).

\begin{proposition}\label{cor:stability}
     Regular epimorphisms in $\Hopf$ (equivalently, cokernels in $\Hopf$) are stable under pullbacks.
\end{proposition}

\begin{proof}
By Lemma \ref{lem:regepipreservedbybinary}, it is enough to prove that regular epimorphisms in $\Hopf$ are stable under pullbacks along split monomorphisms. By Corollary \ref{cor:splitmonoff}, we know that split monomorphisms in $\Hopf$ are equalizers as in \eqref{iequalizer}. Therefore, by Proposition \ref{prop:stabilityregepi}, given a regular epimorphism $p:A\to B$ in $\Hopf$ and a split monomorphism $i:C\to B$ in $\Hopf$, $p$ is stable under pullback along $i$ if $\mathsf{ker}(\mathsf{coker}(p\iota))$ is an equalizer as in \eqref{iequalizer}, where $\iota=\psi_{A}(\phi_{B}(i) p)$. Since $\mathsf{ker}(\mathsf{coker}(p\iota)):\mathrm{Im}(p\iota)\to B$ is a morphism in $\Hopf$ (Remark \ref{rmk:kerHopfideal}) which is a monomorphism in $\Mm$, by assumption on $\Mm$ we have that $(-)\ot_{\mathrm{Im}(p\iota)} B$ preserves and reflects monomorphisms, hence $\mathsf{ker}(\mathsf{coker}(p\iota))$ is an equalizer as in \eqref{iequalizer} by  Proposition \ref{prop:stability2}.
\end{proof}

As a consequence of Proposition \ref{prop:factmorp2} and Proposition \ref{cor:stability}, we finally obtain:

\begin{theorem}\label{thm:Hopfisregular}
    Let $(\Mm,\ot,\mathbf{1},\sigma)$ be an abelian symmetric monoidal category that satisfies the ``faithful coflatness condition'' and the ``faithful flatness condition''. Then, the category $\Hopf$ is regular. 
\end{theorem}

\section{On the semi-abelianness of $\Hopf$}\label{sec:semiabelian}

From now on, we assume that $(\Mm,\ot,\mathbf{1},\sigma)$ is an abelian symmetric monoidal category that satisfies the ``faithful coflatness condition'' and the ``faithful flatness condition'', so that $\Hopf$ is regular by Theorem \ref{thm:Hopfisregular}. Since $\Hopf$ is pointed (Lemma \ref{lem:Hopfpointed}) and  protomodular (Proposition \ref{prop:protomodularity}), we obtain:

\begin{theorem}\label{thm:homological}
    The category $\Hopf$ is homological.
\end{theorem}

We recall the equivalent characterization of semi-abelian categories given in \cite[3.7]{JMT}: a category $\Cc$ is semi-abelian provided that:
\begin{itemize}
    \item[1)] $\Cc$ has binary products and binary coproducts and a zero object;
    \item[2)] $\Cc$ has pullbacks of (split) monomorphisms;
    \item[3)] $\Cc$ has cokernels of kernels and every morphism with zero kernel is a monomorphism;
    \item[4)] the Split Short Five Lemma holds true in $\Cc$;
    \item[5)] cokernels are stable under pullbacks;
    \item[6)] images of kernels along cokernels are kernels.
\end{itemize}

We prove that 6) holds true in order to obtain that the category $\Hopf$ is (Barr)-exact, then it is semi-abelian once it has binary coproducts.

\begin{proposition}
   The category $\Hopf$ is (Barr)-exact. 
\end{proposition}
      
\begin{proof}
To obtain that $\Hopf$ is exact it remains to prove that images of kernels along cokernels are kernels. 
More precisely, we want to show that, given the kernel of a morphism $g:X\to Z$ in $\Hopf$ and the cokernel
of a morphism $f:A\to X$ in $\Hopf$, there exist a morphism $p:\mathsf{Hker}(g)\to H$ in $\Hopf$ and a kernel $\iota:H\to \mathsf{Hcoker}(f)$ in $\Hopf$ such that the following diagram commutes:
\[\begin{tikzcd}
	& A \\
	\mathsf{Hker}(g) & X & Z \\
	H & \mathsf{Hcoker}(f)
	\arrow[from=1-2, to=2-2,"f"]
	\arrow[from=2-1, to=2-2, "\mathsf{hker}(g)"]
	\arrow[from=2-1, to=3-1, "p"']
	\arrow[from=2-2, to=2-3,"g"]
	\arrow[from=2-2, to=3-2,"\mathsf{hcoker}(f)"]
	\arrow[from=3-1, to=3-2, "\iota"']
\end{tikzcd}\]
Since the category $\Hopf$ is regular (Theorem \ref{thm:Hopfisregular}), the morphism $\mathsf{hcoker}(f)\mathsf{hker}(g):\mathsf{Hker}(g) \to\mathsf{Hcoker}(f)$ has a regular epi-mono factorization in $\Hopf$. Hence, there exist a regular epimorphism $p:\mathsf{Hker}(g)\to H$ in $\Hopf$ and a monomorphism $\iota:H\to\mathsf{Hcoker}(f)$ in $\Hopf$ such that $\mathsf{hcoker}(f)\mathsf{hker}(g)=\iota p$. By 2) of Corollary \ref{cor:regepimono}, $\iota$ is a monomorphism in $\Mm$. By assumption on $\Mm$, we have that $(-)\ot_{H}\mathsf{Hcoker}(f)$ preserves and reflects monomorphisms and then, by Proposition \ref{prop:stability2}, we have that $\iota$ is an equalizer as in \eqref{iequalizer}. Therefore, by Theorem \ref{cor:normalobjects}, $\iota$ is a kernel in $\Hopf$ if and only if it is normal. Since $\mathsf{hker}(g)$ is normal by Theorem  \ref{cor:normalobjects} and $\mathsf{hcoker}(f)$ is a morphism in $\Hopf$ which is an epimorphism in $\Mm$ by 1) of Corollary \ref{cor:regepimono}, we have that $\mathsf{ker}(\mathsf{coker}(\mathsf{hcoker}(f)\mathsf{hker}(g)))$
is normal by 3) of Lemma \ref{properties:ad}. Since $p$ is an epimorphism in $\Mm$ and $\iota$ is a monomorphism in $\Mm$, we have $\mathsf{ker}(\mathsf{coker}(\mathsf{hcoker}(f)\mathsf{hker}(g)))=\mathsf{ker}(\mathsf{coker}(\iota p))=\mathsf{ker}(\mathsf{coker}(\iota))=\iota$. Hence, $\iota$ is normal, so a kernel in $\Hopf$ by Theorem \ref{cor:normalobjects}.
\end{proof}

Finally, we obtain the main result of this paper. 

\begin{theorem}\label{thm:semiabelian}
Let $(\Mm,\ot,\mathbf{1},\sigma)$ be an abelian symmetric monoidal category that satisfies the ``faithful coflatness condition'' and the ``faithful flatness condition''. Then, the category $\Hopf$ is (Barr)-exact and homological. As a consequence, if $\Hopf$ has binary coproducts, then it is a semi-abelian category.
\end{theorem}

\begin{remark}
    If $(\Mm,\ot,\mathbf{1},\sigma)=(\mathsf{Vec}_{\Bbbk},\ot_{\Bbbk},\Bbbk,\tau)$ we recover the semi-abelianness of the category $\mathsf{Hopf}_{\Bbbk,\mathrm{coc}}$ of cocommutative Hopf algebras over an arbitrary field $\Bbbk$, achieved in \cite[Theorem 2.10]{GSV}. As observed in Remark \ref{rmk:prototypefaithcoflat} and Remark \ref{rmk:prototypefaithflat}, the abelian symmetric monoidal category $(\mathsf{Vec}_{\Bbbk},\ot_{\Bbbk},\Bbbk,\tau)$ satisfies the ``faithful coflatness condition'' and the ``faithful flatness condition''. The same is obtained for $(\mathsf{Vec}_{\mathbb{Z}_{2}},\ot_{\Bbbk},\Bbbk,\sigma)$ with $\mathrm{char}(\Bbbk)\not=2$, the category of super vector spaces, in \cite[Theorem 3.10 (1) and (2)]{Masuoka2}; for the latter we recall that a comodule over a coalgebra is an injective (cogenerator) if and only if it is (faithfully) coflat, see \cite[Proposition A.2.1]{Takeuchi2} and, dually, a module over an algebra is a projective (generator) if and only if it is (faithfully) flat. 
\end{remark}  

Moreover, the Newman bijection \eqref{bijectionNewman}, which is given for any object $A$ in $\mathsf{Hopf}_{\mathrm{coc}}(\mathsf{Vec}_{\Bbbk})$ and goes back to \cite{Newman}, is generalized to any object $A$ in $\mathsf{Hopf}_{\mathrm{coc}}(\mathsf{Vec}_{\mathbb{Z}_{2}})$ in \cite[Theorem 3.10 (3)]{Masuoka2}. This result was then extended in \cite[Theorem 5.20]{AS} for any object $A$ in $\mathsf{Hopf}_{\mathrm{coc}}(\mathsf{Vec}_{G})$, where $G$ is a finitely generated abelian group and $\mathrm{char}(\Bbbk)\not=2$. It is known that $\mathsf{Vec}_{G}\cong\mm^{\Bbbk G}$, hence it becomes natural to consider the category $\mm^{H}$ of comodules over a bialgebra (or Hopf algebra).

\begin{remark}
Let $H$ be a bialgebra (or Hopf algebra). It is known that $(\mm^{H},\otimes,\Bbbk)$ is an abelian monoidal category. In fact, as said in e.g. \cite[Example 4.25]{A08}, this category is abelian, monoidal and $X\ot(-)$ and $(-)\ot X$ are left exact, for any $X$ in $\mm^{H}$. Moreover, $X\ot(-)$ and $(-)\ot X$ are also right exact for any $X$ in $\mm^{H}$ as they admit a right adjoint, i.e.\ the category $(\mm^{H},\otimes,\Bbbk)$ is \textit{closed monoidal}; this fact can be derived from \cite[Corollary V.8]{MacLane-book} and an explicitly description of the adjoint can be found 
in \cite[Proposition 2.3]{AsWe}. In fact, the antipode is unnecessary as one can see in e.g.\ \cite[Theorem 1.3.1]{Hovey}. Moreover, the category $(\mm^{H},\otimes,\Bbbk)$ is symmetric if and only if $H$ is \textit{cotriangular}, i.e.\ there exists a convolution invertible morphism $\Rr:H\ot H\to\Bbbk$ satisfying some axioms, see e.g.\ \cite[Definition 2.2.1]{Majid-book} and \cite[Exercise 9.2.9]{Majid-book}. Thus, for a cotriangular bialgebra (or Hopf algebra) $(H,\Rr)$, the category $(\mm^{H},\ot,\Bbbk,\sigma^{\Rr})$ is an abelian symmetric monoidal category. To apply Theorem \ref{thm:semiabelian}, we need that it satisfies the ``faithful coflatness condition'' and the ``faithful flatness condition''. 

It is known that epimorphisms in $\mm^{H}$ coincide with surjective right $H$-colinear maps since $\mm^{H}$ is a  Grothendieck category. Then, the ``faithful coflatness condition'' can be written as: for any $A$ in $\mathsf{Hopf}_{\mathrm{coc}}(\mm^{H})$ and any $\pi:A\to Q$ in $\mathsf{Comon}_{\mathrm{coc}}(_{A}\mm^{H})$ which is surjective, the functor $(-)\square_{Q}A:(\mm^{H})^{Q}\to\mm^{H}$ preserves and reflects epimorphisms. In \cite[Proposition 1.3]{CDR}, it was proven that the category $(\mm^{H})^{Q}$ is isomorphic to $\mm^{H\ltimes Q}$, where $H\ltimes Q$ is the smash coproduct coalgebra of $H$ and $Q$. Given the smash coproduct Hopf algebra $H\ltimes A$, the surjective morphism $\id_{H}\ot\pi:H\ltimes A\to H\ltimes Q$ is in $\mathsf{Comon}(_{H\ltimes A}\mm)$. As observed in Remark \ref{rmk:prototypefaithcoflat}, the faithful coflatness of $H\ltimes A$ over $H\ltimes Q$ can be obtained by applying \cite[Theorem 1.3 (2)]{Masuoka} for the Hopf algebra $H\ltimes A$. To do this, we need that the coradical of $H\ltimes A$ is cocommutative. As pointed out in the proof of \cite[Theorem 3.10 (2)]{Masuoka2}, this happens for $H=\Bbbk\mathbb{Z}_{2}$ 
because an object $A$ in $\mathsf{Hopf}_{\mathrm{coc}}(\mm^{\Bbbk\mathbb{Z}_{2}})$, with $\Bbbk$ algebraically closed, is a \textit{pointed coalgebra}, i.e.\ all its simple subcoalgebras are 1-dimensional.
\end{remark}

It is of significant interest to determine the minimal conditions on $H$ such that the ``faithful coflatness condition'' and the ``faithful flatness condition'' are satisfied, but this would require deep and specific Hopf algebraic tools and would go beyond the scope of this paper. We will investigate this in the future. We leave this question open:\medskip

\noindent\textit{Question}: Is the category $\mathsf{Hopf}_{\mathrm{coc}}(\mm^{H})$ semi-abelian for any cosemisimple Hopf algebra $H$? \medskip

Theorem \ref{thm:semiabelian} opens other interesting directions. In fact, we recall that the category $\mathsf{HBr}_{\mathrm{coc}}$ of cocommutative Hopf braces, introduced in \cite{AnGaVe}, was proven to be semi-abelian in \cite{GranSciandra}, using the semi-abelianness of $\mathsf{Hopf}_{\Bbbk,\mathrm{coc}}$. Hence, by employing the semi-abelianness of $\Hopf$, it would be reasonable to try to extend the semi-abelianness of $\mathsf{HBr}_{\mathrm{coc}}$ to the category $\mathsf{HBr}_{\mathrm{coc}}(\Mm)$ of cocommutative Hopf braces in a braided monoidal category, see e.g.\ \cite[Definition 8]{FGRR} for the definition, under the assumption that $(\Mm,\ot,\mathbf{1},\sigma)$ is an abelian symmetric monoidal category that satisfies the ``faithful coflatness condition'' and the ``faithful flatness condition''. This would give an answer to the question opened in \cite[Remark 5.8]{GranSciandra}. \medskip

In the next subsection, we characterize abelian objects in the semi-abelian category $\Hopf$. Therefore, from now on, we assume that $(\Mm,\ot,\mathbf{1},\sigma)$ is such that $\Hopf$ has binary coproducts. 

\subsection{Abelian objects}

Since the category $\mathsf{Ab}(\Cc)$ of abelian objects, i.e.\ internal abelian groups in $\Cc$, in a semi-abelian category $\Cc$ is abelian \cite[Theorem 3.2]{Barr}, the category $\mathsf{Ab}(\Hopf)$ is abelian. We now provide an explicit description of the latter category. Recall that an object $C$ in a semi-abelian category $\Cc$ is abelian if and only if the morphism $\langle\mathrm{Id}_{C},\mathrm{Id}_{C}\rangle:C\to C\times C$ is a kernel in $\Cc$, see \cite[Proposition 9]{Bourn-normal}. 

\begin{remark}\label{rmk:abelianobjects}
Given $H$ in $\Hopf$, we know that $H\times H=H\ot H$ and $\langle \mathrm{Id}_{H},\mathrm{Id}_{H}\rangle=\Delta_{H}$. Hence, $\mathsf{Ab}(\Hopf)$ is the full subcategory of $\Hopf$ whose objects $H$ are such that $\Delta_{H}$ is a kernel in $\Hopf$ or, equivalently by Theorem \ref{cor:normalobjects}, such that $\Delta_{H}$ is normal.
\end{remark}

In order to characterize abelian objects in $\Hopf$, we first prove the following results.

\begin{proposition}\label{lem:adjcommutative}
Let $A$ be an object in $\Hopf$. Then, $A$ is commutative, i.e.\ $m_{A}=m_{A}\sigma_{A,A}$, if and only if $\mathrm{ad}_{A}=\varepsilon_{A}\ot\mathrm{Id}_{A}$. As a consequence, if $A$ is commutative, any monomorphism $i:K\to A$ in $\Hopf$ is normal.
\end{proposition}

\begin{proof}
    Suppose $A$ is commutative, i.e.\ $m_{A}=m_{A}\sigma_{A,A}$. Then, we obtain
\[
\begin{split}
\mathrm{ad}_{A}&=m_{A}(\mathrm{Id}_{A}\ot m_{A})(\mathrm{Id}_{A}\ot\sigma_{A,A})((\mathrm{Id}_{A}\ot S_{A})\Delta_{A}\ot\mathrm{Id}_{A})=m_{A}(\mathrm{Id}_{A}\ot m_{A})((\mathrm{Id}_{A}\ot S_{A})\Delta_{A}\ot\mathrm{Id}_{A})\\&=m_{A}(m_{A}(\mathrm{Id}_{A}\ot S_{A})\Delta_{A}\ot\mathrm{Id}_{A})=m_{A}(u_{A}\varepsilon_{A}\ot\mathrm{Id}_{A})=\varepsilon_{A}\ot\mathrm{Id}_{A}.
\end{split}
\]
Now, suppose $\mathrm{ad}_{A}=\varepsilon_{A}\ot\mathrm{Id}_{A}$. Since
\[
\begin{split}
&m_{A}(\mathrm{ad}_{A}\ot\mathrm{Id}_{A})(\mathrm{Id}_{A}\ot\sigma_{A,A})(\Delta_{A}\ot\mathrm{Id}_{A})=\\&m_{A}(m_{A}\ot\mathrm{Id}_{A})(\mathrm{Id}_{A}\ot m_{A}\ot\mathrm{Id}_{A})(\mathrm{Id}_{A}\ot\sigma_{A,A}\ot\mathrm{Id}_{A})((\mathrm{Id}_{A}\ot S_{A})\Delta_{A}\ot\mathrm{Id}_{A\ot A})(\mathrm{Id}_{A}\ot\sigma_{A,A})(\Delta_{A}\ot\mathrm{Id}_{A})=\\&m_{A}(\mathrm{Id}_{A}\ot m_{A})(\mathrm{Id}_{A\ot A}\ot m_{A})(\mathrm{Id}_{A}\ot\sigma_{A,A}\ot\mathrm{Id}_{A})(\mathrm{Id}_{A\ot A}\ot\sigma_{A,A})((\mathrm{Id}_{A}\ot S_{A})\Delta_{A}\ot\mathrm{Id}_{A\ot A})(\Delta_{A}\ot\mathrm{Id}_{A})=\\&m_{A}(\mathrm{Id}_{A}\ot m_{A})(\mathrm{Id}_{A\ot A}\ot m_{A})(\mathrm{Id}_{A}\ot\sigma_{A\ot A,A})(\mathrm{Id}_{A}\ot( S_{A}\ot\mathrm{Id}_{A})\Delta_{A}\ot\mathrm{Id}_{A})(\Delta_{A}\ot\mathrm{Id}_{A})=\\&m_{A}(\mathrm{Id}_{A}\ot m_{A})(\mathrm{Id}_{A\ot A}\ot m_{A})(\mathrm{Id}_{A\ot A}\ot( S_{A}\ot\mathrm{Id}_{A})\Delta_{A})(\mathrm{Id}_{A}\ot\sigma_{A,A})(\Delta_{A}\ot\mathrm{Id}_{A})=\\&m_{A}(\mathrm{Id}_{A}\ot m_{A})(\mathrm{Id}_{A\ot A}\ot u_{A}\varepsilon_{A})(\mathrm{Id}_{A}\ot\sigma_{A,A})(\Delta_{A}\ot\mathrm{Id}_{A})=\\&m_{A}(\mathrm{Id}_{A}\ot\varepsilon_{A}\ot\mathrm{Id}_{A})(\Delta_{A}\ot\mathrm{Id}_{A})=m_{A}
\end{split}
\]
and
\[
\begin{split}
\mathrm{ad}_{A}&=\varepsilon_{A}\ot\mathrm{Id}_{A}=m_{A}(\varepsilon_{A}\ot\mathrm{Id}_{A}\ot u_{A})=m_{A}(\mathrm{Id}_{A}\ot\varepsilon_{A}\ot\mathrm{Id}_{A})(\sigma_{A,A}\ot\mathrm{Id}_{A})(\mathrm{Id}_{A}\ot\mathrm{Id}_{A}\ot u_{A})\\&=m_{A}(\mathrm{Id}_{A}\ot\varepsilon_{A}\ot\mathrm{Id}_{A})(\mathrm{Id}_{A}\ot\mathrm{Id}_{A}\ot u_{A})\sigma_{A,A}=m_{A}(\mathrm{Id}_{A}\ot u_{A}\varepsilon_{A})\sigma_{A,A}\\&=m_{A}(\mathrm{Id}_{A}\ot m_{A}(\mathrm{Id}_{A}\ot S_{A})\Delta_{A})\sigma_{A,A}
\end{split}
\]
we get
\[
\begin{split}
m_{A}&=m_{A}(\mathrm{ad}_{A}\ot\mathrm{Id}_{A})(\mathrm{Id}_{A}\ot\sigma_{A,A})(\Delta_{A}\ot\mathrm{Id}_{A})\\&=m_{A}(m_{A}\ot\mathrm{Id}_{A})(\mathrm{Id}_{A}\ot m_{A}(\mathrm{Id}\ot S_{A})\Delta_{A}\ot\mathrm{Id}_{A})(\sigma_{A,A}\ot\mathrm{Id}_{A})(\mathrm{Id}_{A}\ot\sigma_{A,A})(\Delta_{A}\ot\mathrm{Id}_{A})\\&=m_{A}(m_{A}\ot\mathrm{Id}_{A})(\mathrm{Id}_{A}\ot m_{A}(\mathrm{Id}\ot S_{A})\Delta_{A}\ot\mathrm{Id}_{A})\sigma_{A\ot A,A}(\Delta_{A}\ot\mathrm{Id}_{A})\\&=m_{A}(\mathrm{Id}_{A}\ot m_{A})(\mathrm{Id}_{A}\ot m_{A}(\mathrm{Id}\ot S_{A})\Delta_{A}\ot\mathrm{Id}_{A})(\mathrm{Id}_{A}\ot\Delta_{A})\sigma_{A,A}\\&=m_{A}(\mathrm{Id}_{A}\ot m_{A})(\mathrm{Id}_{A\ot A}\ot m_{A}(S_{A}\ot\mathrm{Id}_{A})\Delta_{A})(\mathrm{Id}_{A}\ot\Delta_{A})\sigma_{A,A}\\&=m_{A}(\mathrm{Id}_{A}\ot m_{A})(\mathrm{Id}_{A\ot A}\ot u_{A}\varepsilon_{A})(\mathrm{Id}_{A}\ot\Delta_{A})\sigma_{A,A}\\&=m_{A}\sigma_{A,A}.
\end{split}
\]
Hence $m_{A}=m_{A}\sigma_{A,A}$, i.e.\ $A$ is commutative.

Suppose that $A$ is commutative, so $\mathrm{ad}_{A}=\varepsilon_{A}\ot\mathrm{Id}_{A}$. Given a monomorphism $i:K\to A$ in $\Hopf$, we have $\mathrm{ad}_{A}(\mathrm{Id}_{A}\ot i)=(\varepsilon_{A}\ot\mathrm{Id}_{A})(\mathrm{Id}_{A}\ot i)=i(\varepsilon_{A}\ot\mathrm{Id}_{K})$. Thus, by defining $\psi:=\varepsilon_{A}\ot\mathrm{Id}_{K}$, the diagram \eqref{diagram:normal} commutes. Then $i$ is normal.
\end{proof}

The next technical lemma will allow us to characterize abelian objects in Proposition \ref{prop:abelianobjects}.

\begin{lemma}
    Let $H$ be in $\Hopf$. The following equality holds:
\begin{equation}\label{eq:adjointtensorproduct}
    \mathrm{ad}_{H\ot H}(\mathrm{Id}_{H}\ot u_{H}\ot\mathrm{Id}_{H\ot H})=\mathrm{ad}_{H}\ot\mathrm{Id}_{H},
\end{equation}
where $H\ot H$ is in $\Hopf$ since $(\Mm,\ot,\mathbf{1},\sigma)$ is a symmetric monoidal category.
\end{lemma}

\begin{proof}
    We compute
\[
\begin{split}
&\mathrm{ad}_{H\ot H}(\mathrm{Id}_{H}\ot u_{H}\ot\mathrm{Id}_{H\ot H})\\&=m_{H\ot H}(m_{H\ot H}\ot S_{H\ot H})(\mathrm{Id}_{H\ot H}\ot\sigma_{H\ot H,H\ot H})(\Delta_{H\ot H}\ot\mathrm{Id}_{H\ot H})(\mathrm{Id}_{H}\ot u_{H}\ot\mathrm{Id}_{H\ot H})\\&=m_{H\ot H}(m_{H\ot H}\ot S_{H\ot H})(\mathrm{Id}_{H\ot H}\ot\sigma_{H\ot H,H\ot H})(\mathrm{Id}_{H}\ot\sigma_{H,H}\ot\mathrm{Id}_{H\ot H\ot H})\\&\hspace{0.5cm}(\Delta_{H}\ot\Delta_{H}\ot\mathrm{Id}_{H\ot H})(\mathrm{Id}_{H}\ot u_{H}\ot\mathrm{Id}_{H\ot H})\\&=m_{H\ot H}(m_{H\ot H}\ot S_{H\ot H})(\mathrm{Id}_{H\ot H}\ot\sigma_{H\ot H,H\ot H})(\mathrm{Id}_{H}\ot\sigma_{H,H}\ot\mathrm{Id}_{H\ot H\ot H})(\Delta_{H}\ot u_{H}\ot u_{H}\ot\mathrm{Id}_{H\ot H})\\&=m_{H\ot H}(m_{H\ot H}\ot S_{H\ot H})(\mathrm{Id}_{H\ot H}\ot(\sigma_{H,H\ot H}\ot\mathrm{Id}_{H})(\mathrm{Id}_{H}\ot\sigma_{H,H\ot H}))\\&\hspace{0.5cm}(\mathrm{Id}_{H}\ot u_{H}\ot\mathrm{Id}_{H}\ot u_{H}\ot\mathrm{Id}_{H\ot H})(\Delta_{H}\ot \mathrm{Id}_{H\ot H})\\&=m_{H\ot H}(m_{H\ot H}\ot S_{H\ot H})(\mathrm{Id}_{H\ot H}\ot\sigma_{H,H\ot H}\ot\mathrm{Id}_{H})(\mathrm{Id}_{H}\ot u_{H}\ot\mathrm{Id}_{H\ot H\ot H}\ot u_{H})(\Delta_{H}\ot \mathrm{Id}_{H\ot H}).
\end{split}
\]
On the other hand, recalling that $\sigma$ is a symmetry, we have
\[
\begin{split}
&\mathrm{ad}_{H}\ot\mathrm{Id}_{H}\\&=(m_{H}\ot\mathrm{Id}_{H})(m_{H}\ot\mathrm{Id}_{H\ot H})(\mathrm{Id}_{H}\ot\sigma_{H,H}\ot\mathrm{Id}_{H})((\mathrm{Id}_{H}\ot S_{H})\Delta_{H}\ot\mathrm{Id}_{H\ot H})\\&=(m_{H}\ot\mathrm{Id}_{H})(m_{H}\ot\mathrm{Id}_{H\ot H})(\mathrm{Id}_{H\ot H}\ot S_{H}\ot\mathrm{Id}_{H})(\mathrm{Id}_{H}\ot\sigma_{H,H}\ot\mathrm{Id}_{H})(\Delta_{H}\ot\mathrm{Id}_{H\ot H})\\&=(m_{H}\ot m_{H})(m_{H}\ot S_{H}\ot\mathrm{Id}_{H}\ot u_{H})(\mathrm{Id}_{H\ot H}\ot\sigma_{H,H})(\mathrm{Id}_{H\ot H}\ot\sigma_{H,H})\\&\hspace{0.5cm}(\mathrm{Id}_{H}\ot\sigma_{H,H}\ot\mathrm{Id}_{H})(\Delta_{H}\ot\mathrm{Id}_{H\ot H})\\&=(m_{H}\ot m_{H})(m_{H}\ot S_{H}\ot\mathrm{Id}_{H}\ot u_{H})(\mathrm{Id}_{H\ot H}\ot\sigma_{H,H})(\mathrm{Id}_{H}\ot\sigma_{H,H\ot H})(\Delta_{H}\ot\mathrm{Id}_{H\ot H})\\&=(m_{H}\ot m_{H})(\mathrm{Id}_{H}\ot\sigma_{H,H}\ot\mathrm{Id}_{H})(m_{H}\ot\mathrm{Id}_{H}\ot S_{H}\ot u_{H})(\mathrm{Id}_{H}\ot\sigma_{H,H\ot H})(\Delta_{H}\ot\mathrm{Id}_{H\ot H})\\&=m_{H\ot H}(m_{H}\ot\mathrm{Id}_{H}\ot S_{H}\ot S_{H}u_{H})(\mathrm{Id}_{H}\ot\sigma_{H,H\ot H})(\Delta_{H}\ot\mathrm{Id}_{H\ot H})\\&=m_{H\ot H}(m_{H}\ot m_{H}(u_{H}\ot\mathrm{Id}_{H})\ot S_{H}\ot S_{H}u_{H})(\mathrm{Id}_{H}\ot\sigma_{H,H\ot H})(\Delta_{H}\ot\mathrm{Id}_{H\ot H})\\&=m_{H\ot H}(m_{H}\ot m_{H}\ot S_{H}\ot S_{H})(\mathrm{Id}_{H}\ot\sigma_{H,H}\ot\mathrm{Id}_{H\ot H\ot H})(\mathrm{Id}_{H}\ot u_{H}\ot\mathrm{Id}_{H\ot H\ot H}\ot u_{H})\\&\hspace{0.5cm}(\mathrm{Id}_{H}\ot\sigma_{H,H\ot H})(\Delta_{H}\ot\mathrm{Id}_{H\ot H})\\&=m_{H\ot H}(m_{H\ot H}\ot S_{H\ot H})(\mathrm{Id}_{H\ot H}\ot\sigma_{H,H\ot H}\ot\mathrm{Id}_{H})(\mathrm{Id}_{H}\ot u_{H}\ot\mathrm{Id}_{H\ot H\ot H}\ot u_{H})(\Delta_{H}\ot\mathrm{Id}_{H\ot H}),
\end{split}
\]
so we get the thesis.
\end{proof}

We denote by $\mathsf{Hopf}_{\mathrm{coc,com}}(\Mm)$ the category of commutative and cocommutative Hopf monoids in $\Mm$. We finally obtain the following result.

\begin{proposition}\label{prop:abelianobjects}
    The category $\mathsf{Ab}(\Hopf)$ is exactly $\mathsf{Hopf}_{\mathrm{coc,com}}(\Mm)$.
\end{proposition}

\begin{proof}
By Remark \ref{rmk:abelianobjects}, we know that $\mathsf{Ab}(\Hopf)$ is the full subcategory of $\Hopf$ whose objects $H$ are such that $\Delta_{H}$ is a normal monomorphism.

If $H$ is in $\mathsf{Hopf}_{\mathrm{coc,com}}(\Mm)$, i.e.\ it is commutative, the object $H\ot H$ in $\Hopf$ is also commutative since commutative monoids in $\Mm$ form a monoidal category (since $(\Mm,\ot,\mathbf{1},\sigma)$ is a symmetric monoidal category). Therefore, by Proposition \ref{lem:adjcommutative}, we get that $\Delta_{H}:H\to H\ot H$ is normal. Then, $H$ is an object in $\mathsf{Ab}(\Hopf)$.

If $H$ is in $\mathsf{Ab}(\Hopf)$, i.e.\ $\Delta_{H}:H\to H\ot H$ is normal, there exists a morphism $\psi:H\ot H\ot H\to H$ in $\Mm$ such that $\Delta_{H}\psi=\mathrm{ad}_{H\ot H}(\mathrm{Id}_{H\ot H}\ot\Delta_{H})$. We have
\begin{equation}\label{eq:adDelta}
(\mathrm{ad}_{H}\ot\mathrm{Id}_{H})(\mathrm{Id}_{H}\ot\Delta_{H})\overset{\eqref{eq:adjointtensorproduct}}{=}\mathrm{ad}_{H\ot H}(\mathrm{Id}_{H\ot H}\ot\Delta_{H})(\mathrm{Id}_{H}\ot u_{H}\ot\mathrm{Id}_{H})=\Delta_{H}\psi(\mathrm{Id}_{H}\ot u_{H}\ot\mathrm{Id}_{H})
\end{equation}
and then, since $\mathrm{ad}_{H}$ is counitary, we get
\[
\begin{split}
\psi(\mathrm{Id}_{H}\ot u_{H}\ot\mathrm{Id}_{H})&=(\varepsilon_{H}\ot\mathrm{Id}_{H})\Delta_{H}\psi(\mathrm{Id}_{H}\ot u_{H}\ot\mathrm{Id}_{H})\overset{\eqref{eq:adDelta}}{=}(\varepsilon_{H}\ot\mathrm{Id}_{H})(\mathrm{ad}_{H}\ot\mathrm{Id}_{H})(\mathrm{Id}_{H}\ot\Delta_{H})\\&=(\varepsilon_{H}\ot\varepsilon_{H}\ot\mathrm{Id}_{H})(\mathrm{Id}_{H}\ot\Delta_{H})=\varepsilon_{H}\ot\mathrm{Id}_{H}.
\end{split}
\]
But we also have
\[
    \mathrm{ad}_{H}=(\mathrm{Id}_{H}\ot\varepsilon_{H})(\mathrm{ad}_{H}\ot\mathrm{Id}_{H})(\mathrm{Id}_{H}\ot\Delta_{H})\overset{\eqref{eq:adDelta}}{=}(\mathrm{Id}_{H}\ot\varepsilon_{H})\Delta_{H}\psi(\mathrm{Id}_{H}\ot u_{H}\ot\mathrm{Id}_{H})=\psi(\mathrm{Id}_{H}\ot u_{H}\ot\mathrm{Id}_{H}).
\]
Therefore, we obtain $\mathrm{ad}_{H}=\varepsilon_{H}\ot\mathrm{Id}_{H}$. Hence, by Proposition \ref{lem:adjcommutative}, $H$ is commutative. Consequently, it is an object in $\mathsf{Hopf}_{\mathrm{coc,com}}(\Mm)$.
\end{proof}

In the last subsection, we discuss an interesting application of the semi-abelianness of the category $\Hopf$.

\subsection{Action representability}
Using that $\Hopf$ is a semi-abelian category, we obtain an important feature of it, namely that it is action representable (in the sense of \cite{BJK2}). Let us first recall what this property means. \medskip

A semi-abelian category $\Cc$ is \textit{action representable} \cite{BJK2} if, for any object $X$ in $\Cc$, there exists a split extension with kernel $X$
\[\begin{tikzcd}
	X & \overline{X} & {[X]} 
	\arrow[from=1-1, to=1-2,"\mathsf{ker}(\alpha)"]
	\arrow[shift left, from=1-2, to=1-3,"\alpha"]
	\arrow[shift left, from=1-3, to=1-2,"\beta"]
\end{tikzcd}\]
such that, given any other split extension in $\Cc$ with kernel $X$
\[\begin{tikzcd}
	X & A & B
	\arrow[from=1-1, to=1-2,"\mathsf{ker}(p)"]
	\arrow[shift left, from=1-2, to=1-3,"p"]
	\arrow[shift left, from=1-3, to=1-2,"i"]
\end{tikzcd}\]
there is a unique (up to isomorphism) morphism $f:B\to[X]$ in $\Cc$ (and then a unique morphism $g:A\to\overline{X}$ in $\Cc$) such that the following diagram commutes.
\[\begin{tikzcd}
	X & A & B \\
	X & \overline{X} & {[X]}
	\arrow[from=1-1, to=1-2,"\mathsf{ker}(p)"]
	\arrow[from=1-1, to=2-1,"\mathrm{Id}_{X}"']
	\arrow[shift left, from=1-2, to=1-3,"p"]
	\arrow[shift left, from=1-3, to=1-2, "i"]
	\arrow[from=1-2, to=2-2,"g"]
	\arrow[from=1-3, to=2-3,"f"]
	\arrow[from=2-1, to=2-2,"\mathsf{ker}(\alpha)"']
	\arrow[shift left, from=2-2, to=2-3,"\alpha"]
	\arrow[shift left, from=2-3, to=2-2,"\beta"]
\end{tikzcd}\]

In order to obtain the action representability of $\Hopf$ we consider categories $(\Mm,\ot,\mathbf{1},\sigma)$ that are closed monoidal. We recall that these categories are always \textit{admissible} in the sense of \cite{Porst3}, see \cite[Remark 1, page 2]{Porst3}. We obtain the following result:

\begin{proposition}\label{prop:actionreprestable}
    The category $\Hopf$ is action representable, if $(\Mm,\ot,\mathbf{1},\sigma)$ is closed monoidal.
\end{proposition}

\begin{proof}
    By \cite[\S3.2 Proposition]{Porst3}, the category $\mathsf{Comon}_{\mathrm{coc}}(\Mm)$ is cartesian closed since $\Mm$ is monoidally closed. By \cite[Theorem 4.4]{BJK2} one knows that the category of internal groups in a cartesian closed category is always action representable, provided it is semi-abelian. Then, since $\Hopf=\mathsf{Grp}(\mathsf{Comon}_{\mathrm{coc}}(\Mm))$ is semi-abelian by Theorem \ref{thm:semiabelian}, we get the thesis.
\end{proof}

\begin{remark}
Since $\Hopf=\mathsf{Grp}(\mathsf{Comon}_{\mathrm{coc}}(\Mm))$ with $\mathsf{Comon}_{\mathrm{coc}}(\Mm)$ cartesian closed, we also obtain that $\Hopf$ is locally algebraically cartesian closed in the sense of \cite{Gray}, by using \cite[Proposition 5.3]{Gray}, and then also algebraically coherent in the sense of \cite{CGV}, by \cite[Theorem 4.5]{CGV}.
\end{remark}

Semi-abelian categories provide a good categorical framework to develop an approach to commutator theory and they present natural notions of semi-direct product \cite{BournJanelidze}, internal action \cite{BJK2} and crossed module \cite{Janelidze}. The study of these features for $\Hopf$ deserves to be undertaken in the future, generalizing the corresponding results achieved in \cite{GSV} and \cite{Sciandra2} for $\mathsf{Hopf}_{\mathrm{coc}}(\mathsf{Vec}_{\Bbbk})$ and $\mathsf{Hopf}_{\mathrm{coc}}(\mathsf{Vec}_{G})$, respectively.

\bigskip

\noindent\textbf{Acknowledgements.} 
The authors thank A. Ardizzoni for a nice discussion on the topic. This paper was written while the authors were members of the ``National Group for Algebraic and Geometric Structures and their Applications'' (GNSAGA-INdAM). A. Sciandra was supported by a postdoctoral fellowship at the Université libre de Bruxelles within the framework of the PDR project “Reconstruction of modules and algebraic objects from closed and monoidal structures on their representation categories" funded by the FNRS under the grant number T.0318.25F (PI Joost Vercruysse). Z. Zuo sincerely acknowledges the support provided by CSC (China Scholarship Council) through a PhD
student fellowship (No. 202406190047). This work was partially supported by the project funded by the European Union - NextGenerationEU under NRRP, Mission 4 Component 2 CUP D53D23005960006 - Call PRIN 2022 No.\ 104 of February 2, 2022 of Italian Ministry of University and Research; Project 2022S97PMY ``Structures for Quivers, Algebras and Representations'' (SQUARE).


\begin{thebibliography}{}

\bibitem{AsWe}
P. Aschieri, T. Weber, 
\emph{Metric compatibility and Levi-Civita Connections on Quantum Groups}. J. Algebra. 661 (2025) 479-544.

\bibitem{A97} M. Aguiar, \emph{Internal categories and quantum groups}. 
Thesis (Ph.D.)-Cornell University, ProQuest LLC, Ann Arbor, MI, 1997.

\bibitem{A08} 
A. Ardizzoni, \emph{Wedge products and cotensor coalgebras in monoidal categories}. Algebr. Represent. Theory 11 (2008), no. 5, 461--496.

\bibitem{AMS}
A. Ardizzoni, C. Menini, D. Ştefan, 
\emph{Hochschild cohomology and ``smoothness'' in monoidal categories}. J. Pure Appl. Algebra 208 (2007), no. 1, 297–330.



\bibitem{AGV}
A.L. Agore, A.S. Gordienko, J. Vercruysse, 
\emph{Lifting of locally initial objects and universal (co)acting Hopf algebras}. Adv. Math. 479 (2025), part B, Paper No. 110442, 68 pp.

\bibitem{AR}
J. Adámek, J. Rosický, 
\emph{Locally Presentable and Accessible Categories}, London Mathematical Society Lecture Note Series, vol. 189, Cambridge University Press, Cambridge, 1994.

\bibitem{AM}
M. Aguiar, S. Mahajan, 
\emph{Monoidal Functors, Species and Hopf Algebras}. With Forewords by Kenneth Brown and Stephen Chase and André Joyal, CRM Monograph Series, vol. 29, American Mathematical Society, Providence, RI, 2010.



\bibitem{AnGaVe}
I. Angiono, C. Galindo, L. Vendramin, 
\emph{Hopf braces and Yang–Baxter operators}, Proc. Am. Math.
Soc. 145 (5) (2017) 1981--1995.


\bibitem{Barr}
M. Barr, 
\emph{Exact Categories}, Springer Lecture Notes in Math., vol. 236, 1971, pp. 1--120.



\bibitem{Bespalov}
Yu. N. Bespalov, 
\emph{Crossed modules and quantum groups in braided categories}, Appl. Categ. Structures 5 (1997), no. 2, 155–204.

\bibitem{BeDr}
Yu. N. Bespalov, B. Drabant, 
\emph{Hopf (bi-)modules and crossed modules in braided monoidal categories}, J. Pure Appl. Algebra 123 (1998), no. 1-3, 105–129,

\bibitem{Borceux}
F. Borceux, 
\emph{Handbook of Categorical Algebra. 1. Basic Category Theory}, Encyclopedia of Mathematics and Its Applications, vol. 50, Cambridge University Press, Cambridge, 1994.

\bibitem{BB}
F. Borceux, D. Bourn, 
\emph{Mal’cev, Protomodular, Homological and Semi-Abelian Categories}.
Mathematics and Its Applications, vol. 566, Kluwer Academic Publishers, Dordrecht, 2004.

\bibitem{Bournproto}
D. Bourn, 
\emph{Normalization equivalence, kernel equivalence and affine categories}. In Category theory (Como, 1990), vol. 1488 of Lecture Notes in Math. Springer, Berlin, 1991, pp. 43–62.

\bibitem{Bourn-normal}
D. Bourn, 
\emph{Normal subobjects and abelian objects in protomodular categories}, J. Algebra 228 (2000)
143--164.

\bibitem{Bourn-book}
D. Bourn, 
\emph{From groups to categorial algebra}. Compact Textb. Math.
Birkhäuser/Springer, Cham, 2017, xii+106 pp.

\bibitem{BG}
D. Bourn, M. Gran,
\emph{Regular, protomodular, and abelian categories}. Encyclopedia Math. Appl., 97
Cambridge University Press, Cambridge, 2004, 165–211.

\bibitem{BournJanelidze}
D. Bourn, G. Janelidze, \emph{Protomodularity, descent and semi-direct product}. Theory Appl. Categories 4 (1998)
37–46.

\bibitem{BJK2}
F. Borceux, G. Janelidze, G. M. Kelly, 
\emph{Internal object actions}. Comment. Math. Univ. Carol. 46 (2) (2005) 235–255.




\bibitem{CDR}
S. Caenepeel, S. Dascalescu, S. Raianu, 
\emph{Cosemisimple Hopf algebras coacting on coalgebras}, Comm. Algebra 24 (1996), 1649–1677.

\bibitem{CGV}
A.S. Cigoli, J.R.A. Gray, T. Van der Linden, 
\emph{Algebraically coherent categories}, Theory Appl. Categ. 30 (54) (2015) 1864–1905.


\bibitem{GKV}
M. Gran, G. Kadjo, J. Vercruysse, \emph{A torsion theory in the category of cocommutative Hopf algebras}. Appl. Categ. Struct.
24 (2016) 269–282.

\bibitem{EGSV}
L. El Kaoutit, A. Ghobadi, P. Saracco, J. Vercruysse,
\emph{Correspondence theorems for Hopf algebroids with applications to affine groupoids}. Canad. J. Math. 76 (2024), no. 3, 830–880.

\bibitem{EGSV1}
L. El Kaoutit, A. Ghobadi, P. Saracco, J. Vercruysse,
\emph{Addenda to ``Correspondence theorems for Hopf algebroids with applications to affine groupoids''}. Canad. J. Math. 77 (2025), no. 1, 347–350.

\bibitem{FGRR}
J.M. Fernández Vilaboa, R. González Rodríguez, B. Ramos Pérez, A.B. Rodríguez Raposo, 
\emph{Modules
over invertible 1-cocycles}, Turk. J. Math. 48 (2) (2024) 248--266.

\bibitem{GranSciandra}
M. Gran, A. Sciandra, 
\emph{Hopf braces and semi-abelian categories}. J. Algebra 690 (2026), 266-303.  

\bibitem{GSV}
M. Gran, F. Sterck, J. Vercruysse, 
\emph{A semi-abelian extension of a theorem by Takeuchi}, J. Pure Appl. Algebra 223 (10) (2019) 4171–4190.

\bibitem{Gray}
J.R.A. Gray, 
\emph{Algebraic exponentiation in general categories}, Appl. Categ. Struct. 20 (2012) 543–567.

\bibitem{HS}
I. Heckenberger, H.-J. Schneider, 
\emph{Hopf algebras and root systems}. Math. Surveys Monogr., 247
American Mathematical Society, Providence, RI, 2020, xix+582 pp.

\bibitem{HLFV}
M. Hyland, I. López Franco, C. Vasilakopoulou, 
\emph{Hopf measuring comonoids and enrichment}. Proc. Lond. Math. Soc. (3) 115 (2017), no. 5, 1118–1148.

\bibitem{Hovey}
M. Hovey, 
\emph{Homotopy theory of comodules over a Hopf algebroid, in Homotopy theory: relations with algebraic geometry, group cohomology, and algebraic K-theory} (Evanston, IL, 2002), pp. 261-304, Contemp. Math., Vol. 346, Amer. Math. Soc., Providence, RI, 2004.

\bibitem{Janelidze}
G. Janelidze, 
\emph{Internal crossed modules}. Georgian Math. J. 10 (1) (2003) 99–114.

\bibitem{JMT}
G. Janelidze, L. Márki, W. Tholen, \emph{Semi-abelian categories}, in: Category Theory 1999 (Coimbra), J. Pure Appl. Algebra 168 (2–3) (2002) 367–386.

\bibitem{KM}
U. Krähmer, M. Mahaman,  
\emph{Clones from comonoids}. Revista de la Unión Matemática Argentina (2024).

\bibitem{MacLane-book}
S. Mac Lane, 
\emph{Categories for the working mathematician}.
Second edition
Grad. Texts in Math., 5
Springer-Verlag, New York, 1998. xii+314 pp.

\bibitem{Majid-book}
S. Majid,
\emph{Foundations of Quantum Group Theory}, Cambridge University Press, Cambridge, 1995.

\bibitem{Masuoka}
A. Masuoka, 
\emph{On Hopf algebras with cocommutative coradicals}. J. Algebra 144 (1991), no. 2, 451–466.

\bibitem{Masuoka2}
A. Masuoka, 
\emph{The fundamental correspondences in super aﬃne groups and super formal groups}, J. Pure Appl. Algebra
202 (1–3) (2005) 284–312.

\bibitem{Newman}
K. Newman, 
\emph{A correspondence between bi-ideals and sub-Hopf algebras in cocommutative Hopf algebras}, J. Algebra 36 (1) (1975) 1–15.

\bibitem{Porst2}
H-E. Porst, 
\emph{The formal theory of Hopf algebras part I: Hopf monoids in a monoidal category}, Quaest. Math. 38 (5) (2015) 631–682.

\bibitem{Porst}
H-E. Porst, 
\emph{Universal constructions of Hopf algebras}, J. Pure Appl. Algebra 212 (11) (2008) 2547–2554.

\bibitem{Porst3}
H.-E. Porst, 
\emph{On categories of monoids, comonoids, and bimonoids}, Quaest. Math. 31 (2008) 127–139.

\bibitem{Porstequalizer}
H.-E. Porst,  
\emph{Colimits of monoids}. Theory Appl. Categ. 34 (2019), 456–467.


\bibitem{Saracco}
P. Saracco,
\emph{A remark on the Galois-type correspondence between ideal coideals and comodule subrings of a Hopf algebroid}. Bull. Belg. Math. Soc. Simon Stevin 30 (2023), no. 5, 668–682.

\bibitem{AS}
A. Sciandra,  
\emph{Semi-abelian condition for color Hopf algebras}. J. Pure Appl. Algebra 228 (2024), no. 9, Paper No. 107677, 34 pp.

\bibitem{Sciandra2}
A. Sciandra, 
\emph{Commutators and crossed modules of color Hopf algebras} (2023), arXiv:2312.00156.

\bibitem{Sterck}
F. Sterck, 
\emph{S-protomodularity of the category of cocommutative bialgebras}. J. Algebra Appl. 22 (2023), no. 12, Paper No. 2350252, 27 pp.


\bibitem{Takeuchi}
M. Takeuchi, 
\emph{A correspondence between Hopf ideals and sub-Hopf algebras}, Manuscr. Math. 7 (1972) 252–270.

\bibitem{Takeuchi2}
M. Takeuchi, 
\emph{Formal schemes over fields}, Comm. Algebra 5 (1977) 1483–1528.

\bibitem{Vercruysse}
J. Vercruysse, 
\emph{Hopf Algebras—Variant Notions and Reconstruction Theorems}. In Quantum Physics and
Linguistics: A Compositional, Diagrammatic Discourse. Oxford University Press, 02 2013.

\bibitem{Yana1}
H. Yanagihara, 
\emph{On isomorphism theorems of formal groups}. J. Algebra 55 (1978) 341–347.

\bibitem{Yana2}
H. Yanagihara, 
\emph{On group theoretic properties of cocommutative Hopf algebras}. Hiroshima Math. J. 9 (1979) 179–200.

    
\end{thebibliography}
\end{document}